\documentclass[onefignum,onetabnum]{siamonline220329}

\usepackage[utf8]{inputenc}

\usepackage{amsmath, amsfonts, amssymb, amscd}

\usepackage{mathtools}

\usepackage{bm}

\usepackage{dsfont}
\usepackage[mathscr]{eucal}
\usepackage{xfrac}

\usepackage{xcolor}

\usepackage{url}

\expandafter\def\csname ver@etex.sty\endcsname{3000/12/31}

\usepackage{autonum}

\usepackage[normalem]{ulem}

\newsiamremark{assumption}{Assumption}
\newsiamthm{fact}{Fact}
\newsiamremark{remark}{Remark}
\newsiamthm{inftheorem}{Theorem}
\newsiamthm{conjecture}{Conjecture}

\usepackage{footmisc}
\usepackage{booktabs}

\ifpdf
  \DeclareGraphicsExtensions{.eps,.pdf,.png,.jpg}
\else
  \DeclareGraphicsExtensions{.eps}
\fi

\usepackage{enumitem}
\setlist[enumerate]{leftmargin=.5in}
\setlist[itemize]{leftmargin=.5in}

\headers{Asymptotics of the Sketched Pseudoinverse}{D. LeJeune, P. Patil, H. Javadi, R. G. Baraniuk, and R. J. Tibshirani}

\title{Asymptotics of the Sketched Pseudoinverse%
}

\author{Daniel LeJeune{\thanks{Equal contribution.}}\phantom{a}\thanks{Department of Statistics, Stanford University, Stanford, CA 94305 (\email{daniel@dlej.net}).}
\and Pratik Patil{\footnotemark[1]}\phantom{a}\thanks{Department of Statistics, University of California, Berkeley, CA 94720 (\email{pratikpatil@berkeley.edu}, \mbox{\email{ryantibs@berkeley.edu}}).}
\and Hamid Javadi{\thanks{Google, Mountain View, CA 94043 (\email{hamid71reza@gmail.com}).}} ,\\
Richard G.\ Baraniuk\thanks{Department of Electrical and Computer Engineering, Rice University, Houston, TX 77005 (\email{richb@rice.edu}).}
\and Ryan J.\ Tibshirani\footnotemark[3]
}

\makeatletter
\newcommand*{\addFileDependency}[1]{%
  \typeout{(#1)}
  \@addtofilelist{#1}
  \IfFileExists{#1}{}{\typeout{No file #1.}}
}
\makeatother

\usepackage{bookmark}

\hypersetup{
    pdfencoding=auto,
    unicode=true,
    pdftoolbar=true,
    pdfmenubar=true
}

\usepackage{crossreftools}
\usepackage{preamble/defs-daniel} %
\newcommand{\EE}{\mathbb{E}}

\newcommand{\CC}{\mathbb{C}}
\newcommand{\RR}{\mathbb{R}}

\newcommand{\bA}{{\mathbf{A}}}
\newcommand{\bB}{{\mathbf{B}}}
\newcommand{\bC}{{\mathbf{C}}}
\newcommand{\bD}{{\mathbf{D}}}

\newcommand{\bI}{{\mathbf{I}}}

\newcommand{\bS}{{\mathbf{S}}}

\newcommand{\bX}{{\mathbf{X}}}

\newcommand{\bZ}{{\mathbf{Z}}}

\newcommand{\bx}{{\mathbf{x}}}

\newcommand{\bTheta}{{\bm{\Theta}}}

\newcommand{\bSigma}{{\bm{\Sigma}}}

\newcommand{\tv}{\widetilde{v}}

\newcommand{\op}{\rm op}

\newcommand{\asto}{\xrightarrow{\text{a.s.}}}

\def\1{\mathds{1}}

\newcommand{\asympequi}{\simeq}

\newcommand{\error}{\mathcal{E}}

\renewcommand{\Re}{\operatorname{Re}}
\renewcommand{\Im}{\operatorname{Im}}

\newcommand{\lambdaminnz}{\lambda_{\min}^{+}}

\ifpdf
\hypersetup{
  pdftitle={Asymptotics of the Sketched Pseudoinverse},
  pdfauthor={D. LeJeune, P. Patil, H. Javadi, R. G. Baraniuk, and R. J. Tibshirani}
}
\fi

\begin{document}

\maketitle

\begin{abstract}
We take a random matrix theory approach to random sketching and show an asymptotic first-order equivalence of the regularized sketched pseudoinverse of a positive semidefinite matrix to a certain evaluation of the resolvent of the same matrix. We focus on real-valued regularization and extend previous results on an asymptotic equivalence of random matrices to the real setting, providing a precise characterization of the equivalence even under negative regularization, including a precise characterization of the smallest nonzero eigenvalue of the sketched matrix, which may be of independent interest. We then further characterize the second-order equivalence of the sketched pseudoinverse. 
\editedinline{We also apply our results to the analysis of the sketch-and-project method and to sketched ridge regression.}
Lastly, we 
\editedinlinetwo{prove}
that these results generalize to asymptotically free sketching matrices, obtaining the resulting equivalence for orthogonal sketching matrices and comparing our results to several common sketches used in practice.
\end{abstract}



\section{Introduction}

In large-scale data processing systems, \emph{sketching} or \emph{random projections} play an essential role in making computation efficient and tractable. The basic idea is to replace high-dimensional data by relatively low-dimensional random linear projections of the data such that distances are preserved.
It is well-known that sketching can significantly reduce the size of the data without harming statistical performance, while providing a dramatic computational advantage \cite{pmlr-v80-aghazadeh18a,gower2015randomized,lacotte2019adaptive,wang_lee_mahdavi_kolar_srebro_2017}. 
For a summary of 
results on the applications of 
sketching in optimization and 
numerical linear algebra, we refer the reader to \cite{mahoney2011randomized,woodruff2014sketching}. 

In this work, we present a different kind of result than the usual sketching guarantee. Typically, sketching is guaranteed to preserve the output or statistical performance of computational methods with an error term that vanishes for sufficiently large sketch sizes \cite{avron2017faster, bakshi2020robust,clarkson2014sketching, ivkin2019communication, pilanci2016iterative, woodruff2021very}. In contrast, we characterize the precise way in which the solution to a computational problem changes when operating on a sketched version of data instead of the original data, showing that sketching induces a specific type of regularization.

Our primary contribution is a statement about the effect of sketching on the (regularized) pseudoinverse of a matrix. An informal statement of our result is as follows. Here the notation $\mA \asympequi \mB$ for two matrices $\mA$ and $\mB$ indicates an asymptotic first-order equivalence, which we define in \Cref{sec:preliminaries}, and $\lambdaminnz(\mA)$ is the smallest nonzero eigenvalue of a matrix $\mA$. We refer to $\mS \biginv{\mS^\ctransp \mA \mS + \lambda \mI_q} \mS^\ctransp$ as the sketched (regularized) pseudoinverse of $\mA$, because when $\mS$ has orthonormal columns, the pseudoinverse of $\mS \mS^\ctransp \mA \mS \mS^\ctransp$ is equal to $\mS (\mS^\ctransp \mA \mS)^{-1} \mS^\ctransp$. This expression is also related to the Nystr\"om approximation of $\mA$.

\begin{inftheorem}[\Cref{thm:sketched-pseudoinverse,thm:general-free-sketching}, informal]
    \label{thm:sketched-pseudoinverse-informal}
    Given a positive semidefinite matrix $\mA \in \complexset^{p \times p}$ and sketching matrix $\mS \in \complexset^{p \times q}$, for any $\lambda > -\lambdaminnz(\mS^\ctransp \mA \mS)$, there exists $\mu \in \reals$ such that
    \begin{align}
        \mS \biginv{\mS^\ctransp \mA \mS + \lambda \mI_q} \mS^\ctransp \asympequi \inv{\mA + \mu \mI_p}.
    \end{align}
\end{inftheorem}

The general implication of this result is that when we do computation using the sketched version of a matrix, there is a sense in which it is as if we were using additional ridge regularization. More precisely, when we solve (regularized) linear systems on a sketched version of the data and apply this solution to the sketched data, it is equivalent in a first-order sense to solving a regularized linear system in the original space. To see this, consider for example a least squares problem $\min_\vbeta \norm[2]{\vy - \mX \vbeta}^2$. The first-order optimality condition is $\mX^\ctransp \mX \vbeta = \mX^\ctransp \vy$, and if we replace $\mX$ by a sketch $\mX \mS$, we have the solution in the sketched domain $\widehat{\vbeta}_\mS = (\mS^\ctransp \mX^\ctransp \mX \mS)^{-1} \mS^\ctransp \mX^\ctransp \vy$. 
\begin{editedthree}
If we then measure this solution in some sketched direction $\mS^\ctransp \vu$ for some independent unit vector $\vu$, we obtain $\hat{\beta}_{\vu} = \vu^\ctransp \mS \widehat{\vbeta}_\mS = \vu^\ctransp \mS (\mS^\ctransp \mX^\ctransp \mX \mS)^{-1} \mS^\ctransp \mX^\ctransp \vy$. By our result, this is asymptotically equivalent to measuring $\hat{\beta}_{\vu} \asympequi \vu^\ctransp (\mX^\ctransp \mX + \mu \mI)^{-1} \mX^\ctransp \vy$---that is, as if we had solved the original least squares problem using some regularization $\mu$.%
\end{editedthree}
\subsection*{Summary of contributions}

Below we summarize the main contributions of the paper.

\begin{enumerate}
    \item \textbf{Real-valued equivalence.} We extend previous results from random matrix theory \cite{rubio_mestre_2011} \editedinlinetwo{for i.i.d.\ random matrices} to real-valued regularization, explicitly characterizing
    the behaviour of the associated fixed-point
    equation extended from the complex half-plane to the reals, allowing
    for consideration of negative regularization. This result includes what is to the best of our knowledge the first characterization of the limiting smallest nonzero eigenvalue of arbitrary \editedinlinetwo{Wishart type} sample covariance matrices, which may be of independent interest.
    \item \textbf{First-order equivalence.} Applying the real-valued equivalence, we obtain a first-order equivalence for the ridge-regularized \editedinlinetwo{i.i.d.}\ sketched pseudoinverse.
    \item \textbf{Second-order equivalence.} Using the calculus of asymptotic equivalents,
    we also obtain a second-order equivalence
    for the ridge-regularized \editedinlinetwo{i.i.d.}\ sketched pseudoinverse that captures a variance-like inflation due to the randomness of sketching.
    \item \textbf{Equivalence properties.} We provide a thorough investigation of the theoretical properties of the equivalence relationship, such as how the induced regularization depends on the original applied regularization, sketch size, and matrix rank.
    \item 
    \editedinline{\textbf{Applications.} We demonstrate how to apply our results by performing novel analysis of sketch-and-project \cite{gower2015randomized} and sketched ridge regression.}
    \item \editedinlinetwo{\textbf{Free sketching.} Finally, we extend the scope of our results
    for first-order equivalence of the sketched pseudoinverse beyond i.i.d.\ sketching to general asymptotically free sketching and specialize to orthogonal sketching matrices.}
\end{enumerate}

\subsection*{Related work}

The existence of an implicit regularization effect of sketching or random projections has been known for some time \cite{derezinski2021determinantal,leamer1976bayesian, rudi2015less, thanei_heinze_meinshausen_2017}. %
While prior works have demonstrated clear theoretical and empirical statistical advantages of sketching, our understanding of the precise nature of this implicit regularization has been largely limited to quantities such as error bounds. We provide, in contrast, a precise asymptotic characterization of the solution obtained by a sketching-based solver, not only enabling the understanding of the statistical performance of sketching-based methods, but also opening the door for exploiting the specific regularization induced by sketching in future algorithms.

Our results in this work provide a general extension 
of a few results appearing in recent works that have revealed explicit characterizations of the implicit regularization effects induced by random subsampling. To the best of our knowledge, the first such result was presented by \cite{lejeune_javadi_baraniuk_2020}, who showed that ensembles of (unregularized) ordinary least squares predictors on randomly subsampled observations and features converge in an $\ell_2$ metric to an (optimal) ridge regression solution in the proportional asymptotics regime. 
This result was limited in several aspects: 
a) it required a strong isotropic Gaussian data assumption;
b) it required the subsampled data to have more observations than features;
c) it considered only unregularized base learners in the ensemble;
d) it required an ensemble of infinite size to show 
the ridge regression equivalence;
e) it provided only a marginal guarantee of convergence 
over the data distribution 
rather than a single-instance convergence guarantee;
and 
f) it did not provide the relationship between the subsampling ratio and the amount of induced ridge regularization. In addition, the proof relied on rote computation of expectations of matrix quantities, providing limited insight into the underlying mathematical principles at work. The result we present in this work in \cref{thm:sketched-pseudoinverse} addresses all of these issues.

Around the same time, \cite{pmlr-v108-mutny20a} showed the remarkably simple result that the expected value of the pseudoinverse of any positive definite matrix sampled by a determinantal point process (DPP) is equal to a resolvent of the matrix. Similarly to the result by \cite{lejeune_javadi_baraniuk_2020}, this result demonstrated that when random subsampling is applied in techniques without any regularization, the resulting solution is as if a regularized technique was used on the original data. This result provided a simple form of the argument of the induced resolvent as a solution to a matrix trace equation, which is analogous the results we present in this work for sketching. 
The same authors later empirically demonstrated that the same effects occur when using i.i.d.\ Gaussian and Rademacher sketches \cite{derezinski_surrogate_2020} 
\editedinline{and obtained a first-order equivalent for certain sub-Gaussian sketched projection operators \cite{derezinski2020precise}} \editedinlinetwo{and first- and second-order moments for certain debiased sketches \cite{derezinski2021newtonless}.}
\editedinlinetwo{Our work generalizes these later developments and}
also 
differs from these works in that we provide a single-instance equivalent ridge regularization in the asymptotic regime, rather than an expectation over the random \editedinline{projections}.

Our results also echo the finite-sample results of \cite{pmlr-v134-derezinski21a}, who showed that the unregularized inverse of a particular sketched matrix form has a merely multiplicative bias for sketch size minimally larger than the rank of the original matrix. This is captured by \cref{cor:basic-ridge-asympequi-in-r} in our work when $z \to 0$, combined with \cref{rem:mu-prime-to-0} in which we observe that there is asymptotically no spectral distortion in the range of the original matrix for sketches larger than the rank.

\begin{edited}
Our work leverages techniques from random matrix theory \cite{rubio_mestre_2011}, and the techniques employed bear some resemblance to other recent work in high dimensional statistical analysis \cite{derezinski2020exact,dobriban_wager_2018,hastie_montanari_rosset_tibshirani_2022}. In particular, we leverage the calculus of deterministic equivalences as presented by \cite{dobriban_sheng_2021}. 
However, 
instead of characterizing only very specific quantities such as in-distribution generalization error, requiring tedious updates to the proof to adapt to other quantities of interest, we have isolated the expressions that will be needed to analyze \emph{any} quadratic functional of the sketched pseudoinverse.
In addition, instead of characterizing $\mA^{1/2} \mS \biginv{\mS^\ctransp \mA \mS + \lambda \mI_q} \mS^\ctransp \mA^{1/2}$ (as considered, e.g., by \cite{derezinski2020precise} for $\lambda = 0$) which is a simple reparameterization of $\biginv{\mA^{1/2} \mS^\ctransp \mS \mA^{1/2} + \lambda \mI_p}$ and therefore straightforwardly understood through equivalences for sample covariance matrices \cite{ledoit_peche_2011,rubio_mestre_2011}, we characterize the quantity $\mS \biginv{\mS^\ctransp \mA \mS + \lambda \mI_q} \mS^\ctransp$ which is essential for asymmetric applications such as ridge regression without data assumptions (see example in \Cref{sec:sketched-ridge}). 

Our application of our results to sketch-and-project \cite{gower2015randomized} improves upon recent work by \cite{derezinski2020precise} in that we are also able to calculate asymptotic computational complexity as a function of sketch size thanks to the uniformity of convergence over bounded sketching ratios and the ability to consider sparse sketches that can be applied in $O(q^2)$ time (see \cref{rem:sparse-sketching}).

Other works have considered other types of sketches that do not have the same random matrix properties as the matrices we consider in our \editedinlinetwo{main} results. In particular, fast sketching techniques such as CountSketch \cite{charikar2002frequent} and the subsampled randomized Hadamard transform (SRHT) \cite{tropp2011hadamard} are among the most popular random projections in practice, since they can be applied in only $O(p \log p)$ time rather than $O(pq)$ or $O(q^2)$ for i.i.d.\ sketches. Very little is known about the properties of these sketches under proportional asymptotics; we know only of \cite{lacotte2020optimal} who analyzed specific first and second moments in the isotropic case for the SRHT.
\begin{editedtwo}
Other prior work has shown universality of certain sketching inversion bias behavior under any rotationally invariant sketch~\cite{pmlr-v134-derezinski21a}.
We show that our results generalize to the broader class of ``free'' sketches in \Cref{thm:general-free-sketching} using free probability \cite{voiculescu1992free,mingo2017free} and specialize to an exact formula for orthogonal sketching in \Cref{cor:orthonormal-sketch}.
\end{editedtwo}
Then we empirically show that fast sketches commonly used in practice behave according to our
\editedinlinetwo{generalization. }

A few works have shown that under certain data geometry and noise, the optimal ridge regression parameter can be negative \cite{kobak_lomond_sanchez_2020, wu_xu_2020}. For this reason, we take special care to determine the limit of allowable negative regularization in sketched settings. Then in a ridge regression example in \Cref{sec:sketched-ridge}, we demonstrate how negative regularization can be optimal for standard noisy learning problems in undersampled distributed optimization settings.
\end{edited}

\subsection*{Organization}

The rest of the paper is structured as follows.
In \Cref{sec:preliminaries},
we start with some preliminaries
on the language of asymptotic equivalence
of random matrices that we will use to state our results.
In \Cref{sec:real-valued-equivalence},
we extend a previous result on asymptotic equivalence
for a ridge regularized resolvent
to include real-valued negative regularization
and provide a precise limiting lower limit
of the permitted negative regularization.
In \Cref{sec:main_results},
we provide our main results
about the first- and second-order
equivalence of the sketched pseudoinverse.
Then, in \Cref{sec:properties}, we explore properties of the equivalence and present illustrative examples.
\editedinline{In \Cref{sec:applications}, we perform novel analysis of two sketching based optimization methods.}
Finally, in \Cref{sec:discussion},
we conclude by giving
various extensions and providing
a
\editedinlinetwo{generalization of}
the asymptotic
behaviour of sketched pseudoinverse
for a broad family of sketching matrices
using the insights obtained from the proof
of our main result and experimentally compare sketches commonly used in practice to our theory.
Our code for generating all figures can be found at \url{https://github.com/dlej/sketched-pseudoinverse}.

\subsection*{Notation}
We denote the real line by $\RR$
and the complex plane by $\CC$.
For a complex number $z = x + iy$,
$\Re(z)$ denotes its real part $x$,
$\Im(z)$ denotes its imaginary part $y$,
and $\overline{z} = x - iy$ denotes its conjugate.
We use $\RR_{\ge 0}$ and $\RR_{> 0}$ to be denote 
the set of non-negative and positive real numbers, respectively;
similarly, $\RR_{\le 0}$ and $\RR_{< 0}$ respectively denote
the set of non-positive and negative real numbers.
We use $\CC^{+} = \{ z \in \CC : \Im(z) > 0 \}$ to denote 
the upper half of the complex plane
and $\CC^{-} = \{ z \in \CC : \Im(z) < 0 \}$ to denote 
the lower half of the complex plane. 

We denote vectors in lowercase bold letters (e.g., $\vy$)
and matrices in uppercase bold letters (e.g., $\mX$).
For a vector $\vy$, 
$\| \vy \|_2$ denotes its $\ell_2$ norm.
For a rectangular matrix $\bS \in \CC^{p \times q}$, 
$\bS^\ctransp \in \CC^{q \times p}$
denotes its conjugate or Hermitian transpose
(such that $[\mS^\ctransp]_{ij} = \overline{[\mS]_{ji}}$),
$\norm[\tr]{\mS}$ denotes its trace norm (or nuclear norm),
that is $\norm[\tr]{\mS} = \tr\bracket{(\mS^\ctransp \mS)^{1/2}}$,
and $\| \bS \|_{\op}$ denotes the operator norm
with respect to the $\ell_2$ vector norm
(which is also its spectral norm).
For a square matrix $\bA \in \CC^{p \times p}$,
$\tr[\bA]$ denotes its trace,
$\mathrm{rank}(\mA)$ denotes its rank,
$r(\mA) = \frac{1}{p} \mathrm{rank}(\mA)$
denotes its relative rank,
and $\bA^{-1} \in \CC^{p \times p}$ denotes its inverse,
if it is invertible.
\editedinline{For any matrix $\mA \in \CC^{p \times q}$,
$\mA^\dagger$ denotes the Moore--Penrose inverse.}
For a positive semidefinite matrix $\bA \in \CC^{p \times p}$,
$\bA^{1/2} \in \CC^{p \times p}$ denotes its positive semidefinite 
principal square root,
$\lambda_{\min}(\bA)$ its smallest eigenvalue, 
and $\lambdaminnz(\bA)$ its smallest positive eigenvalue.

A sequence $x_n$ converging to $x_{\infty}$ from the left 
or right is denoted by $x \nearrow x_{\infty}$ or $x \searrow x_{\infty}$, respectively.
We denote almost sure convergence by $\asto$.

\section{Preliminaries}
\label{sec:preliminaries}

We will use the language of asymptotic equivalence of sequences of random matrices
to state our main results.
In this section, we define the notion of asymptotic equivalence,
review some of the basic properties that such equivalence satisfies, 
and present an asymptotic equivalence for the ridge resolvent.
We then extend that result to handle real-valued resolvents, which
will form the building block for our subsequent results.

To begin, consider two sequences $\bA_n$ and $\bB_n$ of $p(n) \times q(n)$ matrices, where $p$ and $q$ are increasing in $n$.
We will say that $\bA_n$ and $\bB_n$ are asymptotically equivalent
if for any sequence of deterministic 
matrices $\bTheta_n$ with trace norm uniformly bounded in $n$,
we have $\tr[\bTheta_n (\bA_n - \bB_n)] \asto 0$
as $n \nearrow \infty$.
We write $\bA_n \asympequi \bB_n$ to denote this asymptotic equivalence.\footnote{\editedinline{When we use the same notation for a vector or scalar equivalence, it can be understood as applying this definition to a $p(n) \times 1$ or $1 \times 1$ matrix, respectively.}}
The notion of \emph{deterministic} equivalence,
where the right-hand sequence is a sequence of deterministic matrices, 
has been typically used in random matrix theory to obtain limiting behaviour
of functionals of random matrices;
for example,
see 
\cite{couillet_debbah_silverstein_2011,hachem_loubaton_najim_2006,serdobolskii_2000},
among others.
More recently, the notion of deterministic equivalence 
has been popularized and developed further 
in \cite{dobriban_wonder_2020,dobriban_sheng_2021}\footnote{Note that 
\cite{dobriban_wonder_2020,dobriban_sheng_2021}
use the notation $\bA_n \asymp \bB_n$
to denote deterministic equivalence of sequence $\bA_n$ to $\bB_n$.
We instead use the notation $\bA_n \asympequi \bB_n$
to emphasize that this equivalence is asymptotically exact, rather than up to constants.}.
We will use a slightly more general notion of asymptotic equivalence
in this paper, where both sequences of matrices may be random. 

The notion of asymptotic equivalence enjoys some properties
that we list next.
The majority of these are stated in the context of deterministic equivalence
in \cite{dobriban_wonder_2020,dobriban_sheng_2021}, but they also hold more generally for
asymptotic equivalence.
For the statements to follow,
let $\bA_n$, $\bB_n$, $\bC_n$, and $\bD_n$ be sequences of random or deterministic matrices
(of appropriate dimensions).
Then the following properties hold:
\begin{enumerate}
    \item \textbf{Equivalence.}
    The relation $\asympequi$ is an equivalence relation.
    \item \textbf{Sum.}
    If $\bA_n \asympequi \bB_n$ and $\bC_n \asympequi \bD_n$,
    then $\bA_n + \bC_n \asympequi \bB_n + \bD_n$.
    \item \textbf{Product.}
    If $\mA_n \asympequi \mB_n$, and $\mC_n$ is independent of $\mA_n$ and $\mB_n$ with operator norm bounded in $n$
    almost surely, 
    then $\bA_n \bC_n \asympequi \bB_n \bC_n$.
    \item \textbf{Trace.}
    If $\bA_n \asympequi \bB_n$ for square matrices $\bA_n$ and $\bB_n$ of dimension $p(n) \times p(n)$, 
    then \editedinline{$\tfrac{1}{p(n)} \tr[\bA_n] \asympequi \tfrac{1}{p(n)} \tr[\bB_n]$}.
    \item \textbf{Elements.} If $\mA_n \asympequi \mB_n$ for $\mA_n, \mB_n$ of dimension $p(n) \times q(n)$ and $i(n) \in \set{1, \ldots, p(n)}$ and $j(n) \in \set{1, \ldots, q(n)}$, then \editedinline{$[\mA_n]_{i(n), j(n)} \asympequi [\mB_n]_{i(n), j(n)}$}.
    \item \textbf{Differentiation.}
    Suppose $f(z, \bA_n) \asympequi g(z, \bB_n)$
    where the entries of $f$ and $g$ are analytic
    functions in $z \in D$ and $D$ is an open connected subset of $\CC$.
    Furthermore, suppose for any sequence $\bTheta_n$ of deterministic 
    matrices with trace norm uniformly bounded in $n$,
    we have that $|\tr[\bTheta_n (f(z, \bA_n) - g(z, \bB_n))]| \le M$
    for every $n$ and $z \in D$ for some constant $M < \infty$.
    Then we have that $f'(z, \bA_n) \asympequi g'(z, \bB_n)$
    for every $z \in D$,
    where the derivatives are taken entry-wise with respect to $z$.
\end{enumerate}

The almost sure convergence in the statements above
is with respect to the entire randomness in the random variables involved.
One can also consider
the notion of conditional asymptotic equivalence
wherein we condition on a sequence of random matrices.
More precisely,
suppose $\bA_n$, $\bB_n$ are sequence of random matrices
that may depend of another sequence of random matrices $\bZ_n$.
We call $\bA_n$ and $\bB_n$ to be asymptotically equivalent
conditioned on $\bZ_n$,
if for any sequence of deterministic matrices
$\bTheta_n$
with trace norm uniformly bounded in $n$,
we have $\lim_{n \nearrow \infty} \tr[\bTheta_n (\bA_n - \bB_n)] = 0$ almost surely conditioned on $\bZ_n$.
Similar properties to those listed above for unconditional asymptotic equivalence
also hold for conditional equivalence by considering all the statements
conditioned on the sequence $\bZ_n$.
In particular,
for the product rule,
we require that the sequence $\bC_n$ be \emph{conditionally}
independent of $\bA_n$ and $\bB_n$ given $\bZ_n$.
Finally,
for our asymptotic statements,
we will work with sequences of matrices,
indexed by either $n$ or $p$.
However, for notational brevity,
we will drop the index from now on whenever it is clear from the context.

Equipped with the notion of asymptotic equivalence,
below we state a result on the \mbox{asymptotic} deterministic
equivalence for ridge resolvents of \editedinlinetwo{Wishart type} matrices,
adapted from Theorem 1 of \cite{rubio_mestre_2011} and Theorem 3.1 of \cite{dobriban_sheng_2021},
that will form a base for our results.

\begin{lemma}[Basic asymptotic equivalent for ridge resolvents, complex-valued regularization]
    \label{lem:basic-ridge-asympequi}%
    Let $\mZ \in \complexset^{n \times p}$ be a random matrix consisting of i.i.d.\ random variables that have mean 0, variance 1, and finite 
    absolute moment of order
    $8 + \delta$ for some $\delta > 0$. Let $\mSigma \in \complexset^{p \times p}$ be a positive semidefinite matrix with operator norm uniformly bounded in $p$, and let $\mX = \mZ \mSigma^{1/2}$. 
    Then, for $z \in \complexset^+$,
    as $n, p \nearrow \infty$ such that
    $0 < \liminf \tfrac{p}{n} \le \limsup \tfrac{p}{n} < \infty$,
    we have
    \begin{equation}
        \label{eq:basic-ridge-asympequi-in-c}
        \big( \tfrac{1}{n} \bX^\ctransp \bX - z \bI_p \big)^{-1}
        \asympequi
        \inv{c(z) \bSigma -z \bI_p},
    \end{equation}
    where $c(z)$ is the unique solution in $\complexset^-$ to the fixed point equation
    \begin{equation}
        \label{eq:basic-ridge-fp-in-c}
        \frac{1}{c(z)} - 1
        = \tfrac{1}{n} \tr \bracket{\bSigma \inv{c(z) \bSigma - z \bI_p}}.
    \end{equation}
    Furthermore, 
    $\tfrac{1}{p} \tr\bracket{\bSigma (c(z) \bSigma - z \bI_p)^{-1}}$
    is a Stieltjes transform of a certain positive measure on $\RR_{\ge 0}$
    with total mass $\tfrac{1}{p} \tr[\bSigma]$.
\end{lemma}       
Strictly speaking,
the results in \cite{rubio_mestre_2011} and \cite{dobriban_sheng_2021}
require that the sequence $\bSigma$ be deterministic.
However, one can take $\bSigma$ to be a random sequence of matrices
that are independent of $\bZ$;
see, for example, \cite{ledoit_peche_2011}.
In this case,
the asymptotic equivalence is treated conditionally on $\bSigma$.

\section{Real-valued equivalence}
\label{sec:real-valued-equivalence}

For real-valued negative $z$, corresponding to positive ridge regularization, 
we remark that one can use \cref{lem:basic-ridge-asympequi}
to derive limits of linear and 
certain non-linear functionals 
(through the calculus rules of asymptotic equivalence)
of the ridge resolvent 
$(\tfrac{1}{n} \bX^\ctransp \bX - z \bI_p)^{-1}$
by considering $z \in \CC^{+}$ with $\Re(z) < 0$
and letting $\Im(z) \searrow 0$.
This follows because
a short calculation (see proof of \cref{cor:basic-ridge-asympequi-in-r}) 
shows that $\Im(c(z)) \nearrow 0$
as $\Im(z) \searrow 0$ for $z \in \CC^{+}$ with $\Re(z) < 0$.
Thus one can recover a real limit
from the right hand side of \cref{eq:basic-ridge-asympequi-in-c}
through a limiting argument.
Moreover,
it is easy to see that
the fixed-point equation \cref{eq:basic-ridge-fp-in-c}
has a unique (real) solution $c(z) > 0$
for $z \in \RR_{< 0}$. 

However, it has recently been pointed out that
under certain special data geometry,
negative regularization is often beneficial, in real data experiments \cite{kobak_lomond_sanchez_2020} as well as in theoretical formulations where it can achieve optimal squared prediction risk \cite{wu_xu_2020}.
One can still recover such a case
by considering $z \in \CC^{+}$
with $\Re(z) > 0$ over a valid range, and taking the limit as $\Im(z) \searrow 0$.
However, solving the fixed-point equation \cref{eq:basic-ridge-fp-in-c}
over reals directly in this case, which is the most efficient way to compute the solution numerically, poses certain subtleties
as we no longer can guarantee a unique real solution for $c(z)$. 

Our next theorem shows how to handle this case.
We will make use of this for our results on sketching
in \Cref{sec:main_results}, but we believe the result to be of independent interest and worth stating on its own. In addition to enabling the computation of the asymptotic equivalence for non-negative real-valued $z$, it also provides the asymptotic value of $\lambdaminnz(\tfrac{1}{n} \mX^\ctransp \mX)$ (given by $z_0$ in the theorem statement) for arbitrary $\mSigma$, 
which to our knowledge
is the first
\editedinlinethree{explicit} general characterization of the smallest nonzero eigenvalue of \editedinlinethree{Wishart-type matrices, although the underlying principles are known in random matrix theory~\cite{silverstein1995analysis} and have been applied algorithmically~\cite{dobriban2015efficient}. We note that our characterization enables an extremely efficient and simple approach for computing $z_0$ via direct root finding in $\zeta_0$.}
\editedinlinethree{
Furthermore, $z_0$ improves significantly on the na\"ive lower bounds commonly used in theoretical works~\cite{pmlr-v151-patil22a,wu_xu_2020}, as~seen~in~\Cref{fig:lambda-minnz-bound}.}

\begin{figure}[t]
    \centering
    \includegraphics[width=5.5in]{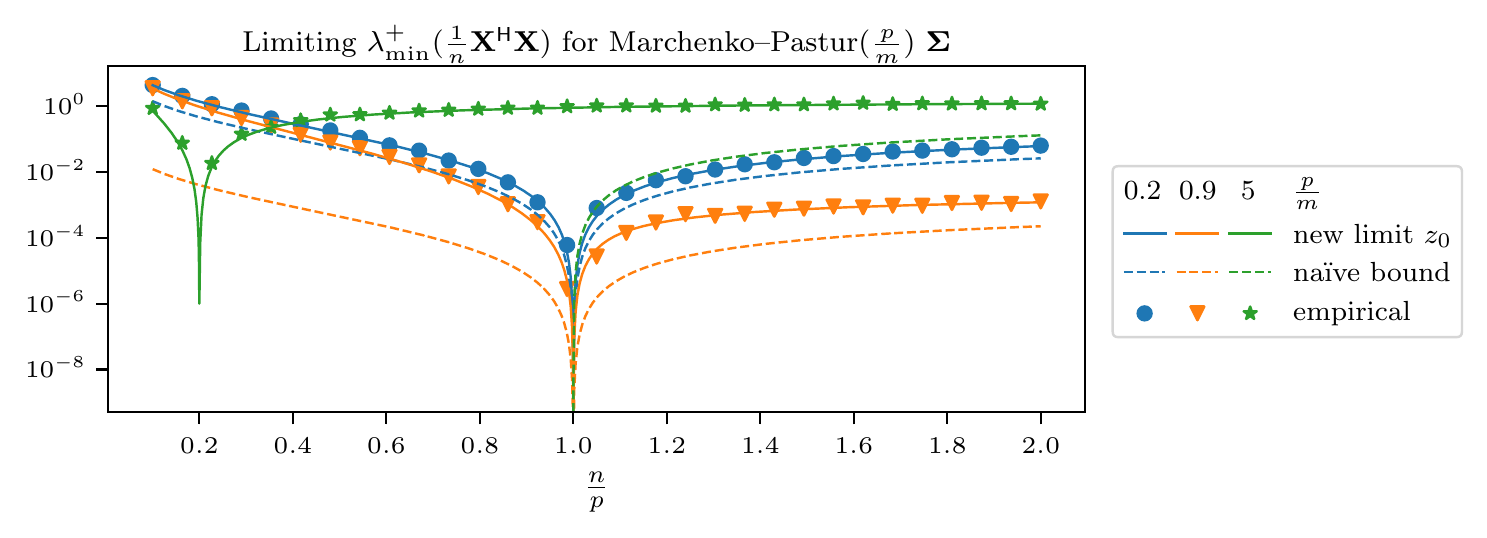}
    \caption{Plots showing how $z_0$ (solid) from \cref{eq:basic-ridge-asympequi-in-r:bounds} matches the empirical minimum nonzero eigenvalue (markers) of $\tfrac{1}{n} \mX^\transp \mX$ when $\mSigma = \tfrac{1}{m} \mY^\transp \mY$ for $\mY \in \reals^{m \times p}$ with i.i.d.\ $\normal(0, 1)$ elements, such that the limiting spectrum of $\mSigma$ follows the $\mathrm{Marchenko}\text{--}\mathrm{Pastur}(\tfrac{p}{m})$ distribution for $\tfrac{p}{m} \in \set{0.2, 0.9, 5}$. In contrast, the commonly used na\"ive bound (dashed)
    $\liminf \lambdaminnz(\tfrac{1}{n} \mX^\transp \mX) \geq 
    \bigparen{1 - \sqrt{\tfrac{p}{m}}}^2 
    \bigparen{1 - \sqrt{\tfrac{p}{n}}}^2 
    \ind\set{p < \max\set{m, n}}$,
    obtained by multiplying the minimum nonzero eigenvalues of $\tfrac{1}{n} \mZ^\transp \mZ$ and $\mSigma$ when at most one of them is singular,
    is quite loose outside of the $m \gg p$ and $n \gg p$ cases and fails to capture the correct behavior at all when both are singular ($p > \max\set{m, n}$). Empirical values are computed for $p = 500$ for a single trial. }
    \label{fig:lambda-minnz-bound}
\end{figure}

\begin{theorem}
[Basic asymptotic equivalent for ridge resolvents, real-valued regularization]
\label{cor:basic-ridge-asympequi-in-r}
Assume the setting of \cref{lem:basic-ridge-asympequi}. 
Let $\zeta_0, z_0 \in \reals$ be the unique solutions, satisfying $\zeta_0 < \lambdaminnz(\bSigma)$,
to system of equations
\begin{align}
    \label{eq:basic-ridge-asympequi-in-r:bounds}
    1 = \tfrac{1}{n} \tr \bracket{\bSigma^2 \paren{\bSigma - \zeta_0 \bI_p}^{-2}}, \quad 
    z_0 = \zeta_0 \paren{1 - \tfrac{1}{n} \tr \bracket{\bSigma \inv{\bSigma - \zeta_0 \bI_p}}}.
\end{align}
Then, for each $z \in \reals$ satisfying $z < \liminf z_0$, as $n, p \nearrow \infty$ such that
$0 < \liminf \tfrac{p}{n} \le \limsup \tfrac{p}{n} < \infty$,
we have
\begin{align}
    \label{eq:basic-ridge-asympequi-in-r}
    z \big( \tfrac{1}{n} \bX^\ctransp \bX - z \bI_p \big)^{-1}
    \asympequi
    \zeta \inv{\bSigma - \zeta \bI_p},
\end{align}
where $\zeta \in \reals$ is the unique solution in $(-\infty, \zeta_0)$ to
the fixed-point equation
\begin{equation+}
    \label{eq:basic-ridge-fp-in-r}
    z = \zeta \paren{1 - \tfrac{1}{n} \tr \bracket{\bSigma \inv{\bSigma - \zeta \bI_p}}}.
\end{equation+}
Furthermore, as $n, p \nearrow \infty$,
\editedinline{$\zeta \asympequi -\tfrac{1}{v(z)}$,}
where 
$v(z)$ is the companion Stieltjes transform of the spectrum of $\frac{1}{n} \mX^\ctransp \mX$
given by
\[
    v(z) = \tfrac{1}{n} \tr \Big[\big(\tfrac{1}{n} \mX \mX^\ctransp - z \mI_n \big)^{-1}\Big],
\]
and \editedinline{$z_0 \asympequi \lambdaminnz(\frac{1}{n} \mX^\ctransp \mX)$}.
\end{theorem}
\begin{proof}[Proof sketch]
To prove this corollary, we define $\zeta \defeq \tfrac{z}{c(z)}$ to obtain \cref{eq:basic-ridge-asympequi-in-r} from \cref{eq:basic-ridge-asympequi-in-c} for $z \in \complexset^+$, and also observe that \editedinlinetwo{$-\tfrac{1}{\zeta}$} is the limiting companion Stieltjes transform $v(z)$ of $\tfrac{1}{n} \mX \mX^\ctransp$ at $z$. This implies that $\zeta \in \complexset^+$ and that the mapping $z \mapsto \zeta$ is a holomorphic function on its domain, which includes all real $z < \liminf \lambdaminnz(\tfrac{1}{n} \mX \mX^\ctransp)$. We then identify the analytic continuation of the mapping $z \mapsto \zeta$ to the reals, which consists of careful bookkeeping to determine $z_0$, the least positive value of $z$ for which $\zeta$ does not exist, which must be asymptotically equal to $\lambdaminnz(\tfrac{1}{n} \mX \mX^\ctransp)$. 
The proof details can be found in \Cref{sec:proof:cor:basic-ridge-asympequi-in-r} 
of the supplementary material.
\end{proof}

\begin{remark}
    [The case of $z = 0$]
    The form of the equivalence \cref{eq:basic-ridge-asympequi-in-c}
    is slightly different as compared with \cref{eq:basic-ridge-asympequi-in-r}
    in that the resolvent $(\tfrac{1}{n} \bX^\ctransp \bX - z \bI_p)^{-1}$ has a normalizing multiplier of $z$ in the latter case.
    This enables continuity of the left-hand side at $z = 0$, in contrast to
    specializing the equivalence \cref{eq:basic-ridge-asympequi-in-c}
    to real $z$, where both the left- and right-hand sides may diverge as $z \nearrow 0$.
\end{remark}

Our main result in the next section for sketching follows directly from this theorem and shares a very similar form. For this reason, we defer discussion about the interpretation of the solutions to the above equations for our reformulation under the sketching setting; however, analogous interpretations will apply to the above theorem.

\section{Main results}
\label{sec:main_results}

One way to think about \cref{cor:basic-ridge-asympequi-in-r}
is that the data matrix $\bX = \bZ \mSigma^{1/2}$ is a sketched version of the (square root) covariance matrix $\mSigma^{1/2}$,
where $\bZ$ acts as a sketching matrix. The sketching is done by ``nature'' in the form of the $n$ observations,
rather than by the statistician, but is otherwise mathematically identical to sketching. Using this insight, along with the Woodbury identity,
we can adapt the random matrix resolvent equivalence in \cref{cor:basic-ridge-asympequi-in-r} to a sketched (regularized) pseudoinverse equivalence. To emphasize the shift in perspective, we denote the dimensionality of the sketched data as $q$ (replacing $n$), replace $\mSigma$ with $\mA$, and absorb the normalization by $\tfrac{1}{q}$ (replacing $\tfrac{1}{n}$) into the sketching matrix $\mS$ (replacing $\mZ$), so that the sketching transformation is norm-preserving (see \Cref{rem:norm-preserving-sketch} for more details).

\subsection{First-order equivalence}

Our first result provides a first-order equivalence for the sketched regularized pseudoinverse.
By first-order equivalence,
we refer to equivalence for matrices that involve 
the \emph{first} power of the ridge resolvent.
We also present a second-order equivalence
for matrices that involve the \emph{second} power
of the ridge resolvent in \Cref{sec:second-order-sketch-equi}.

In preparation for the statements to follow,
recall that $r(\bA) = \frac{1}{p} \sum_{i=1}^{p} \1\{ \lambda_i(\bA) > 0 \}$,
or in other words, the normalized number of non-zero eigenvalues of $\bA$.
Note that $0 \le r(\bA) \le 1$.

\begin{theorem}
    [Isotropic sketching equivalence]
    \label{thm:sketched-pseudoinverse}
    Let $\mA \in \complexset^{p \times p}$ be a positive semidefinite
    matrix such that $\| \mA \|_{\op}$ is uniformly bounded in $p$
    and $\liminf \lambdaminnz(\bA) > 0$.
    Let $\sqrt{q}\mS \in \complexset^{p \times q}$ be a random matrix
    consisting of i.i.d.\ random variables that have mean 0, variance 1, and finite $8 + \delta$ moment for some $\delta > 0$. 
    Let $\lambda_0, \mu_0 \in \reals$ be the unique solutions, satisfying $\mu_0 > - \lambdaminnz(\mA)$, to the system of equations
    \begin{align}
        \label{eq:mu0-lambda0-fps}
        1 = \tfrac{1}{q} \tr \bracket{\mA^2 \paren{\mA + \mu_0 \mI_p}^{-2}}, \quad
        \lambda_0 = \mu_0 \paren{1 - \tfrac{1}{q} \tr \bracket{\mA \inv{\mA + \mu_0 \bI_p}}}.
    \end{align}
    Then, 
    as $q, p \nearrow \infty$ 
    such that
    $0 < \liminf \tfrac{q}{p} \le \limsup \tfrac{q}{p} < \infty$, 
    the following asymptotic equivalences hold:
    \begin{enumerate}[topsep=1em,parsep=0pt,label=(\roman*)]
        \item for any $\lambda > \limsup \lambda_0$, 
        we have
    \end{enumerate}
    \vspace{-\abovedisplayskip}
    \begin{align}
        \label{eq:thm:sketched-pseudoinverse-A-half}
        \mA^{1/2} \mS \big( \mS^\ctransp \mA \mS + \lambda \mI_q
        \big)^{-1} \mS^\ctransp
        \asympequi \mA^{1/2} \inv{\mA + \mu \mI_p};
    \end{align}
    \begin{enumerate}[resume*]
        \item if furthermore either $\lambda \neq 0$ or $\limsup \tfrac{q}{p} < \liminf r(\mA)$,
        we have
    \end{enumerate}
    \vspace{-\abovedisplayskip}
    \begin{align}
        \label{eq:thm:sketched-pseudoinverse}
        \mS \big( \mS^\ctransp \mA \mS + \lambda \mI_q
        \big)^{-1} \mS^\ctransp
        \asympequi \inv{\mA + \mu \mI_p},
    \end{align}
    where $\mu$
    is the unique solution in $(\mu_0, \infty)$ to the fixed point equation
    \begin{align}
        \label{eq:sketched-modified-lambda}
        \lambda = \mu \paren{ 1 - \tfrac{1}{q} \tr \bracket{\mA \inv{\mA + \mu \mI_p}}}.
    \end{align}
    Furthermore, as $p, q \to \infty$,
    \editedinline{$\mu \asympequi \tfrac{1}{ \widetilde{v}(\lambda)}$,}
    where
    \begin{align}
        \widetilde{v}(\lambda) = \tfrac{1}{q} \tr \bracket{\big( \mS^\ctransp \mA \mS + \lambda \mI_q
        \big)^{-1} },
    \end{align}
    and \editedinline{$\lambda_0 \asympequi - \lambdaminnz(\mS^\ctransp \mA \mS)$}. 
\end{theorem}

\begin{proof}[Proof sketch]
We begin by considering the case that $\mA$ satisfies $\limsup \bignorm[\op]{\mA^{-1}} < \infty$. Then we can rewrite the left-hand side of \cref{eq:thm:sketched-pseudoinverse-A-half} or \cref{eq:thm:sketched-pseudoinverse}
such that we can apply \cref{cor:basic-ridge-asympequi-in-r} with $\mX = \sqrt{q} \mS^\ctransp \mA^{1/2}$, $\lambda = -z$, and $\mu = -\zeta$. For any $\lambda > -\liminf z_0$,
\begin{subequations}
\begin{align}
    \mA^{1/2} \mS \biginv{\mS^\ctransp \mA \mS + \lambda \mI_q } \mS^\ctransp \mA^{1/2} &= 
    \mA^{1/2} \mS \mS^\ctransp \mA^{1/2} \biginv{\mA^{1/2} \mS \mS^\ctransp \mA^{1/2} + \lambda \mI_p } \\
    & = \mI_p - \lambda \biginv{\mA^{1/2} \mS \mS^\ctransp \mA^{1/2} + \lambda \mI_p } \\
    & \asympequi \mI_p - \mu \biginv{\mA + \mu \mI_p} \\
    & = \mA^{1/2} \biginv{\mA + \mu \mI_p} \mA^{1/2}.
\end{align}
\end{subequations}
We can then multiply on the right, or both left and right, by $\mA^{-1/2}$ to obtain the results in \cref{eq:thm:sketched-pseudoinverse-A-half} and \cref{eq:thm:sketched-pseudoinverse}, respectively, by the product rule of asymptotic equivalences. If $\mA$ does not have a norm-bounded inverse, we can apply the above result for $\mA_\delta \defeq \mA + \delta \mI_p$ for $\delta > 0$ and make a uniform convergence argument for interchanging limits of $p$ and $\delta$ to prove the equivalence in \cref{eq:thm:sketched-pseudoinverse}. We then multiply by $\mA^{1/2}$ and make another uniform convergence argument to extend this equivalence to the case $\lambda = 0$ to obtain the equivalence in \cref{eq:thm:sketched-pseudoinverse-A-half}. The details can be found in \Cref{sec:proof:thm:sketched-pseudoinverse} of the supplementary material.
\end{proof}

In words, the sketched pseudoinverse of $\mA$ with regularization $\lambda$ is asymptotically equivalent to the regularized inverse of $\mA$ with regularization $\mu$, and the relationship between $\lambda$ and $\mu$ asymptotically depends only on $\mA$, $p$, and $q$. As mentioned in \Cref{sec:preliminaries}, this implies for example that the elements of the sketched pseudoinverse converge to the elements of the ridge-regularized inverse. We illustrate this in \cref{fig:empirical-concentration}, where for a diagonal $\mA$, the off-diagonals of the sketched pseudoinverse quickly converge to zero as $p$ increases, while the diagonals converge to the diagonals of the regularized inverse of $\mA$.

\begin{figure}[t]
    \centering
    \includegraphics[width=5in]{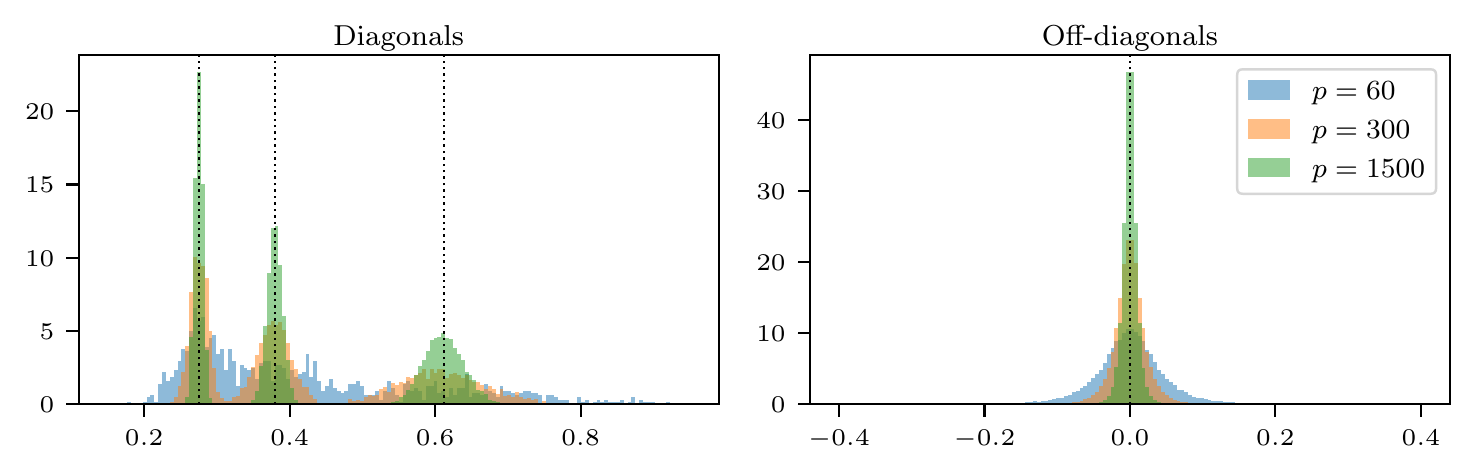}
    \caption{Empirical density histograms over 20 trials demonstrating the concentration of the elements of $\mS \inv{\mS^\transp \mA \mS + \lambda \mI} \mS^\transp$ for real Gaussian $\mS$ and diagonal $\mA$ taking values $\set{0, 1, 2}$ with equal frequency along the diagonal. We choose $\lambda = 1$ and $q = \floor{\alpha p}$ for $\alpha = 0.8$ over $p \in \set{60, 300, 1500}$.  As expected by \cref{thm:sketched-pseudoinverse}, the individual elements of the sketched pseudoinverse converge to those of $\inv{\mA + \mu \mI}$, where for this problem $\mu \approx 1.63$. Therefore, the diagonals concentrate with equal mass around $\set{1 / (a + \mu) : a \in \set{0, 1, 2}}$ (black, dotted), and the off-diagonals concentrate around 0.}
    \label{fig:empirical-concentration}
\end{figure}

Below we provide several remarks on the assumptions and implications
of \cref{thm:sketched-pseudoinverse}. It will be useful to interpret the equations in terms of the sketching aspect ratio $\alpha \defeq \tfrac{q}{p}$.

\begin{remark}
    [Normalization choice for the sketching matrix]
    \label{rem:norm-preserving-sketch}
    We remark that the normalization factor $\sqrt{q}$ in $\sqrt{q} \bS$
    of the sketching matrix is such that
    the norm of the rows of $\bS$ is $1$ in expectation.
    This is done so that 
    $\EE[ \| \bS^\ctransp \bx \|_2^2] = \| \bx \|_2^2$
    as $\EE[\bS \bS^\ctransp] = \bI_p$.
    One can alternately consider sketching matrices with normalization 
    $\sqrt{p} \bS$ such that the columns have norm $1$ in expectation.
    It is easy to write an equivalent version of \Cref{thm:sketched-pseudoinverse}
    with such a normalization. 
    We choose to focus on the former scaling
    because it is more common in practice.
\end{remark}

\begin{remark}
    [On assumptions]
    The assumptions imposed in \cref{thm:sketched-pseudoinverse}
    are quite mild.
    In particular,  the sequences of matrices $\bA$ 
    being sketched can be random, so long as they are independent of $\bS$.
    Furthermore, the spectrum of the sequences of matrices $\bA$
    need not converge to a fixed spectrum.
    The aspect ratio $\alpha$ of the sketching matrices $\bS$
    also need not converge to a fixed number.
    The reason this is possible is because
    we are not expressing the sketched resolvent
    in terms of the limiting spectrum of $\bS$ and $\bA$,
    but rather relating it through $\bA$ and a parameter $\mu$
    that depends on $\alpha$ and $\bA$ 
    (and the original regularization level $\lambda$),
    which allows us to keep our assumptions weak. 
\end{remark}

\begin{edited}
\begin{remark}[Rotationally invariant unregularized sketching]
\label{rem:rotationally-invariant}
When $\lambda = 0$, the first-order equivalence in fact holds for any sketching matrix $\mS$ that is rotationally invariant on the left and is not limited to i.i.d.\ sketching matrices. That is, if we look at the singular value decomposition of $\mS = \mU \mD \mV^\ctransp$, the left singular vectors $\mU$ are drawn from the Haar distribution over matrices with orthonormal columns. For $q \leq \mathrm{rank}(\mA)$, $\mS \biginv{\mS^\ctransp \mA \mS} \mS^\ctransp = \mU \biginv{\mU^\ctransp \mA \mU} \mU^\ctransp$, and so the sketched pseudoinverse does not depend on the spectrum of $\mS \mS^\ctransp$ at all and we can without any loss of generality apply \cref{thm:sketched-pseudoinverse}.
Given the universality of this result, it is no surprise that essentially all prior results for unregularized random projections \cite{derezinski2020precise,lejeune_javadi_baraniuk_2020,pmlr-v108-mutny20a}  agree even for sketches of varying spectra or determinantal point processes.
However, this universality does not extend to $\lambda \neq 0$ or to higher order equivalences; 
\editedinlinetwo{see \cref{thm:general-free-sketching}.}
\end{remark}

\begin{remark}[Proportionally sparse sketching]
\label{rem:sparse-sketching}
Although i.i.d.\ sketching is commonly referred to as ``dense sketching,'' \cref{thm:sketched-pseudoinverse} easily accommodates \editedinlinetwo{relatively} sparse sketches that are faster to apply. We can draw $[\mS]_{ij}$ from a distribution taking value $0$ with probability $1 - \tfrac{q}{p}$ and still satisfy the bounded $8 + \delta$ moment condition, leading to an $\mS$ with $O(q^2)$ nonzero elements with high probability. This means that a vector multiply $\mS^\ctransp \vu$ has cost $O(q^2)$ rather than $O(pq)$, which can be sufficient in many cases to make the cost of sketching negligible (see an example in \Cref{sec:sketch-and-project}). \editedinlinethree{This approach is essentially identical to the LESS-uniform embedding proposed by \cite{derezinski2021newtonless} as a special case, although LESS-uniform sketches can be ``truly sparse'' (less than $O(q^2)$) with additional incoherence assumptions on $\mA$.} 
\editedinlinetwo{It is worth remininding that since the ratio $\tfrac{q}{p}$ is bounded, strictly speaking all of these costs are $O(p^2)$; however, the relative advantages are often still computationally meaningful (see \Cref{fig:sketch-and-project}).}
Faster $O(p \log p)$ sketches are 
\editedinlinetwo{not covered by this theorem, but we expect most such sketches to be covered by our extension in \Cref{thm:general-free-sketching}.}
\end{remark}
\end{edited}

\begin{remark}
    [The case of $\lambda = 0$]
    \label{rem:lamda-zero-case}
    While the form in \cref{eq:thm:sketched-pseudoinverse} is the most general, it does not hold for $\lambda = 0$ if the sketch size is larger than the rank of $\mA$, since the inverse is unbounded. However, in machine learning settings such as ridge(less) regression, we only need to evaluate the regularized pseudoinverse $\mS (\mS^\ctransp \tfrac{1}{n} \mX^\ctransp \mX \mS + \lambda \mI_p)^{-1} \mS^\ctransp \tfrac{1}{\sqrt{n}} \mX$. Thus, we can apply the form in \cref{eq:thm:sketched-pseudoinverse-A-half} with $\mA^{1/2} = (\tfrac{1}{n} \mX^\ctransp \mX)^{1/2}$, which is sufficient for any downstream analysis.
\end{remark}

\begin{remark}
    [Alternate form of equivalence representation]
    Expressed in terms of $\widetilde{v}(\lambda)$,
    the equivalence \cref{eq:thm:sketched-pseudoinverse} becomes 
    \begin{equation}
        \bS \big( \bS^\ctransp \bA \bS + \lambda \bI_q \big)^{-1} \bS^\ctransp
        \asympequi
         \tv(\lambda) \inv{\tv(\lambda) \bA + \bI_p},
    \end{equation}
    and the fixed-point equation \cref{eq:sketched-modified-lambda} becomes
    \begin{equation}
        \lambda 
        = \frac{1}{\tv(\lambda)}
        - \tfrac{1}{q} \tr \bracket{\bA \inv{\tv(\lambda) \bA + \bI_p} }.
    \end{equation}
\end{remark}

\subsection{Second-order equivalence}
\label{sec:second-order-sketch-equi}

Although the equivalence in \cref{thm:sketched-pseudoinverse} holds for first order trace functionals, this equivalence does not hold for higher order functionals. To intuitively understand why, it is helpful to reason about the asymptotic equivalence similarly to an equivalence of expectation in classical random variables. That is, we may have two random variables $X, Y$ with $\expect{X} = \expect{Y}$, but this does not allow us to make any conclusions about the relationship between $\expect{X^k}$ and $\expect{Y^k}$ for $k > 1$. In the same way, our first-order asymptotic equivalence does not directly tell us higher order equivalences.

Fortunately, however, because of the resolvent structure of the regularized pseudoinverse, we can cleverly apply the derivative rule of the calculus of asymptotic equivalences to obtain a second order equivalence from the first order equivalence. Such a derivative trick has been employed in several prior works  \cite{dobriban_wager_2018, hastie_montanari_rosset_tibshirani_2022, karoui_kolsters_2011, ledoit_peche_2011, liu_dobriban_2019} \editedinlinetwo{for computing some specific second-order functionals, but we extend to generic second-order functionals}. This approach could in principle be repeated for higher order functionals.

\begin{theorem}
    [Second-order isotropic sketching equivalence]
    \label{thm:second-order-sketch}
    Consider the setting of \cref{thm:sketched-pseudoinverse}. If $\mPsi \in \complexset^{p \times p}$ is a deterministic or random positive semidefinite matrix independent of $\mS$ with $\norm[\op]{\mPsi}$ uniformly bounded in $p$, then if either $\lambda \neq 0$ or $\limsup \tfrac{q}{p} < \liminf r(\mA)$,
    \begin{align}
        \mS \biginv{\mS^\ctransp \mA \mS + \lambda \mI_q} \mS^\ctransp \mPsi \mS \biginv{\mS^\ctransp \mA \mS + \lambda \mI_q} \mS^\ctransp
        \asympequi
        \inv{\mA + \mu \mI_p} (\mPsi + \mu' \mI_p) \inv{\mA + \mu \mI_p},
    \end{align}
    where $\mu$ is as in \cref{thm:sketched-pseudoinverse}, and
    \begin{align}
        \label{eq:mu-prime}
        \mu' = \frac{\frac{1}{q} \tr \bracket{\mu^3 \inv{\mA + \mu \mI_p} \mPsi \inv{\mA + \mu \mI_p} }}{\lambda + \frac{1}{q} \tr \bracket{\mu^2 \mA \paren{\mA + \mu \mI_p}^{-2}} } \geq 0.
    \end{align}
\end{theorem}

\begin{proof}[Proof]
By assumption, there exists $M < \infty$ such that $M > \limsup \bignorm[\op]{\biginv{\mS^\ctransp \mA \mS + \lambda \mI_q}}$ and $M > \limsup \bignorm[\op]{\inv{\mA + \mu \mI_p}}$ almost surely (see proof details for \cref{thm:sketched-pseudoinverse} in the supplementary material). Define $\mB_z \defeq \mA + z \mPsi$. Then for all $z \in D$, where
\begin{align}
    D = \bigset{z \in \complexset \colon \limsup \bigparen{|z| M \norm[\op]{\mPsi} \max \bigset{\norm[\op]{\mS}^2, 1}} < \tfrac{1}{2}},
\end{align} we have that $\max \bigset{\limsup \bignorm[\op]{\biginv{\mS^\ctransp \mB_z \mS + \lambda \mI_q}}, \limsup \bignorm[\op]{\inv{\mB_z + \mu \mI_p}}}
\leq 2 M$. Therefore, we can apply the differentiation rule of asymptotic equivalences for all $z \in D$:
\begin{subequations}
\begin{align}
    -\mS \biginv{\mS^\ctransp \mB_z \mS &+ \lambda \mI_q} \mS^\ctransp \mPsi \mS \biginv{\mS^\ctransp \mB_z \mS + \lambda \mI_q} \mS^\ctransp
    = \tfrac{\partial}{\partial z} \mS \biginv{\mS^\ctransp \mB_z \mS + \lambda \mI_q} \mS^\ctransp \\
    &\asympequi \tfrac{\partial}{\partial z} \biginv{\mB_z + \mu(z) \mI_p} \\
    &= - \biginv{\mB_z + \mu(z) \mI_p} \paren{\mPsi + \tfrac{\partial}{\partial z} \mu(z) \mI_p} \biginv{\mB_z + \mu(z) \mI_p}.
\end{align}
\end{subequations}
We let $\mu'(z) = \tfrac{\partial}{\partial z} \mu(z)$, and then we can divide \cref{eq:sketched-modified-lambda} by $\mu(z)$ and differentiate to obtain
\begin{align}
    \frac{\lambda \mu'(z)}{\mu(z)^2} = \tfrac{1}{q} \tr \bracket{\mPsi \inv{\mB_z + \mu(z) \mI_p} - \mB_z \inv{\mB_z + \mu(z) \mI_p} \paren{\mPsi + \mu'(z) \mI_p } \inv{\mB_z + \mu(z) \mI_p} }.
    \nonumber
\end{align}
Solving for $\mu'(0)$ gives the expression in $\cref{eq:mu-prime}$. For the non-negativity of $\mu'$, see \cref{rem:alt-mu-prime} and its proof.
\end{proof}

\begin{figure}[t]
    \centering
    \includegraphics[width=5in]{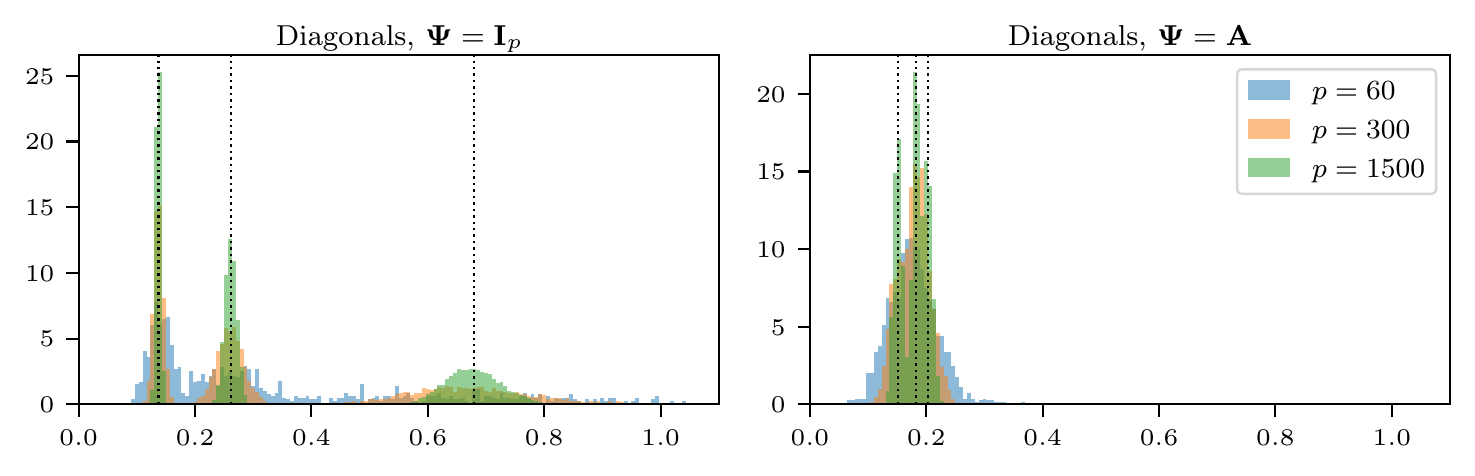}
    \caption{
    Empirical density histograms over 20 trials demonstrating the concentration of diagonal elements of  $\mS \inv{\mS^\transp \mA \mS + \lambda \mI} \mS^\transp \mPsi \mS \inv{\mS^\transp \mA \mS + \lambda \mI} \mS^\transp$ for $(\mS, \mA, \lambda)$ as in \cref{fig:empirical-concentration} and $\mPsi \in \set{\mI_p, \mA}$. As expected by \cref{thm:second-order-sketch}, the individual elements of the sketched pseudoinverse converge to those of $\inv{\mA + \mu \mI}(\mPsi + \mu' \mI) \inv{\mA + \mu \mI}$ (black, dotted), where $\mu' \approx 0.813$ and $0.403$ for $\mPsi = \mI_p$ and $\mA$, respectively.}
    \label{fig:empirical-concentration-second}
\end{figure}

That is, the second-order equivalence is the same as plugging in the first-order equivalence and then adding a non-negative inflation $\mu' \inv[2]{\mA + \mu \mI}$. The inflation factor $\mu'$ depends linearly on the matrix $\mPsi$, but the inflation is always isotropic, rather than in the direction of $\mPsi$. It is non-negative in the same way that the variance of an estimator is also non-negative. Examples of quadratic forms where this second-order equivalence can be used include estimation error ($\mPsi = \mI$) and prediction error ($\mPsi = \mSigma$, the population covariance) in ridge regression problems. We give a demonstration of the concentration in \cref{fig:empirical-concentration-second}.
While typically $\mu' > 0$, it can go to 0 in the special case of $\mu = 0$ and $\mPsi$ sharing a subspace with $\mA$, as we discuss in \cref{rem:mu-prime-to-0}. 

\begin{remark}
    [The case of $\lambda = 0$]
    Similar to the variant form in \cref{eq:thm:sketched-pseudoinverse-A-half} of \cref{thm:sketched-pseudoinverse}, if we consider the slightly different form
    \begin{align}
        \mA^{1/2}
        \mS \biginv{\mS^\ctransp \mA \mS + \lambda \mI_q} \mS^\ctransp \mPsi &\mS \biginv{\mS^\ctransp \mA \mS + \lambda \mI_q} \mS^\ctransp
        \mA^{1/2} \\
        &\asympequi
        \mA^{1/2}
        \inv{\mA + \mu \mI_p} (\mPsi + \mu' \mI_p) \inv{\mA + \mu \mI_p}
        \mA^{1/2}
    \end{align}
    for the second-order resolvent,
    we do not need the $\lambda \neq 0$ or $\limsup \tfrac{q}{p} < \liminf r(\mA)$ restriction as stated in the theorem. Because the proof of this case is entirely analogous to the results in \cref{thm:sketched-pseudoinverse,thm:second-order-sketch}, we omit the proof.
\end{remark}

\section{Properties and examples}
\label{sec:properties}

Below we provide various analytical properties
of the quantities that appear in \Cref{thm:sketched-pseudoinverse,thm:second-order-sketch}.
See \Cref{sec:proofs-properties} in the supplementary material for their proofs.

\subsection{Lower limits}

The quantities $\lambda_0$ and $\mu_0$
provide the lower limits of regularization
in \Cref{thm:sketched-pseudoinverse}.
The following two remarks describe their behaviour
in terms of $\alpha$.

\begin{remark}
    [Dependence of $\mu_0$ and $\lambda_0$ on $\alpha$]
    \label{rem:mu0-lambda0-vs-alpha}
    Writing the first equation in \cref{eq:mu0-lambda0-fps} as
    \begin{equation+}
        \label{eq:mu0-fp-in-alpha}
        \alpha = \tfrac{1}{p} \tr \bracket{\mA^2 \paren{\mA + \mu_0 \mI_p}^{-2}},
    \end{equation+}
    note that for fixed $\mA$, $\mu_0$ only depends on $\alpha$.
    Furthermore, the equation indeed admits a unique solution for $\mu_0$ for a given $\alpha$.
    This can be seen by noting that the function
    $f: \mu_0 \mapsto \tfrac{1}{p} \tr[\bA^2 (\bA + \mu_0 \bI_p)^{-2}]$
    is monotonically decreasing in $\mu_0$,
    and 
    \[
        \tfrac{1}{p} \lim_{\mu_0 \searrow - \lambdaminnz(\mA)} \tr[\bA^2 (\bA + \mu_0 \bI_p)^{-2}] = \infty,
        \quad
        \text{and}
        \quad
        \tfrac{1}{p}\lim_{\mu_0 \nearrow \infty} \tr[\bA^2 (\bA + \mu_0 \bI_p)^{-2}] = 0.
    \]
    In addition,
    because $\mu_0(\alpha) = f^{-1}(\alpha)$,
    $\mu_0$
    is monotonically decreasing in $\alpha$,
    and $\lim_{\alpha \searrow 0} \mu_0(\alpha) = \infty$
    and $\lim_{\alpha \nearrow \infty} \mu_0(\alpha) = - \lambdaminnz(\bA)$.
    
    Given $\mu_0$,
    the second equation in \cref{eq:mu0-lambda0-fps} then
    provides $\lambda_0$ as
    \begin{equation+}
        \label{eq:lambda0-mu0-in-alpha}
        \lambda_0 
        = \mu_0 \paren{1 - \tfrac{1}{\alpha} \tfrac{1}{p} \tr \bracket{\mA \inv{\mA + \mu_0 \bI}}}.
    \end{equation+}
    For $\alpha \in (0, r(\bA))$,
    $\lambda_0 : \alpha \mapsto \lambda_0(\alpha)$
    is monotonically increasing,
    and $\lim_{\alpha \searrow 0} \lambda_0(\alpha) = - \infty$
    and $\lim_{\alpha \to r(\bA)} \lambda_0(\alpha) = 0$.
    When $\alpha = r(\bA)$, $\mu_0 = 0$
    and consequently $\lambda_0 = 0$.
    Finally, for $\alpha \in (r(\bA), \infty)$,
    $\lambda_0 : \alpha \mapsto \lambda_0(\alpha)$
    is monotonically decreasing in $\alpha$,
    and $\lim_{\alpha \nearrow \infty} \lambda_0(\alpha) = -\lambdaminnz(\bA)$.
    This follows from a short limiting calculation.
\end{remark}

\begin{remark}
    [Joint sign patterns of $\mu_0$ and $\lambda_0$]
    \label{rem:mu0-lambda0-signs}
    Observe from \cref{eq:mu0-lambda0-fps}
    the sign pattern summarized in \cref{tab:mu0-lambda0-signs}.
    \begin{table}[h!]
    \label{tab:mu0-lambda0-signs}
    \centering
    \caption{Sign patterns of $\lambda_0$ and $\mu_0$.}
    \begin{tabular}{c | c | c | c}
    $\alpha$ vs. $r(\bA)$ 
    & $\mu_0$ 
    & $\alpha$ vs. $\tfrac{1}{p} \tr[\bA (\bA + \mu_0 \bI)^{-1}]$
    & $\lambda_0$ \\
    \hline
    $\alpha > r(\bA)$ 
    & $< 0$
    & $\alpha = \tfrac{1}{p} \tr[\bA^2 (\bA + \mu_0 \bI)^{-2}]
    > \tfrac{1}{p} \tr[\bA (\bA + \mu_0 \bI)^{-1}]$ 
    & $< 0$ \\
    $\alpha = r(\bA)$ 
    & 0
    & $\alpha 
    = \lim_{x \searrow 0}
    \tfrac{1}{p} \tr[\bA^2 (\bA + x \bI)^{-2}]
    = \lim_{x \searrow 0}
    \tfrac{1}{p} \tr[\bA (\bA + x \bI)^{-1}]$
    & 0 \\
    $\alpha < r(\bA)$ 
    & $> 0$ 
    & $\alpha =  \tfrac{1}{p} \tr[\bA^2 (\bA + \mu_0 \bI)^{-2}]
    < \tfrac{1}{p}  \tr[\bA (\bA + \mu_0 \bI)^{-1}]$ 
    & $< 0$
    \end{tabular}
    \end{table}
\end{remark}

\subsection{First-order equivalence}

In general,
the exact $\mu$ depends on $\lambda$, $\alpha$, and $\bA$
via the fixed-point equation \cref{eq:sketched-modified-lambda}.
However, we can infer several properties of the behaviour
of $\mu$ as a function of $\lambda$ and $\alpha$
as summarized below.

\begin{proposition}
    [Monotonicities of $\mu$ in $\lambda$ and $\alpha$]
    \label{prop:monotonicities-lambda-alpha}
    For a fixed $\alpha \ge 0$, 
    the map $\lambda \mapsto \mu(\lambda)$,
    where $\mu(\lambda)$ is as defined in \cref{eq:sketched-modified-lambda}
    is monotonically increasing in $\lambda$
    over $(\lambda_0, \infty)$,
    and $\lim_{\lambda \searrow \lambda_0} \mu(\lambda) = \mu_0$, 
    while $\lim_{\lambda \nearrow \infty} \mu(\lambda) = \infty$.
    For a fixed $\lambda \ge 0$,
    the map $\alpha \mapsto \mu(\alpha)$
    where $\mu(\alpha)$ is as defined in \cref{eq:sketched-modified-lambda}
    is monotonically decreasing in $\alpha$ over $(0, \infty)$;
    when $\lambda < 0$,
    the map $\alpha \to \mu(\alpha)$ is monotonically decreasing
    over $(0, r(\bA))$ 
    and monotonically increasing over $(r(\bA), \infty)$.
    Furthermore,
    for any $\lambda \in (\lambda_0, \infty)$,
    $\lim_{\alpha \searrow 0} \mu(\alpha) = \infty$,
    and $\lim_{\alpha \nearrow \infty} \mu(\alpha) = \lambda$.
\end{proposition}

\begin{remark}
    [Joint signs of $\lambda$ and $\mu$]
    \label{rem:joint-signs-lambda-mu}
    When $\lambda \ge 0$, 
    for any $\alpha > 0$, 
    we have $\mu \ge 0$,
    where $\mu$ is the unique solution to 
    \cref{eq:sketched-modified-lambda} in $(\mu_0, \infty)$.
    When $\lambda < 0$,
    for $\alpha \le r(\bA)$,
    we have $\mu \ge 0$,
    while for $\alpha > r(\bA)$,
    we have 
    $\mathrm{sign}(\mu) 
    = \mathrm{sign}(\lambda)$.
\end{remark}

\begin{proposition}
    [Concavity, bounds, and asymptotic behaviour of $\mu$ in $\lambda$]
    \label{rem:concavity-mu-in-lambda}
    The function $\lambda \mapsto \mu(\lambda)$,
    where $\mu(\lambda)$ is the solution to \cref{eq:sketched-modified-lambda}
    is a concave function over $(\lambda_0, \infty)$.
    Furthermore, for any $\alpha \in (0, \infty)$,
    $\mu(\lambda) \le \lambda + \tfrac{1}{q} \tr[\mA]$ for all $\lambda \in (\lambda_0, \infty)$; and when $\alpha \le r(\bA)$,
    $\mu(\lambda) \geq \lambda$ for all $\lambda \in (\lambda_0, \infty)$,
    otherwise $\mu(\lambda) \geq \lambda$
    for $\lambda \geq 0$.
    Additionally,
    $\lim_{\lambda \nearrow \infty} |\mu(\lambda) - (\lambda + \tfrac{1}{q} \tr[\mA])| = 0$. 
\end{proposition}

\subsection{Second-order equivalence}

Below we provide a few additional properties related to the inflation factor $\mu'$
in \eqref{eq:mu-prime}, that appears in the statement of \Cref{thm:second-order-sketch}.

\begin{remark}
    \label{rem:alt-mu-prime}
    We have the following alternative form for $\mu'$:
    \begin{align}
        \mu' = \tfrac{1}{q} \tr \bracket{\mu^2 \inv{\mA + \mu \mI_p} \mPsi \inv{\mA + \mu \mI_p}} \frac{\partial \mu}{\partial \lambda}.
    \end{align}
    Note that the term $\frac{\partial{\mu}}{\partial \lambda}$ does not depend in any way on $\mPsi$, and that the remaining term is well-controlled for any $\mu > \mu_0$. Therefore, $\mu'$ will only diverge when $\frac{\partial{\mu}}{\partial \lambda}$ diverges, which occurs as $\lambda \to \lambda_0$. This is clearly visible in \cref{fig:mu-equiv} (top) as $\lambda$ approaches $\lambda_0$, where the slope of the curve tends to infinity. Additionally, because $\mu$ is increasing in $\lambda$, this decomposition shows that $\mu' \geq 0$.
\end{remark}

\begin{remark}[Vanishing $\mu'$]
\label{rem:mu-prime-to-0}
If $\mathrm{Ker}(\mA) \subseteq \mathrm{Ker}(\mPsi)$, then as $\mu \to 0$, $\mu' \searrow 0$. The best intuition for this is in the case $\mPsi = \mA$. Because we can only have $\mu = 0$ for $\alpha > r(\mA)$ and $\lambda = 0$, we have $\mS \biginv{\mS^\ctransp \mA \mS + \lambda \mI_q} \mS^\ctransp \mA \mS \biginv{\mS^\ctransp \mA \mS + \lambda \mI_q} \mS^\ctransp \big|_{\lambda = 0} = \mS \biginv{\mS^\ctransp \mA \mS + \lambda \mI_q} \mS^\ctransp \big|_{\lambda = 0}$, and the second-order equivalence reduces to the first-order equivalence with no inflation factor. This remarkable property means that sketching leads to extremely accurate estimates with no spectral distortion, but only in low-rank settings with little regularization. 
\end{remark}

\subsection{Illustrative examples}

In order to better understand \cref{thm:sketched-pseudoinverse,thm:second-order-sketch}, 
we consider a few examples with special choices of the matrix $\bA$. When the spectrum of $\mA$ converges to a particular distribution of eigenvalues, $\mu$ will converge to a value that is deterministic given $\mA$. %

\begin{figure}[t]
    \centering
    \includegraphics[width=\textwidth]{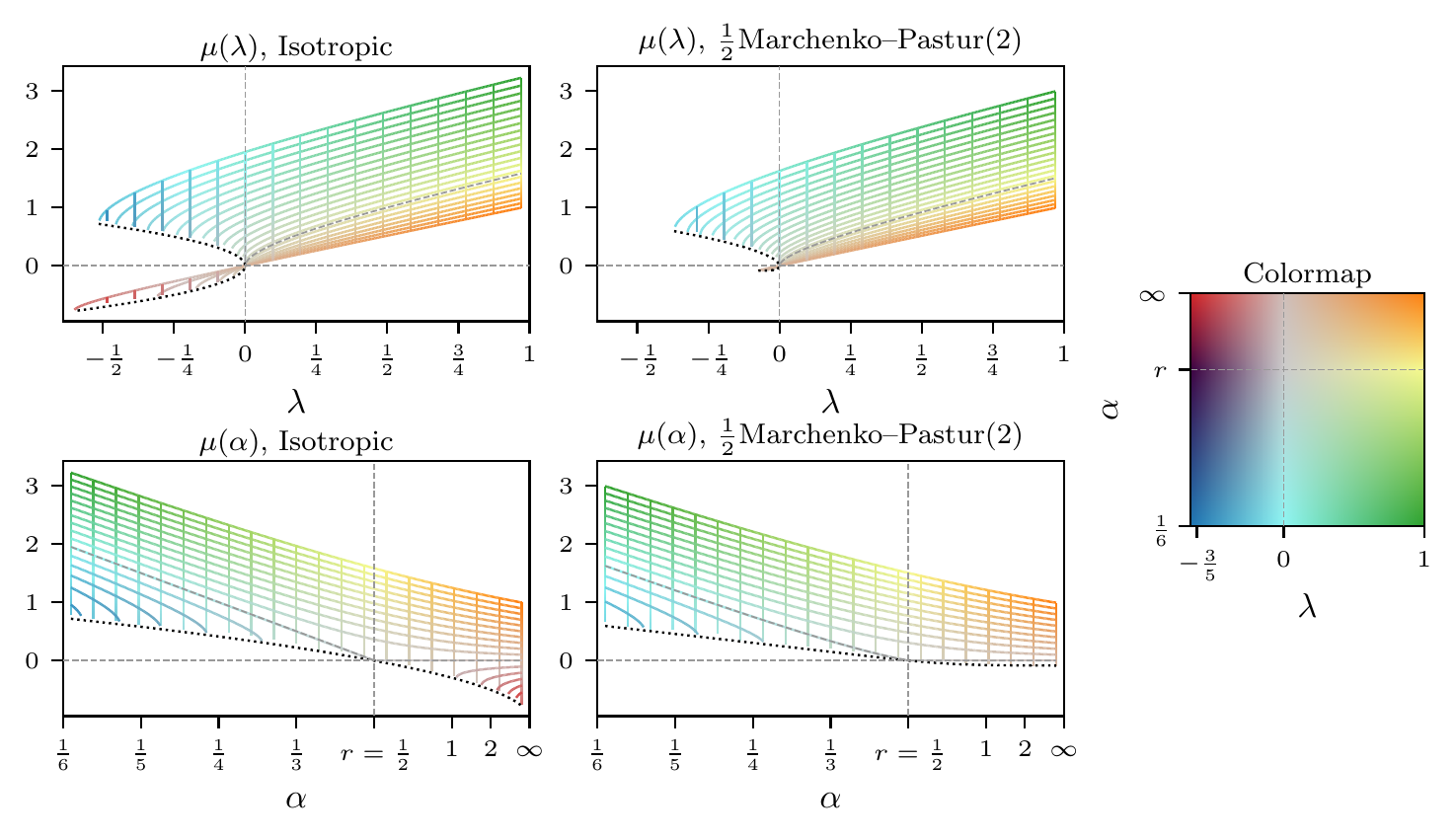}
    \caption{Plots of $\mu$ as a function of $\lambda$ and $\alpha$ for rank-deficient isotropic (left) and Marchenko--Pastur (middle) spectra, normalized so that $\frac{1}{p} \tr[\mA] = r = 1/2$. 
    The values of $\lambda$ and $\alpha$ in each location of the plot are indicated by the colormap (right), shared between the two views of each plot. 
    As we sweep $\alpha$, we also plot $(\alpha, \lambda_0, \mu_0)$ (black, dotted). 
    We also plot the lines $\mu = 0$, $\lambda = 0$, and $\alpha = r$ (gray, dashed).
    The scaling of the $\mu$ and $\lambda$ axes are linear, and the scaling of the $\alpha$ axis is proportional to $1/\alpha$. 
    In this way we can clearly capture the general $\mu \approx \lambda + \tfrac{1}{p} \tr[A] / \alpha$ relationship for $\lambda > 0$, as well as the limiting behavior of $\mu = \lambda$ for large $\alpha$.
    The most significant difference between the two distributions is that for the isotropic distribution, $\lambdaminnz(\mA) = 1$, while for the Marchenko--Pastur case, $\lambdaminnz(\mA) = (\sqrt{2} - 1)^2/2 \approx 0.0859$, limiting the achievable negative values of $\mu$ when $\lambda < 0$ and $\alpha > r$.
    }
    \label{fig:mu-equiv}
\end{figure}

\begin{figure}[t]
    \centering
    \includegraphics[width=5in]{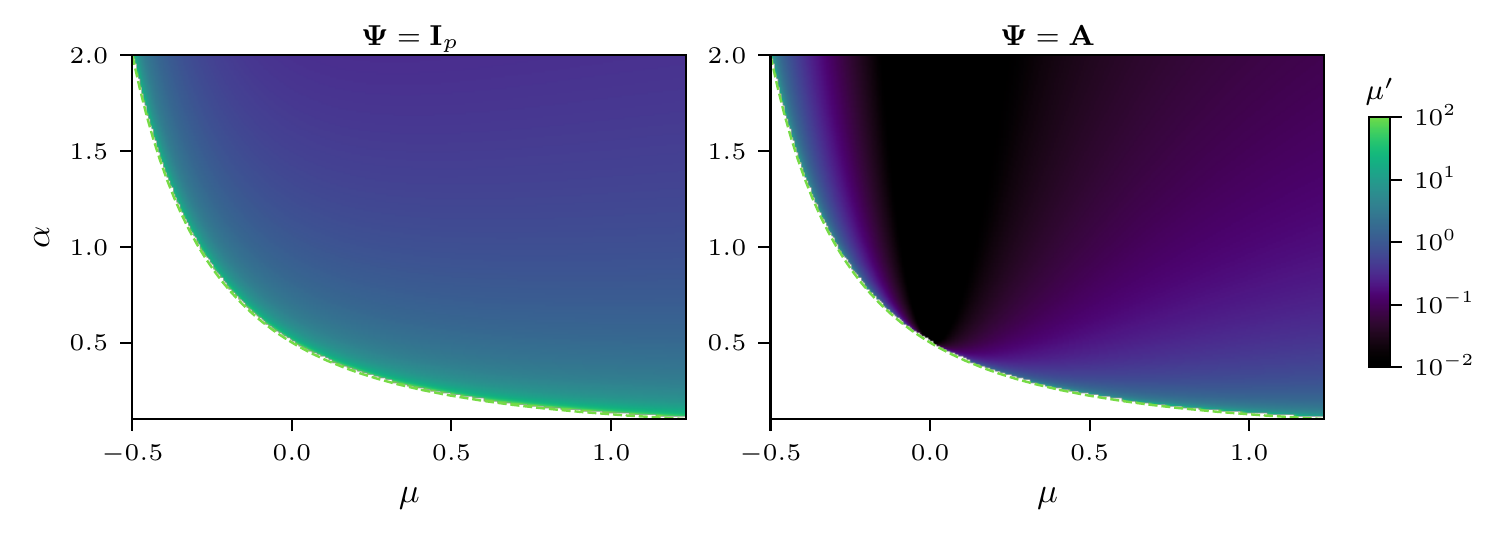}
    \caption{Plot of $\mu'$ as a function of $\mu$ and $\alpha$ for the rank-deficient isotropic spectrum with $r = 1/2$ for $\mPsi \in \set{\mI_p, \mA}$. In both cases, as $\mu \searrow \mu_0$ (dashed), $\mu' \nearrow \infty$. Otherwise, $\mu'$ is not too large. For $\mPsi = \mI_p$, $\mu'$ decays slowly in $\alpha$ and $\mu$. However, for $\mPsi = \mA$, there is a regime for $\alpha > r$ around $\mu = 0$ for which $\mu'$ tends to zero. Thus, the unregularized pseudo-inverse preserves $\mA$ remarkably well on its range when the sketch size is greater than the rank of the matrix, but outside of the range of $\mA$, it has non-negligible error.}
    \label{fig:mu-prime}
\end{figure}

\subsubsection{Isotropic rank-deficient matrix}

For the first example, let $0 < r \le 1$ be a real number.
We then consider $\mA = \begin{bsmallmatrix}\mI_{\floor{r p}} & \vzero \\ \vzero & \vzero \end{bsmallmatrix}$ such that $r(\mA) \to r$ as $p \nearrow \infty$. We have chosen the standard basis representation of this matrix, but the following results also hold for any $\mA$ that is isotropic on a subspace, regardless of basis. Such an $\mA$ includes settings such as $\mA = \mX^\transp \mX$ where $\mX \in \reals^{n \times p}$ is an orthogonal design matrix with orthonormal rows. In this case,
\begin{align}
    \mu = \frac{\lambda + \tfrac{r}{\alpha} - 1 + \sqrt{(\lambda + \frac{r}{\alpha} - 1)^2 + 4 \lambda}}{2}.
\end{align}
Furthermore, we have simple forms for $\mu_0$ and $\lambda_0$:
\begin{align}
    \mu_0 = \sqrt{\tfrac{r}{\alpha}} - 1, \quad
    \lambda_0 = - \paren{ \sqrt{\tfrac{r}{\alpha}} - 1}^2.
\end{align}
The expression for $\lambda_0$ can also be obtained directly from the minimum nonzero eigenvalue of the Marchenko--Pastur distribution with aspect ratio $\tfrac{\alpha}{r}$ and variance scaling $\tfrac{r}{\alpha}$, which describes $\mS^\ctransp \mA \mS$. In the case $\lambda = 0$, we have a very simple expression for $\mu$:
\begin{align}
    \mu = \begin{cases}
        \frac{r}{\alpha} - 1 & \text{if } \alpha < r, \\
        0 & \text{otherwise}.
    \end{cases}
\end{align}
We can also obtain the limiting behavior of $\mu$ for large $\lambda$ or small $\alpha$:
\begin{align}
    \label{eq:limiting-mu}
    \lim_{\lambda + \tfrac{r}{\alpha} \nearrow \infty} \frac{\mu}{\lambda + \tfrac{r}{\alpha}} = 1.
\end{align}
In \Cref{fig:mu-equiv} (left), we plot $\mu$ as a function of both $\lambda$ and $\alpha$.
We see that even for modest values of $\lambda > 0$ or $\alpha < r$, the relationship $\mu \sim \lambda + \tfrac{r}{\alpha}$ holds quite accurately. We see a clear transition point at $\alpha = r$ where $\lambda_0 = 0$, and on either side of which $\lambda_0$ decreases. Other properties from the previous sections, such as monotonicity, concavity in $\lambda$, and sign patterns are clearly visible in this plot as well.
We also plot $\mu'$ as a function of $\mu$ and $\alpha$ in \Cref{fig:mu-prime}, where we see that the inflation vanishes for $\mPsi = \mA$ only if $\alpha > r$ and $\mu = 0$. It is non-negligible otherwise, and tends to infinity as $\mu$ tends to $\mu_0$ for each $\alpha$.

\subsubsection{Marchenko--Pastur spectrum}

We also consider the case when $\bA$ is a random matrix
of the form $\bA = \tfrac{1}{n} \bZ^\top \bZ$,
where $\bZ \in \RR^{n \times p}$ contains i.i.d.\ entries
of mean $0$, variance $1$, and bounded moments of order $4 + \delta$
for some $\delta > 0$.
This case is of interest for real data settings
where $\bA$ will be a sample covariance matrix.
In this case, the spectrum of $\bA$ can be computed explicitly and is given by the Marchenko--Pastur law.
Computing $\mu$ explicitly in this case is possible,
but cumbersome. We instead provide numerical illustrations
on the behaviour of $\mu$ as a function of $\alpha$ and $\lambda$.

From \cref{fig:mu-equiv} (middle), we can see that the behavior of $\mu$ for the Marchenko--Pastur spectrum is not substantially different from the rank-deficient isotropic spectrum. The only regime that differs significantly is when $\alpha > r(\mA)$ and $\lambda < 0$, where $\lambda_0$ is much closer to $0$ than in the isotropic case, and so there is no equivalence for more negative values of $\lambda$.

It is also worth noting that when $\alpha < r(\bA) < 1$,
the na\"ive bound on the smallest
regularization $\lambda$ permissible is $0$
(as explained in the caption of \Cref{fig:lambda-minnz-bound}).
However, from \Cref{fig:mu-equiv}
we observe that the equivalence in \cref{thm:sketched-pseudoinverse} holds even for quite negative $\lambda$ (blue region), contrary to this na\"ive bound.
In fact, the true bound is almost the same as the rank-deficient isotropic case, $\lambda_0 = - \paren{ \sqrt{\tfrac{r}{\alpha}} - 1}^2$.

\begin{edited}
\section{Applications}
\label{sec:applications}
To demonstrate how to apply our theory to sketching-based algorithms, we give two concrete examples, demonstrating when the first-order equivalence can be sufficient to characterize performance and when the second-order equivalence is necessary.
We leave proof details to  \Cref{sec:proofs:applications} in the supplementary material.

\subsection{Sketch-and-project}
\label{sec:sketch-and-project}

The sketch-and-project algorithm, also known as the generalized Kaczmarz method, solves the satisfiable linear system $\mL \vx = \vb$ for some $\mL \in \complexset^{n \times p}$ via the following iterations:
\begin{align}
    \vx_{t} = \vx_{t-1} - \mL^\ctransp \mS_t (\mS_t^\ctransp \mL \mL^\ctransp \mS_t)^\dagger \mS_t^\ctransp (\mL \vx_{t-1} - \vb).
\end{align}
Here $\mS_t \in \complexset^{n \times m}$ are independently drawn random sketching matrices. 
This algorithm classically enjoys linear convergence of $\E \bigbracket{\smallnorm[2]{\vx_t - \vx_*}^2}$ where $\vx_* = \mL^\dagger \vb$ that depends only on the smallest eigenvalue of $\expect{\mL^\ctransp \mS_t (\mS_t^\ctransp \mL \mL^\ctransp \mS_t)^\dagger \mS_t^\ctransp \mL}$ \cite{gower2015randomized}. 
Since this is the same quantity of interest as in our sketching equivalence, we obtain a similar convergence guarantee
in the asymptotic limit almost surely by applying \cref{thm:sketched-pseudoinverse} with $\mA = \mL \mL^\ctransp$ 
\editedinlinetwo{(see \Cref{sec:proof:eq:sketch-and-project})}:
\begin{align+}
    \label{eq:sketch-and-project}
    \norm[2]{\vx_t - \vx_*}^2 \editedinlinethree{{}\lesssim{}} \rho^t \norm[2]{\vx_0 - \vx_*}^2 
    \quad \text{where} \quad
    \rho \defeq \frac{\mu}{\lambdaminnz(\mL \mL^\ctransp) + \mu}.
\end{align+}
\editedinlinethree{Here by $a_{n,t} \lesssim b_{n,t}$, we mean that for any fixed $t$, $\liminf_{n \to \infty} b_{n,t} - a_{n,t} \geq 0$, and the result holds for an implicit sequence of $\vx_0$, $\vx_*$ with increasing dimensions and uniformly bounded norms such that \Cref{thm:sketched-pseudoinverse} can be applied.}
Since there are no second-order effects, and we use $\lambda = 0$, this convergence result holds in fact for any rotationally invariant sketch by \cref{rem:rotationally-invariant}.
Asymptotically, assuming we can compute the product $\mL^\ctransp \mS_t$ efficiently, the computational bottleneck comes from evaluating the pseudoinverse $\mL^\ctransp \mS_t (\mS_t^\ctransp \mL \mL^\ctransp \mS_t)^\dagger$, which typically has complexity $O(mp \min\set{m, p})$.\footnote{\edited{Our remarks here also hold directly for any possible ``galactic'' matrix inversion algorithm of complexity $O(mp \min{\set{m, p}}^\delta)$ for some $\delta > 0$ \cite{alman2021laser}, provided $\mL^\ctransp \mS_t$ can be computed in similar time.}}
To reach a desired residual $\norm[2]{\mL \vx_t - \vb}^2 \leq \varepsilon$, we must run the algorithm for at most $t_\varepsilon = \ceil{\log(\varepsilon/\lambda_{\max}(\mL \mL^\ctransp) \norm[2]{\vx_0 - \vx_*}^2)/\log (\rho)}$ iterations. The total complexity of the algorithm is therefore $O(m^2 p t_\varepsilon)$ for $m < p$, compared to $O(n p \min{\set{n, p}})$ to solve the system directly.
Since both of these quantities diverge in the asymptotic limit, it is of more interest to study their quotient. To that end, we define the \emph{relative computation factor} $\alpha^2 t_\varepsilon$ for $\alpha = \tfrac{m}{n}$, which is equal to the quotient up to a factor of $\tfrac{\min{\set{n, p}}}{n}$, which does not depend on $\alpha$.

\begin{remark}
[Optimal sketch size for minimizing computation]
\label{rem:sketch-and-project}
The asymptotic relative computation factor $\alpha^2 t_\varepsilon$ is characterized as follows. 
For $\alpha \geq r(\mL)$, $t_\varepsilon = 1$ for all $\varepsilon$, and so $\alpha^2 t_\varepsilon = \alpha^2$.
For all sufficiently small $\varepsilon$, $\lim_{\alpha \searrow 0} \alpha^2 t_\varepsilon = 0$. 
For $0 < \alpha < r(\mL)$, $\lim_{\varepsilon \searrow 0} \alpha^2 t_{\varepsilon} = \infty$. 
Thus, for small $\varepsilon$, the computational complexity of sketch-and-project is minimized globally by letting $\alpha \searrow 0$ and locally by choosing $\alpha = r(\mL)$.
\end{remark}

\begin{figure}
    \centering
    \includegraphics[width=5.5in]{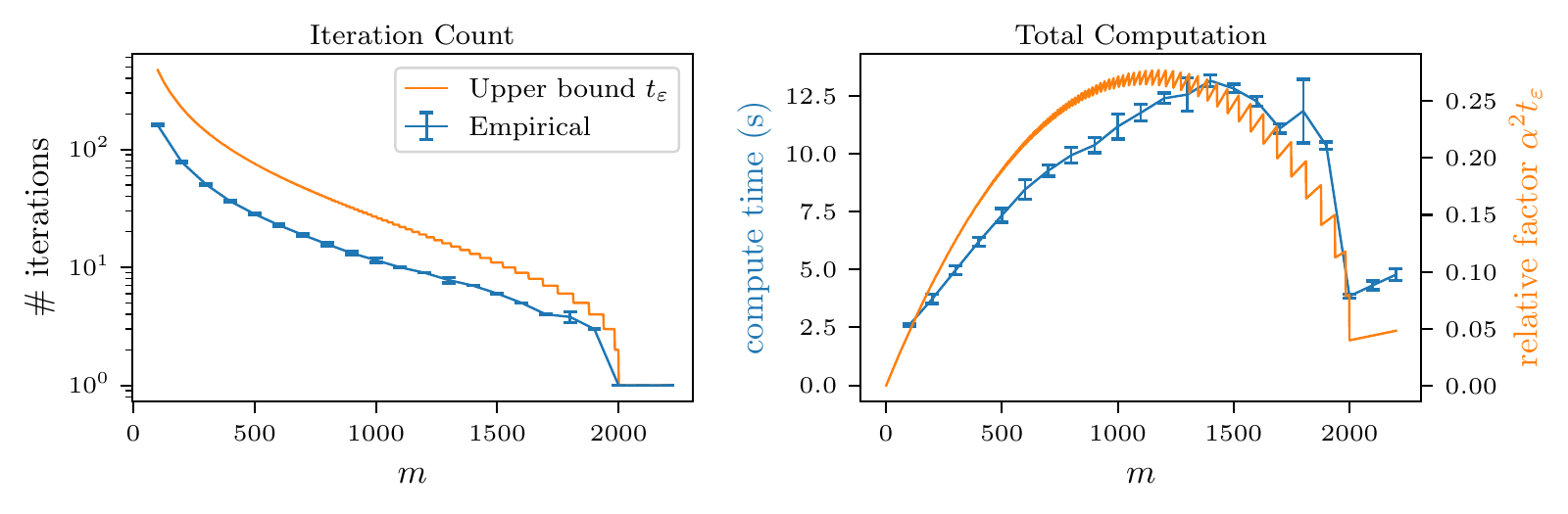}
    \caption{\editedinline{Empirical computation time of sketch-and-project as a function of sketch size $m$. We sample a fixed $[\mL]_{ij} \overset{\iid}{\sim} \normal(0, 1)$ and $\vx_* \sim \normal(\vzero, \tfrac{1}{p} \mI_p)$ for $n = 10^4$, $p = 2000$. We run the algorithm until $\tfrac{1}{n}\norm[2]{\mL \vx_t - \vb}^2 \leq 10^{-3}$. We find that the number of iterations (blue) matches our upper bound $t_\varepsilon$ (orange) up to a constant factor (left). Additionally (right), we find that the trend of the wall-clock time of the algorithm (blue) matches the relative computation factor $\alpha^2 t_\varepsilon$ (orange), and that the computation time is minimized by taking $\alpha$ as small as possible. Error bars denote standard deviation over 10 random trials.}} 
    \label{fig:sketch-and-project}
\end{figure}

We demonstrate this observation empirically in \cref{fig:sketch-and-project}. In order to keep the cost of evaluating $\mL^\ctransp \mS_t$ to $O(m^2p)$, we sample sparse Gaussian matrices $\mS_t$ according to \cref{rem:sparse-sketching} having elements drawn from $\normal(0, \tfrac{n}{m^2})$ with probability $\tfrac{m}{n}$ and 0 otherwise, such that there are $O(m^2)$ nonzero elements of $\mS_t$ with high probability. 

\subsection{Sketched ridge regression}
\label{sec:sketched-ridge}
In sketch-and-project, we introduced new randomness in each iteration, and as a result the first-order equivalence was sufficient to characterize the algorithm's performance. However, with less randomness, the second-order effects are much more pronounced. We illustrate this in the setting of sketched ridge regression, also known as sketch-and-solve, which is an important problem in randomized numerical linear algebra \cite{murray2023randomized}.

Concretely, we can define the sketched ridge regression problem for design matrix $\mL \in \complexset^{n \times p}$, targets $\vb \in \complexset^n$, and sketching matrix $\mS \in \complexset^{n \times m}$ as
\begin{align}
    \widehat{\vx} = \argmin_{\vx} \tfrac{1}{n} \bignorm[2]{\mS^\ctransp (\mL \vx - \vb)}^2 + \lambda \norm{\vx}^2.
\end{align}
To connect back to sketch-and-project from the previous section, a single iteration of sketch-and-project solves this exact problem if we set $\lambda = 0$ and replace $\vb$ by $\vb - \mL \vx_t$. For brevity and parallelism with sketch-and-project, we only consider this formulation of sketched ridge regression. However, similar analyses can be performed for ``dual'' sketching where we consider residuals $\mL \mS' \vx - \vb$, as well as joint sketching with residuals $\mS^\ctransp (\mL \mS' \vx - \vb)$; see \cite{lejeune2022ridge}.

The solution $\widehat{\vx}$ is given in terms of the sketched (regularized) pseudoinverse, which means we can obtain its first-order asymptotic equivalent from \cref{thm:sketched-pseudoinverse} with $\mA = \tfrac{1}{n} \mL \mL^\ctransp$:
\begin{align}
    \widehat{\vx} 
    = \tfrac{1}{n} \mL^\ctransp \mS \biginv{\mS^\ctransp \tfrac{1}{n} \mL \mL^\ctransp \mS + \lambda \mI_p}  \mS^\ctransp \vb
    \asympequi \tfrac{1}{n} \mL^\ctransp \biginv{\tfrac{1}{n} \mL \mL^\ctransp + \mu \mI_p} \vb \rdefeq \widehat{\vx}_\mathrm{equiv}.
\end{align}
Furthermore, we can characterize second-order errors; if we define the quadratic error\begin{align}
    \error_\mPhi(\vx, \vx') \defeq (\vx - \vx')^\ctransp \mPhi (\vx - \vx'),
\end{align}
we can apply \cref{thm:second-order-sketch} with $\mPsi = \tfrac{1}{n} \mL \mPhi \mL^\ctransp$ to obtain
\begin{align}
    \label{eq:ridge-error}
    \error_\mPhi \bigparen{\widehat{\vx}, \vx'}
    \asympequi \error_\mPhi \bigparen{ \widehat{\vx}_\mathrm{equiv}, \vx'} + \frac{\mu'}{n} \vb^\ctransp \biginv[2]{\tfrac{1}{n} \mL \mL^\ctransp + \mu \mI_n } \vb,
\end{align}
where
\begin{align}
    \mu' = \frac{\frac{1}{m} \tr \bracket{ \mu^3 \inv{\frac{1}{n} \mL \mL^\ctransp + \mu \mI_n} \frac{1}{n} \mL \mPhi \mL^\ctransp \inv{\frac{1}{n} \mL \mL^\ctransp + \mu \mI_n} }}{\lambda + \frac{1}{m} \tr \bracket{\mu^2 \frac{1}{n} \mL \mL^\ctransp \paren{\frac{1}{n} \mL \mL^\ctransp + \mu \mI_n}^{-2}} }
    \ge 0.
\end{align}
In other words, the error of the sketched solution can be decomposed into the error of the first-order equivalent solution plus an inflation quantity. Note that this inflation is only the additional effect due to sketching. This should not be conflated with estimate variance, which is generally defined to include the effect of noise in $\vb$, which will appear in both the $\error_\mPhi(\widehat{\vx}_\mathrm{equiv}, \vx')$ and inflation terms.

\begin{figure}
    \centering
    \includegraphics[width=5.5in]{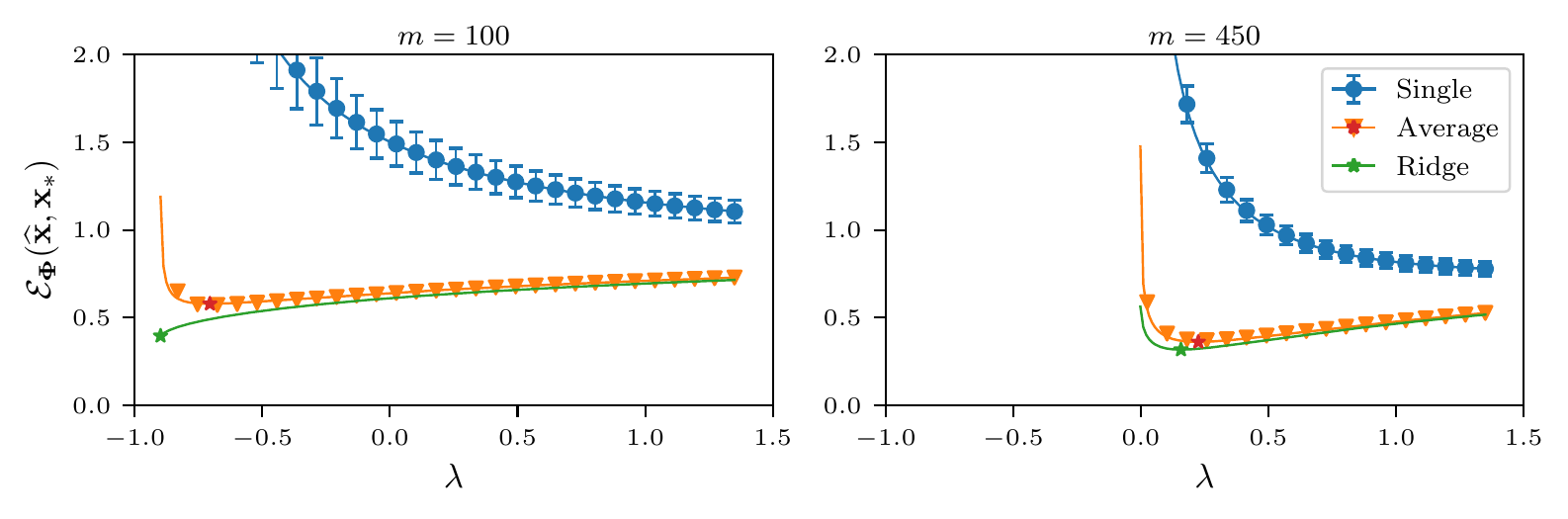}
    \caption{\editedinline{Estimation error $\error_\mPhi(\widehat{\vx} ,\vx_*) = \norm[2]{\widehat{\vx} - \vx_*}^2$ for a sketched ridge regression problem as a function of $\lambda$. We sample a fixed $[\mL]_{ij} \overset{\iid}{\sim} \normal(0, 1)$ and $\vx_* \sim \normal(\vzero, \tfrac{1}{p} \mI_p)$ and generate a fixed $\vb = \mL \vx_* + \vh$ with $\vh \sim \normal(\vzero, \sigma^2 \mI_n)$ for $n=2000$, $p=400$, and $\sigma = 1.5$.
    We plot the theoretical asymptotic error from \cref{eq:ridge-error} (lines) as well as empirical values (circles and triangles), averaging over $K=30$ random sketches $\mS$. We plot the single estimate error (blue), average of $K$ estimates (orange), and equivalent ridge predictor (green) for an undersampled setting ($m = 100$, left) and an oversampled setting ($m = 450$, right). In the undersampled setting, the optimal error (stars) for the averaged estimate is obtained by using negative $\lambda$.
    We emphasize that the data model here is underparameterized with a moderate signal-to-noise ratio and is not contrived to make negative regularization optimal as seen in some overparameterized settings \cite{kobak_lomond_sanchez_2020,wu_xu_2020}.}}
    \label{fig:ridge-error}
\end{figure}

The inflation term can be quite large when $\lambda$ is near $\lambda_0$, meaning the sketched solution is quite poor; however, by averaging $K$ independently sketched solutions we can replace $\mu'$ by $\tfrac{\mu'}{K}$, allowing us to control the inflation via randomized parallelization, such as in distributed settings. We demonstrate this theoretically and empirically in \cref{fig:ridge-error}. Note how in the undersampled regime with $m=100$, which is the regime of interest for distributed optimization as it reduces the computational cost per worker, the optimal regularization penalty $\lambda$ can in fact be \emph{negative}, even if the optimal ridge penalty $\mu$ for the equivalent problem is positive. Our theoretical characterization enables us to handle this case elegantly. \editedinlinetwo{The intuition behind this is that the smaller the sketch size is, the more regularization is added, and so to achieve a target regularization (the optimal ridge penalty), negative regularization may be required.}

\end{edited}

\section{Discussion and extensions}
\label{sec:discussion}

In this paper, we have provided a detailed look at the asymptotic effects of i.i.d.\ sketching on matrix inverses.
We have provided an extension of existing asymptotic equivalence results to real-valued regularization (including negative) and used this result to obtain both first- and second-order asymptotic equivalences for the sketched regularized pseudoinverse. 
\begin{edited}
We have also described how to apply these equivalences to analyze algorithms based on random sketching, providing novel insights into sketch-and-project and ridge regression as concrete examples.
\end{edited}

Our work is far from a complete characterization of sketching. We now list some natural extensions to our results.

\paragraph{Relaxing assumptions, strengthening conclusions}
As mentioned in \Cref{sec:main_results},
we make minimal assumptions on the base matrix $\bA$.
In particular, we do not assume that the empirical spectral
distribution of $\bA$ converges to any fixed limit.
The assumption that the maximum and minimum eigenvalue of $\bA$
be bounded away from $0$ and $\infty$ can be weakened.
In particular, one can let some eigenvalues to escape
to $\infty$,
and have some eigenvalues to decay to $0$,
provided certain functionals of the eigenvalues
remain bounded.
Our assumptions on the sketching matrix $\bS$ are also weak.
We do not assume any distributional structure on its entries
and only require bounded moments of order $8 + \delta$
for some $\delta > 0$.
Using a truncation strategy,
one can push this to only requiring moments of order $4 + \delta$
for some $\delta > 0$ for almost sure equivalences up to order $2$
that we show in this paper.
Finally, while our asymptotic results give practically relevant insights for finite systems, we lack a precise characterization for non-asymptotic settings. In particular, the rate of convergence depends on a number of factors including the choice of $\lambda$ and the higher order moments of the elements of $\mS$.

\paragraph{Generalized sketching}
Our assumption that the elements of the matrix $\mS$ are i.i.d.\ draws from some distribution limits its application in practical settings on two key fronts:
the effect of a rotationally invariant sketch is isotropic regularization, \editedinlinetwo{i.i.d.\ sketches can be slow to apply,} and there is unnecessary distortion of the spectrum of $\mA$ for $q \nearrow p$. We now discuss how to extend our framework to extend to more general classes of sketches that more closely align with those used in practice.

We may desire to use generalized non-isotropic ridge regularization, to perform Bayes-optimal regression 
(see, e.g., Chapter 3 of \cite{van-wieringen_2015})
or to avoid multiple descent \cite{mel2021regression,yilmaz2022descent}, or we may find ourselves using non-isotropic sketching matrices, such as in adaptive sketching~\cite{lacotte2019adaptive} where the sketching matrix depends on the data. We can cover these cases with the following extension of \cref{thm:sketched-pseudoinverse}. 
\begin{corollary}
    [Non-isotropic sketching equivalence]
    \label{cor:sketched-pseudoinverse-noniso}
    Assume the setting of \cref{thm:sketched-pseudoinverse}.
    Let $\mR$ be an invertible $p \times p$ positive semidefinite matrix,
    either deterministic or random but independent of $\mS$ with $\limsup \norm[\mathrm{op}]{\mR} < \infty$, and 
    let $\widetilde{\mS} = \mR^{1/2} \mS$.
    Then for each $\lambda > - \liminf \lambdaminnz(\widetilde{\mS}^\top \mA \widetilde{\mS})$
    as $p, q \nearrow \infty$ such that
    $0 < \liminf \tfrac{q}{p} \le \limsup \tfrac{q}{p} < \infty$,
    \begin{align}
        \widetilde{\mS}
        \biginv{\widetilde{\mS}^\top \mA \widetilde{\mS} + \lambda \mI_q }
        \widetilde{\mS}^\top
        \asympequi
        \biginv{\mA + \mu \mR^{-1} },
    \end{align}
    where $\mu$ is the most positive solution to
    \begin{align}
    \lambda = \mu \paren{ 1 - \tfrac{1}{q} \tr \bracket{\mA \inv{\mA + \mu \mR^{-1}}}}.
   \end{align}
\end{corollary}

\begin{proof}
The proof uses simple algebraic manipulations.
Observe that, since the operator norm is sub-multiplicative,
and $\| \mR \|_{\mathrm{op}}$, $\| \bA \|_{\mathrm{op}}$
are uniformly bounded in $p$,
$\| \mR^{1/2} \bA \mR^{1/2} \|_{\mathrm{op}}$ is also uniformly bounded in $p$.
Using \cref{thm:sketched-pseudoinverse},
we then have that
\[
    \bS
    \biginv{
        \bS^\top \mR^{1/2} \bA \mR^{1/2} \bS
        + \lambda \bI_q
    }
    \bS^\top
    \asympequi
    \biginv{
        \mR^{1/2} \bA \mR^{1/2} + \mu \bI_q
    }.
\]
Right and left multiplying both sides by $\mR^{1/2}$,
and writing $\widetilde{\bS} = \mR^{1/2} \bS$,
we get
\[
    \widetilde{\bS}
    \biginv{
        \widetilde{\bS}^\top
        \bA
        \widetilde{\bS}
        + \lambda \bI_q
    }
    \widetilde{\bS}^\top
    \asympequi
    \mR^{1/2}
    \biginv{
        \mR^{1/2} \bA \mR^{1/2}
        + \mu \bI_p
    }
    \mR^{1/2}
    =
    \biginv{\bA + \mu \mR^{-1} }
\]
as desired, completing the proof.
\end{proof}

Because non-isotropic sketching can be used to induce generalized ridge regularization, this can be exploited adaptively to induce a wide range of structure-promoting regularization via iteratively reweighted least squares, in a manner similar to adaptive dropout methods (see \cite{lejeune2021flipside} and references therein). Additionally, this result shows that methods applying ridge regularization to adaptive sketching methods, using for example $\mR = \mA$ as in \cite{lacotte2019adaptive}, are not equivalent to ridge regression but instead to generalized ridge regression. 

\paragraph{Free sketching}
Even among isotropic sketches, there can be a wide range of behavior beyond i.i.d.\ sketches. 
\begin{editedtwo}
It turns out that a more general result holds for \emph{free} sketching matrices (a notion from free probability that generalizes independence of random variables; see \cite{mingo2017free} for an introductory text). We state a complex version of the result in the following theorem
and defer the general extension to real arguments and investigation of properties to future work. 
\begin{theorem}[General free sketching]
    \label{thm:general-free-sketching}
    Let $\mA \in \complexset^{p \times p}$ be a positive semidefinite matrix
    and $\mS \in \complexset^{p \times q}$ be a 
    sketch such that the spectral distributions of $\mA$ and $\mS \mS^\ctransp$ converge almost surely to bounded distributions, and $\mS \mS^\ctransp$ is asymptotically free from any other matrices\footnote{\editedinlinethree{Standard zero-order freeness suffices when $p \mTheta$ has uniformly bounded operator norm. For general trace norm bounded $\mTheta$, first-order (infinitesimal) freeness~\cite{shlyakhtenko2018infinitesimal} is also required; see proof details. Unitarily invariant ensembles such as the orthogonal sketches in \Cref{cor:orthonormal-sketch} are known to satisfy all the necessary properties~\cite{cebron2022freeness}.}} with respect to the average trace $\tfrac{1}{p} \tr [\cdot]$ and has limiting S-transform $\Sxf_{\mS \mS^\ctransp}$ analytic on $\complexset^-$. Then for all $z \in \complexset^+$ there exists $\zeta \in \complexset^+$ such that
    \begin{align}
        \mS \biginv{\mS^\ctransp \mA \mS - z \mI_q} \mS^\ctransp \asympequi \biginv{\mA - \zeta \mI_p}.
    \end{align}
    Furthermore, 
    \begin{align}
        \zeta 
        \asympequi z \Sxf_{\mS \mS^\ctransp}
        \bigparen{-\tfrac{1}{p} \tr \bigbracket{\mA \biginv{\mA - \zeta \mI_p}}} 
        \quad \text{and} \quad
        \zeta \asympequi z \Sxf_{\mS \mS^\ctransp}
        \bigparen{-\tfrac{1}{p} \tr \bigbracket{\mS^\ctransp \mA \mS \biginv{\mS^\ctransp \mA \mS - z \mI_q}}}.
        \nonumber
    \end{align}
\end{theorem}
\begin{proof}[Proof sketch]
    The key idea of the proof is to use Jacobi's formula for a parameterized matrix: 
    $
        \frac{\partial}{\partial t} \logdet(\mB_t) =
        \tr \bigbracket{ \mB_t^{-1} \frac{\partial \mB_t}{\partial t}}
    $. 
    First we simplify by considering self-adjoint $\mTheta$
    and $\tmS = (\mS \mS^\ctransp)^{1/2}$ so that we can work entirely in dimension $p$.
    We can then define 
    $\mB_{t, \zeta} = \mA + t \mTheta - \zeta \mI_p$
    and
    $\mB_{t, z}^{\tmS} = \tmS (\mA + t \mTheta) \tmS - z \mI_p$.
    What we need to prove is that $\frac{\partial}{\partial t} \tfrac{1}{p} \logdet (\mB_{t,z}^{\tmS}) 
    \asympequi \frac{\partial}{\partial t} \tfrac{1}{p} \logdet (\mB_{t, \zeta})$ for some appropriate $\zeta$ at $t = 0$. We can eliminate the complexity introduced by $\mTheta$ by instead first differentiating with respect to $z$ and controlling the derivative with respect to $t$ using the second derivative.
    In the process, the choice of $\zeta$ presented in the statement naturally arises and can be shown to be correct using differential calculus. The details can be found in \Cref{sec:proofs-free} of the supplementary material.
\end{proof}
\end{editedtwo}

\begin{editedtwo}
That is, a more general version of \Cref{thm:sketched-pseudoinverse} holds for any $\mS$ that has the rotational invariance properties associated with freeness.  By the same reasoning as in \cref{rem:rotationally-invariant}, we expect that in the special case of $z \to 0$, free sketches will generally have the exact same first-order properties as the i.i.d.\ sketching case, since all spectral properties of $\mS \mS^\ctransp$ except the rank (sketch size) become irrelevant. 
\end{editedtwo}
\begin{edited}
In general, however, the mapping $z \mapsto \zeta$ depends 
on the spectrum of $\mS \mS^\ctransp$ and is not the same as in the i.i.d.\ sketching case.
\end{edited}

A particularly important sketching matrix that fits this broader definition is the orthogonal sketch. 
\editedinlinetwo{For example, randomized Fourier transforms are orthogonal and asymptotically free \cite{anderson2014liberating,lacotte2020optimal}.}
Unlike the i.i.d.\ sketch, an orthogonal sketch does not distort the spectrum near $q = p$ and so has less implicit regularization. We give proof details in \Cref{sec:proofs-free}.

\begin{corollary}[Orthogonal sketching]
\label{cor:orthonormal-sketch}
For $q \leq p$ with $\lim \tfrac{q}{p} = \alpha$, let $\sqrt{\tfrac{q}{p}} \mQ \in \complexset^{p \times q}$ be a Haar-distributed matrix with orthonormal columns, and let $\mA \in \complexset^{p \times p}$ be positive semidefinite with eigenvalues converging to a bounded limiting spectral measure. Then for any \editedinlinetwo{$\lambda > 0$,}
\begin{align}
    \mQ \big( \mQ^\ctransp \mA \mQ + \lambda \mI_q \big)^{-1} \mQ^\ctransp \asympequi \big( \mA + \gamma \mI_p \big)^{-1},
\end{align}
where 
$\gamma$ is the most positive solution to
\begin{align}
    \label{eq:ortho-gamma}
    \tfrac{1}{p} \tr \bracket{\inv{\mA + \gamma \mI_p}} (\gamma - \alpha \lambda) = 1 - \alpha.
\end{align}
Furthermore, for $\mu$ from \cref{thm:sketched-pseudoinverse} applied to the same $(\mA, \alpha, \lambda)$, we have $\gamma < \mu$.
\end{corollary}
\begin{proof}
Firstly, $\mQ \mQ^\ctransp$ and $\mA$ are almost surely asymptotically free~\cite[Theorem 4.9]{mingo2017free}.
\begin{editedtwo}
We can therefore apply \Cref{thm:general-free-sketching}. It is straightforward to obtain the analytic limiting S-transform $\Sxf_{\mQ \mQ^\ctransp}(w) = \frac{\alpha (1 + w)}{\alpha + w}$, from which we can obtain \cref{eq:ortho-gamma} from the equation $\gamma = \lambda \Sxf_{\mQ \mQ^\ctransp}(-\tfrac{1}{p} \tr [\mA \biginv{\mA + \gamma \mI_p}])$. That is, if we take $z \to -\lambda$, which is a well defined limit for $\Im(z) \searrow 0$ for any $\lambda > 0$, we have $\zeta \asympequi -\gamma$.
\end{editedtwo}
To see that $\gamma < \mu$, observe that we can write \cref{eq:sketched-modified-lambda} and \cref{eq:ortho-gamma} as 
\begin{align}
    \tfrac{\mu}{p} \tr \bracket{\big(\mA + \mu \mI_p \big)^{-1}} &= 1 - \alpha + \frac{\alpha \lambda}{\mu}, \\
    \tfrac{\gamma}{p} \tr \bracket{\big(\mA + \gamma \mI_p \big)^{-1}} &= 1 - \alpha + \alpha \lambda \tfrac{1}{p} \tr \bracket{\big(\mA + \gamma \mI_p \big)^{-1}}.
\end{align}
The left-hand sides of these two equations are the same increasing function of $\mu$ and $\gamma$, respectively, while the right-hand sides are decreasing functions, with the function of $\mu$ being strictly greater than the function of $\gamma$, since $\tfrac{1}{p} \tr \bracket{\big(\mA + \mu \mI_p \big)^{-1}} < \tfrac{1}{\mu}$ for $\mu > 0$. This means that the intersection with the decreasing function for $\gamma$ must occur for a smaller value than the intersection for $\mu$, proving the claim.
\end{proof}

In the statement, $\gamma < \mu$ means that the orthogonal sketch has less effective regularization than the i.i.d.\ sketch. For settings in which we desire to solve a linear system with as little distortion as possible, we therefore would much prefer an orthogonal sketch to an i.i.d.\ sketch, especially for $q \approx p$.
\editedinlinetwo{With additional work, one could extend this result to negative regularization as we have done in the i.i.d.\ sketching case. We leave it for future work.}

\begin{figure}[t]
    \label{fig:practical-concentration}
    \centering
    \includegraphics[width=6in]{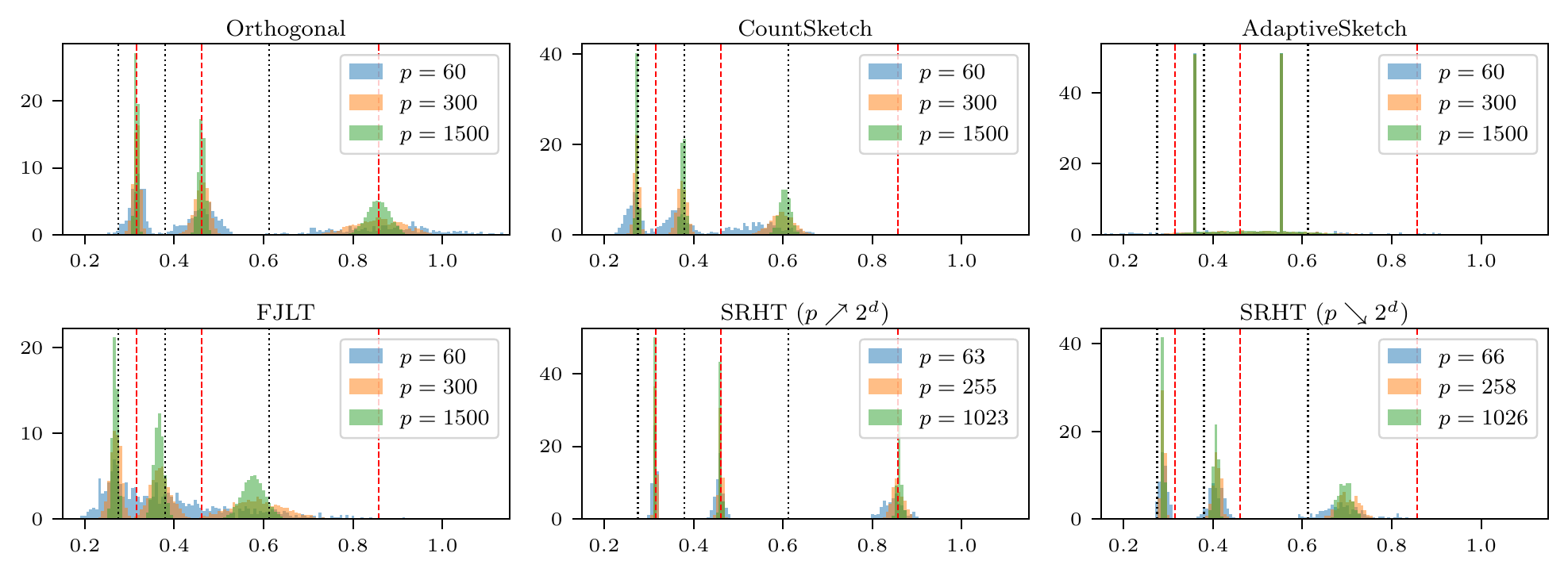}
    \caption{
    Empirical density histograms over 20 trials demonstrating the concentration of diagonal elements of  $\mS \inv{\mS^\transp \mA \mS + \lambda \mI} \mS^\transp$ for $\mA$ as in \cref{fig:empirical-concentration} with $q \approx 0.8 p$, $\lambda = 1$ and several normalized sketches $\mS$ commonly used in practice. We also plot the diagonals of the i.i.d.\ sketching equivalence $\inv{\mA + \mu \mI}$ (black, dotted) and the 
    orthogonal sketching equivalence $\inv{\mA + \gamma \mI}$ 
    from 
    \editedinlinetwo{\cref{cor:orthonormal-sketch}} 
    (red, dashed), where $\mu \approx 1.63$ and $\gamma \approx 1.17$.}
\end{figure}

In \cref{fig:practical-concentration}, we repeat the experiment from \cref{fig:empirical-concentration} for a variety of normalized non-i.i.d.\ sketches used frequently in practice. Both CountSketch \cite{charikar2002frequent} and the fast Johnson--Lindenstrauss transform (FJLT) \cite{ailon2009fast} behave similarly to i.i.d.\ sketching, with the FJLT slightly over-regularizing. As predicted by \cref{cor:sketched-pseudoinverse-noniso}, adaptive sketching with $\mR = \mA$ \cite{lacotte2019adaptive} behaves very differently from the other sketches, showing only two point masses instead of three since $\mA^{-1}$ is not well-defined for its eigenvalues of 0. Lastly, the subsampled randomized Hadamard transform (SRHT) \cite{tropp2011hadamard} is an orthogonal version of the FJLT, and our experiment elucidates the effect of zero padding on the Hadamard transform of the SRHT. The \editedinlinethree{fast} Hadamard transform is defined only for powers of 2, so for other dimensions, the common approach is to simply zero-pad the data to the nearest power of 2. However, from this experiment we can see that this zero-padding can have a significant impact on the effective regularization; for $p$ slightly smaller than a power of 2, the SRHT performs almost identically to an orthogonal sketch as expected. %
However, for $p$ slightly larger than a power of 2, there is significant effective regularization induced, even though the sketch is still norm-preserving.
\editedinlinetwo{This is because zero padding changes the spectrum, so the $S$-transform deviates from the orthogonal case.}

Our proposed framework of first- and second-order equivalence promises to provide a principled means of comparison of different sketching techniques. Once $\zeta$ from \cref{thm:general-free-sketching} can be determined for a given sketch (which depends on its spectral properties), an analogous result to \cref{thm:second-order-sketch} will directly follow to yield inflation with a factor of $\zeta'$. Armed with both $\zeta$ and $\zeta'$ for a collection of sketches, we can compare them using these bias and variance-style decompositions and make principled choices analogously to classical estimation techniques. \editedinline{Our best guidance to practitioners from the insights presented in this work would be to apply a fast sketch with an isotropic spectrum to minimize computation time and distortion, such as the SRHT, but to be aware of issues arising from zero-padding; for this reason we suggest that other Fourier transforms be used instead of the \editedinlinethree{standard fast} Hadamard transform.}

\paragraph{Future work}

As alluded to in the introduction, 
the first- and second-order equivalences developed in this work can be used directly to analyze the asymptotics of the predicted values and quadratic errors of sketched ridge regression. 
We leave a complete detailed analysis of sketched ridge regression for a companion paper,
in which we use the results in this work to study both primal (observation-side) and dual (feature-side) sketching of the data matrix, 
as well as joint primal and dual sketching. 
We believe that our results can also be combined with the techniques in \cite{liao2021hessian} who obtain deterministic equivalents for the Hessian of generalized linear models,
enabling precise asymptotics for the implicit regularization due to sketching in nonlinear prediction models such as classification with logistic regression.

\section*{Acknowledgments}
We are grateful to Arun Kumar Kuchibhotla, Alessandro Rinaldo, Yuting Wei, Jin-Hong Du, 
and other members of the Operational Overparameterized Statistics (OOPS) Working Group 
at Carnegie Mellon University for helpful conversations.
We are also grateful to Edgar Dobriban, \editedinline{Mert Pilanci}, \editedinlinethree{Benson Au, Elad Romanov, and Dimitri Shlyakhtenko},
as well as participants of the Deep Learning ONR MURI seminar series 
for useful discussions and feedback on this work. \editedinline{We thank the anonymous reviewers for their thoughtful suggestions which have strengthened this work,
and the associate editor for the swift review process.}

This work was sponsored by Office of Naval Research MURI grant N00014-20-1-2787. 
DL, HJ, and RGB were also supported by NSF grants CCF-1911094, IIS-1838177, and IIS-1730574; ONR grants N00014-18-12571 and N00014-20-1-2534; AFOSR grant FA9550-22-1-0060; and a Vannevar Bush Faculty Fellowship, ONR grant N00014-18-1-2047.  DL was partially supported by ARO grant 2003514594.

\clearpage
\setcounter{section}{0}
\setcounter{equation}{0}
\setcounter{figure}{0}
\renewcommand{\thesection}{SM\arabic{section}}
\renewcommand{\theequation}{SM\arabic{section}.\arabic{equation}}
\renewcommand{\thefigure}{SM\arabic{figure}}

\thispagestyle{empty}
\begin{center}
\noindent\textcolor{header1}{\textbf{SUPPLEMENTARY MATERIALS: \textsf{Asymptotics of the Sketched Pseudoinverse}}}
\color{gray}\rule{\textwidth}{4pt}
\end{center}
\medskip

This document serves as a supplement to the paper
``Asymptotics of the Sketched Pseudoinverse.''
The contents of this supplement are organized as follows.
In \Cref{sec:useful-facts},
we collect some useful facts regarding Stieltjes transforms
that are used in some of the proofs in later sections.
In \Cref{sec:proof:cor:basic-ridge-asympequi-in-r},
we provide a detailed proof for \Cref{cor:basic-ridge-asympequi-in-r}.
In \Cref{sec:proof:thm:sketched-pseudoinverse},
we provide proof 
for \cref{thm:sketched-pseudoinverse}.
In \Cref{sec:proofs-properties},
we provide proofs of various properties regarding our main equivalences
mentioned \Cref{sec:properties} in the main paper.
\editedinline{In \Cref{sec:proofs:applications}, we give proofs for the application of our equivalence to the sketch-and-project method.}
\editedinlinetwo{Finally, in \Cref{sec:proofs-free}, we give proof for \Cref{thm:general-free-sketching} which extends our results to free sketching.}

\section{Useful facts}
\label{sec:useful-facts}

In this section,
we jot down basic definitions and facts
related Stieltjes transform that we will be using
throughout the paper.

Let $Q$ be a bounded nonnegative measure on $\RR$.
The Stieltjes transform of $Q$ is defined at $z \in \CC^{+}$
by
\[
    m_Q(z) = \int_{\RR} \frac{1}{x - z} \, \mathrm{d}Q(x).
\]

\begin{fact}
    \label{fact:lim_stiletjes_im}
    Let $m$ be the Stieltjes transform of bounded measure $Q$ on $\RR_{\ge 0}$.
    Let $z \in \CC^{+}$ with $\Re(z) < 0$.
    Then, $\Im(m(z)) \searrow 0$ as $\Im(z) \searrow 0$.
\end{fact}
\begin{proof}
    Let $z = x + i y$ with $x < 0$ and $y > 0$.
    Since $m$ is a Stieltjes transform of $Q$,
    we have
    \[
        \Im(m(z)) 
        = \Im \left( \int \frac{1}{r - z} \, \mathrm{d}Q(r) \right)
        = \Im \left( \int \frac{1}{r - (x + i y)} \, \mathrm{d}Q(r) \right)
        = \int \frac{y}{(r - x)^2 + y^2} \, \mathrm{d}Q(r).
    \]
    Thus,
    we can bound
    \[
        | \Im(m(z)) |
        \le \frac{y}{x^2} \int \mathrm{d}Q(r).
    \]
    Since $Q$ is a bounded measure, by letting $y \searrow 0$, one has $\Im(m(z)) \searrow 0$ 
    as $\Im(z) \searrow 0$.
\end{proof}

We will be interested in the Stieltjes transforms of spectral measures.
The spectral distribution of a symmetric matrix $\bA \in \CC^{p \times p}$
with eigenvalues $\lambda_1(\bA), \dots, \lambda_p(\bA)$
is the probability distribution that places a point mass of $\tfrac{1}{p}$ at each
eigenvalue
\[
    F_\mA(\lambda) = \tfrac{1}{p} \sum_{i=1}^{p} \1\{ \lambda_i \le \lambda \}.
\]
The matrices of interest for us will be the population covariance matrix 
$\bSigma \in \CC^{p \times p}$
and the sample covariance matrix 
$\tfrac{1}{n} \bX^\ctransp \bX$ where $\bX \in \CC^{n \times p}$ is the random design matrix.

If the Stieltjes transform 
of spectrum of the sample covariance matrix $\tfrac{1}{n} \mX^\ctransp \mX$
is 
\begin{equation}
    \label{eq:stieltjes-transform}
    m(z) = \tfrac{1}{p} \tr[(\tfrac{1}{n} \mX^\ctransp \mX - z \mI_p)^{-1}],
\end{equation}
then the so-called \emph{companion} Stieltjes transform 
\begin{equation}
    \label{eq:companion-stieltjes-transform}
    v(z) = \tfrac{1}{n} \tr[(\tfrac{1}{n} \mX \mX^\ctransp - z \mI_n)^{-1}] 
\end{equation}
is the Stieltjes transform of $\tfrac{1}{n} \mX \mX^\ctransp$
(and hence the prefix).
The reason it is useful is that
it is often easier to work with the companion Stieltjes transform
than the Stieltjes transform.
The following fact relates the companion Stieltjes transform
to the Stieltjes transform.

\begin{fact}
\label{fact:stieltjes-companion-stieltjes-relation}
The companion Stieltjes transform $v(z)$ 
can be expressed in terms of the Stieltjes transform $m(z)$ at $z \in \CC^{+}$ as
\begin{equation}
    \label{eq:stieltjes-companion-stieltjes-relation}
    v(z) = \frac{p}{n} m(z) + \frac{1}{z} \left(\frac{p}{n} - 1\right).
\end{equation}
\end{fact}
\begin{proof}
    Let $(\lambda_i)_{i = 1}^r$ be the nonzero eigenvalues of 
    $\tfrac{1}{n} \bX^\ctransp \bX$
    (which are also the nonzero eigenvalues of $\tfrac{1}{n} \bX \bX^\ctransp$).
    Define $\Lambda(z) = \sum_{i=1}^r \tfrac{1}{\lambda_i - z}$.
    From \cref{eq:stieltjes-transform,eq:companion-stieltjes-transform},
    note that we can write
    \begin{align}
        m(z) = \frac{\Lambda(z)}{p} - \frac{(p - r)}{pz}, \quad
        v(z) = \frac{\Lambda(z)}{n} - \frac{(n - r)}{nz}.
    \end{align}
Combining these equations proves the claim.
\end{proof}

\section{Proof of \Cref{cor:basic-ridge-asympequi-in-r}}

\label{sec:proof:cor:basic-ridge-asympequi-in-r}

As a preliminary that we will need later, through a standard argument,
we will first show that
$\Im(c(z)) \nearrow 0$ as $\Im(z) \searrow 0$
in \cref{eq:basic-ridge-asympequi-in-r}
for $z \in \CC^{+}$ with $\Re(z) < 0$.
To proceed, 
denote $\tfrac{1}{p} \tr\bracket{\bSigma (c(z) \bSigma - z \bI_p)^{-1}}$ by $d(z)$.
From the last part of \cref{lem:basic-ridge-asympequi},
$d(z)$ is a Stieltjes transform of a certain positive measure on $\RR_{\ge 0}$
with total mass $\tfrac{1}{p} \tr[\bSigma]$.
Since the operator norm of $\bSigma$ is uniformly bounded in $p$,
we have that $\tfrac{1}{p} \tr[\bSigma]$ is bounded above by some constant independent of $p$.
Combining this with \cref{fact:lim_stiletjes_im},
we have that $\Im(d(z)) \searrow 0$ as $\Im(z) \searrow 0$
for $z \in \CC^{+}$ with $\Re(z) < 0$.
Now manipulating \cref{eq:basic-ridge-fp-in-c},
we can write
\begin{equation}
    \label{eq:c-in-d}
    c(z)
    = \frac{1}{1 + \tfrac{p}{n} d(z)}.
\end{equation}
Thus,
we can conclude that $\Im(c(z)^{-1}) \searrow 0$ as $\Im(z) \searrow 0$
for $z \in \CC^{+}$ with $\Re(z) < 0$.
This in turn implies that $\Im(c(z)) \nearrow 0$
for $z \in \CC^{+}$ with $\Re(z) < 0$.

We now begin the proof.

\begin{proof}
We start by considering $z \in \complexset^+$.
To obtain \cref{eq:basic-ridge-asympequi-in-r}, we multiply both sides of \cref{eq:basic-ridge-asympequi-in-c} by $z$:
\begin{align}
    z \big( \tfrac{1}{n} \bX^\ctransp \bX - z \bI_p \big)^{-1}
    &\asympequi z \inv{c(z) \bSigma -z \bI_p} \\
    &= \tfrac{z}{c(z)} \biginv{\mSigma - \tfrac{z}{c(z)} \mI_p} \label{eq:resolvent-symmetric-pre}.
\end{align}
We will let $\zeta = \tfrac{z}{c(z)}$ shortly. 
First let 
$m(z) = \tfrac{1}{p} \tr \bigbracket{\inv{c(z) \mSigma - z \mI_p}}$.
By an additional application of \cref{lem:basic-ridge-asympequi}, 
$m(z)$ is asymptotically equal to $\tfrac{1}{p} \tr \bigbracket{\inv{\tfrac{1}{n} \mX^\ctransp \mX - z \bI}}$, the Stieltjes transform of the spectrum of $\tfrac{1}{n} \mX^\ctransp \mX$. 
Now note that we can write \cref{eq:basic-ridge-fp-in-c} in terms of $m(z)$ as
\[
    \tfrac{1}{c(z)} - 1 = \tfrac{p}{n} m(z).
\]
We can manipulate the equation in the display above into the following form:
\begin{align}
   - \frac{c(z)}{z} = \tfrac{p}{n}  m(z) + \frac{1}{z} \left(\tfrac{p}{n} - 1\right).
\end{align}
From the relationship between Stieltjes and the companion Stieltjes transforms in 
\cref{fact:stieltjes-companion-stieltjes-relation},
this means that $-\tfrac{c(z)}{z}$ is asymptotically equal to $v(z) = \tfrac{1}{n} \tr \bracket{\inv{\tfrac{1}{n} \mX \mX^\ctransp - z \bI}}$, the companion Stieltjes transform
of the spectrum of $\tfrac{1}{n} \bX^\ctransp \bX$.
Thus, letting $\zeta = \tfrac{z}{c(z)}$ in \cref{eq:resolvent-symmetric-pre}, 
we have that
\[
    z \big( \tfrac{1}{n} \bX^\ctransp \bX - z \bI_p \big)^{-1} 
    \asympequi \zeta (\bSigma - \zeta \bI_p)^{-1},
\]
and that asymptotically,
$\zeta = -\tfrac{1}{v(z)}$ is the unique solution in $\complexset^+$ to \cref{eq:basic-ridge-fp-in-r} for $z \in \complexset^+$.
Moreover, through analytic continuation,
one can extend this relationship to the real line
outside the support of the spectrum of $\tfrac{1}{n} \bX \bX^\ctransp$
where by the similar argument as for $c(z)$ above,
both $v(z)$ and $\zeta$ are real.

It remains to determine the interval for which the analytic continuation coincides with a unique solution to \cref{eq:basic-ridge-fp-in-r} for a given $z$. 
Let $z_0 \in \reals$ denote the most negative zero of $v$. Then for all $z < z_0$, $\zeta \in \reals$ is well-defined, asymptotically being a solution to
\begin{align}
    \label{eq:proof:zeta-fp-in-r}
    z - \zeta = -\zeta \tfrac{1}{n} \tr \bracket{\mSigma \inv{\mSigma - \zeta \mI_p}},
\end{align}
which is an algebraic manipulation of \cref{eq:basic-ridge-fp-in-c}. However, as we will now show, the solution to this equation is not in general unique, so we will show that the most negative solution for $\zeta$ is the correct analytic continuation of the corresponding solution in $\complexset^+$.

\begin{figure}[t]
    \centering
    \includegraphics[width=\linewidth]{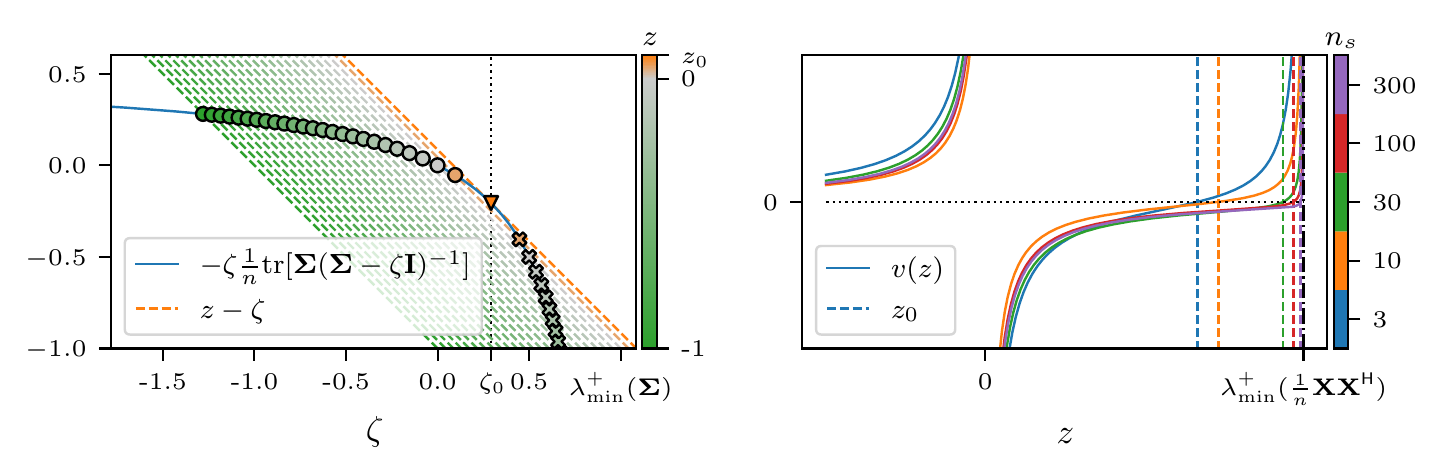}
    \caption{\textbf{Left:} Numerical illustration of the solutions to \cref{eq:proof:zeta-fp-in-r} for $\mSigma = \mI$ and $\tfrac{p}{n} = \tfrac{1}{2}$. 
    The right-hand side of \cref{eq:proof:zeta-fp-in-r} is a fixed function of $\zeta$ (blue, solid), 
    but the left-hand-side is a line with slope $-1$ shifted by $z$ (orange to green, dashed). 
    Solutions are the most negative intersections of the curves (circles), and not the most positive intersections (x's). The greatest possible value of $z$ yielding an intersection, $z = z_0$ (triangle), gives $\zeta = \zeta_0$ (dotted). For this example, we know that $z_0 = (1 - \sqrt{\tfrac{p}{n}})^2 \approx 0.0858$ since the spectrum of $\tfrac{1}{n} \mX \mX^\ctransp$ follows the Marchenko--Pastur distribution.
    \textbf{Right:} Illustration of the convergence of $z_0$ to $\lambdaminnz(\tfrac{1}{n} \mX \mX^\ctransp)$. For $\mSigma = \mI$, $p = 500$, $n = 1000$, we draw a random $\tfrac{1}{n} \mX \mX^\ctransp$ and compute its eigenvalues. To simulate increasing the dimensionality of the matrix while keeping $\lambdaminnz(\tfrac{1}{n} \mX \mX^\ctransp)$ fixed, we then take a subsample of size $n_s$ of the eigenvalues, comprised of $\lambdaminnz(\tfrac{1}{n} \mX \mX^\ctransp)$ and $n_s - 1$ other eigenvalues chosen uniformly at random. We then plot $v(z)$ (solid) using this subsample. For any finite $n_s$, $z_0$ (dashed) will always lie between $0$ and $\lambdaminnz(\tfrac{1}{n} \mX \mX^\ctransp)$, but $z_0$ approaches $\lambdaminnz(\tfrac{1}{n} \mX \mX^\ctransp)$ as $n_s$ tends to infinity.
    }
    \label{fig:zeta_z_proof_illustration}
\end{figure}

Consider the two sides of \cref{eq:proof:zeta-fp-in-r}. The left-hand side is linear in $\zeta$, and the right-hand side is concave for $\zeta < \lambdaminnz(\mSigma)$. To see this, observe that
\begin{align}
    \tfrac{\partial^2}{\partial \zeta^2} \paren{-\zeta \tfrac{1}{n} \tr \bracket{\mSigma \inv{\mSigma - \zeta \mI_p}}}
    &= \tfrac{\partial}{\partial \zeta} \paren{- \tfrac{1}{n} \tr \bracket{\mSigma \inv{\mSigma - \zeta \mI_p}} - \zeta \tfrac{1}{n} \tr \bracket{\mSigma \paren{\mSigma - \zeta \mI_p}^{-2}}} \\
    &= \tfrac{\partial}{\partial \zeta} \paren{-  \tfrac{1}{n} \tr \bracket{\mSigma^2 \paren{\mSigma - \zeta \mI_p}^{-2}}} \\
    &= -\tfrac{2}{n} \tr \bracket{\mSigma^2 \paren{\mSigma - \zeta \mI_p}^{-3}} < 0.
\end{align}
A linear function and a concave function can intersect at zero, one, or two points. If at one point, this must occur at the unique point $(z_1, \zeta_1)$, $\zeta_1 < \lambdaminnz(\mSigma)$ for which the derivatives of each side of \cref{eq:proof:zeta-fp-in-r} coincide, satisfying
\begin{align}
    1 = \tfrac{1}{n} \tr \bracket{ \mSigma^2 \paren{\mSigma - \zeta_1 \mI_p}^{-2} }.
\end{align}
This right-hand side of this equation sweeps the range $(0, \infty)$ for $\zeta_1 \in (-\infty, \lambdaminnz(\mSigma))$, so such a $(z_1, \zeta_1)$ always exists.
Furthermore, since the solutions $\zeta$ are continuous as a function $z$, the analytic continuation of the complex solution to the reals of the map $z \mapsto \zeta$ with domain $(-\infty, z_1)$ must have image of either $(-\infty, \zeta_1)$ or $(\zeta_1, \lambdaminnz(\mSigma))$.
The correct image must be $(-\infty, \zeta_1)$, which we illustrate in \cref{fig:zeta_z_proof_illustration} (left). 

To see why this must be the correct image, consider $z = x + i \varepsilon$ for a fixed $\varepsilon > 0$ with $x$ very negative.
Rewriting \cref{eq:proof:zeta-fp-in-r}, we have the form of \cref{eq:basic-ridge-fp-in-r}:
\begin{align}
    \label{eq:proof:zeta-fp-all-on-right-hand-side}
    z = \zeta \paren{1 - \tfrac{1}{n} \tr \bracket{\mSigma \inv{\mSigma - \zeta \mI_p}}}.
\end{align}
We begin by considering the behavior of the trace term.
Let $\zeta = \chi + i \xi$, and suppose that $\chi < \tfrac{x}{2}$, which means that $\chi$ is also very negative. The trace is a sum of terms of the form
\begin{align}
    \frac{\sigma}{\sigma - \zeta} = \frac{\sigma (\sigma - \chi + i \xi)}{(\sigma - \chi)^2 + \xi^2}.
\end{align}
Let $g(\zeta)$ and $h(\zeta)$ denote the real and imaginary parts of $\tfrac{1}{n} \tr \bracket{\mSigma \inv{\mSigma - \zeta \mI_p}}$.
For $x$ (and therefore $\chi$) sufficiently negative, this gives us the simple bounds
\begin{align}
    \left| g(\zeta) \right| 
    &\leq \frac{\tfrac{p}{n} \sigma_{\max}(\mSigma)}{- \chi} 
    \leq \frac{2 \tfrac{p}{n} \sigma_{\max}(\mSigma)}{- x}, \\
    \left| h(\zeta) \right| 
    &\leq \frac{\tfrac{p}{n} \sigma_{\max}(\mSigma) \xi}{\chi^2}
    \leq \frac{4 \tfrac{p}{n} \sigma_{\max}(\mSigma) \xi}{x^2}.
\end{align}
We therefore have by \cref{eq:proof:zeta-fp-all-on-right-hand-side} that
\begin{align}
    \chi = \frac{x (1 - g(\zeta)) + \varepsilon h(\zeta)}{(1 - g(\zeta))^2 + h(\zeta)^2}, \quad
    \xi = \frac{\varepsilon (1 - g(\zeta)) - x h(\zeta)}{(1 - g(\zeta))^2 + h(\zeta)^2}.
\end{align}
By our bounds on $g$ and $h$, we can conclude that for sufficiently negative $x$, there exists $a > 0$ and $0 < b < 1$ such that $ |\xi| \leq a \varepsilon + b |\xi| $, implying that $|\xi| \leq \tfrac{a \varepsilon}{1 - b}$, and therefore $|\xi|$ is bounded. Since $|\xi|$ is bounded, $|h(\zeta)|$ has an upper bound of the form $\tfrac{1}{x^2}$, so for any $c \in (\tfrac{1}{2}, 1)$ and sufficiently negative $x$, we have the bound $\chi \leq c x$. Therefore, we can confirm that our supposition that $\chi < \tfrac{x}{2}$ leads to the unique solution with $\xi > 0$, since for any $c' \in (0, 1)$ we similarly have $\xi > c' \varepsilon > 0$ for sufficiently negative $x$. One can similarly argue that for solutions with $\chi \nearrow \lambdaminnz(\mSigma)$, it must be that $\xi < 0$, which is the solution in the wrong half-plane. By continuity of $z \mapsto \zeta$, identifying these extreme cases is sufficient to identify the correct image.
Therefore, for real-valued $z < z_1$, the correct $\zeta$ is the most negative solution, which is the unique $\zeta < \zeta_1$, and $\zeta$ is undefined for $z > z_1$.

Lastly, we argue that asymptotically, $z_0 = z_1$. In the case $n < p$, this is straightforward, as the most negative zero of $v$ must lie between the two most negative distinct eigenvalues of $\frac{1}{n} \mX \mX^\ctransp$. This is because there is a pole at each distinct eigenvalue, so the entire range $(-\infty, \infty)$ (including crossing $0$) is mapped to by $v$ between each successive pair of distinct eigenvalues. When $n < p$, there is not a point mass at $0$, so these two most negative eigenvalues must converge to the same value as the discrete eigenvalue distribution converges to a continuous distribution, and this value marks the beginning of the continuous support of the spectrum of $\frac{1}{n} \mX \mX^\ctransp$, so $z_0 \to \lambda_{\min}(\frac{1}{n} \mX \mX^\ctransp)$. Moreover, $\zeta$, being asymptotically equal to $\tfrac{1}{v}$, is undefined only on the support of the limiting spectrum and continuous elsewhere; therefore by the argument in the previous paragraph, the solution to \cref{eq:proof:zeta-fp-in-r} does not exist for $z > z_1$, and it must be that $\lambda_{\min}(\frac{1}{n} \mX \mX^\ctransp) \to z_1$.

For $n > p$, we apply similar reasoning; however, we must take care to consider the point mass of the spectrum at 0. This means that $z_0 \in (0, \lambdaminnz(\frac{1}{n} \mX \mX^\ctransp))$, because like before, the first zero must lie between the two most negative distinct eigenvalues, as we illustrate in \cref{fig:zeta_z_proof_illustration} (right).
However, asymptotically, it must be that $z_0 \nearrow z_1 = \lambdaminnz(\frac{1}{n} \mX \mX^\ctransp)$. 
This is most easily seen by a contradiction argument. Suppose
we have $z_0 < z_1 - \varepsilon$ for some $\varepsilon > 0$.
Because $\tfrac{1}{v}$ has a pole at $z_0$, $-\zeta = \tfrac{1}{v} \nearrow \infty$ as $z \searrow z_0$. In particular, this means that $\tfrac{1}{v}$ is discontinuous at $z_0$, tending to $\infty$ from the right. 
Meanwhile, as argued above, $\lambdaminnz(\frac{1}{n} \mX \mX^\ctransp) \to z_1$, and we know that for $z < z_1$, $\zeta < \zeta_1 \in (-\infty, \lambdaminnz(\mSigma))$.
This is a contradiction, because on the one hand $\zeta$ is upper bounded by $\lambdaminnz(\mSigma)$ for any $z \in (z_0, z_1)$, but on the other hand $\tfrac{1}{v}$ can be made arbitrarily large by taking $z \searrow z_0$.
Therefore, we must have, asymptotically, that $z_0 = z_1 = \lambdaminnz(\frac{1}{n} \mX \mX^\ctransp)$. For this reason, in the theorem statement, we denote $\zeta_0 = \zeta_1$.
\end{proof}

\section[Proof of \cref{thm:sketched-pseudoinverse} for positive semidefinite quadratic form matrix]{Proof of \cref{thm:sketched-pseudoinverse} for positive semidefinite $\mA$}
\label{sec:proof:thm:sketched-pseudoinverse}

\begin{proof}
We begin by proving the equivalence \cref{eq:thm:sketched-pseudoinverse} and then show that the limit as $\lambda \to 0$ is well-behaved when we multiply by $\mA^{1/2}$ to obtain \cref{eq:thm:sketched-pseudoinverse-A-half}.

Let $\mA_\delta \defeq \mA + \delta \mI_p$, $\mU \defeq \mS \biginv{\mS^\ctransp \mA \mS + \lambda \mI_q} \mS^\ctransp$, and $\mV \defeq \inv{\mA + \mu \mI_p}$. By the Woodbury matrix identity, we have the following two identities:
\begin{gather}
    \mS \biginv{\mS^\ctransp \mA_\delta \mS + \lambda \mI_q} \mS^\ctransp
    = \mU - \delta \mU \inv{\mI_p + \delta \mU} \mU, \\
    \inv{\mA_\delta + \lambda \mI_p}
    = \mV - \delta \mV \inv{\mI_p + \delta \mV} \mV.
\end{gather}
If either $\lambda \neq 0$ or $\limsup \tfrac{q}{p} < \liminf r(\mA)$, then we can conclude that (see, e.g., \cite{bai_silverstein_1998}) that $\bignorm[\op]{\biginv{\mS^\ctransp \mA_\delta \mS + \lambda \mI_q}}$ is almost surely uniformly bounded and that $\mu$ is bounded away from zero (see \cref{rem:joint-signs-lambda-mu}). Thus, since $\norm[\op]{\mS}$ is also almost surely bounded asymptotically, $\norm[\op]{\mU}$ and $\norm[\op]{\mV}$ are asymptotically bounded by constants $C_\mU$ and $C_\mV$, respectively. Therefore, for $\delta < \tfrac{1}{2}\min\set{C_\mU, C_\mV}$, we have the following bound on the trace functional difference:
\begin{multline}
    \limsup \big| \tr \bigbracket{\mTheta \bigparen{\mS \biginv{\mS^\ctransp \mA_\delta \mS + \lambda \mI_q} \mS^\ctransp - \inv{\mA_\delta + \lambda \mI_p}}}
    - \tr \bigbracket{\mTheta \bigparen{\mU - \mV}}
    \big| \\
    \leq \tfrac{\delta}{2} \norm[\tr]{\mTheta} \bigparen{C_\mU^2 + C_\mV^2}.
\end{multline}
Thus, as $\delta \searrow 0$, the trace functionals converge uniformly over $p$ for $\mTheta$ with uniformly bounded trace norm. We can therefore apply the Moore--Osgood Theorem to interchange limits, such that almost surely
\begin{align}
    \lim_{p \nearrow \infty} \big| \tr \bigbracket{\mTheta \bigparen{\mU - \mV}}
    \big| &= 
    \lim_{\delta \searrow 0} \lim_{p \nearrow \infty} 
    \big|
    \tr \bigbracket{\mTheta \bigparen{\mS \biginv{\mS^\ctransp \mA_\delta \mS + \lambda \mI_q} \mS^\ctransp - \inv{\mA_\delta + \lambda \mI_p}}}
    \big| \\
    &= 0.
\end{align}

To prove the equivalence in \cref{eq:thm:sketched-pseudoinverse-A-half}, we can apply the equivalence in \cref{eq:thm:sketched-pseudoinverse} proved above unless $\lambda = 0$ and $\limsup \tfrac{q}{p} \geq \liminf r(\mA)$. We need only consider $\limsup \lambda_0 < 0$, so it suffices to consider $\liminf \tfrac{q}{p} > \limsup r(\mA)$ (see \cref{rem:mu0-lambda0-vs-alpha}).
The condition $\limsup \lambda_0 < 0$ implies that there exists $c_\lambda > 0$ such that $\lambdaminnz(\mS^\ctransp \mA \mS) > c_\lambda$. Therefore, $\bignorm[\op]{\mA^{1/2} \mS \biginv{\mS^\ctransp \mA \mS + \lambda \mI_q}}$ is almost surely uniformly bounded in $p$ for all $\lambda \in D_\lambda$, where $D_\lambda = \bigset{z \in \complexset \colon |z| < \tfrac{c_\lambda}{2}}$. We now need to bound $\bignorm[\op]{\mA^{1/2} \inv{\mA + \mu \mI_p}}$. 
From the definition of $\mu_0$ in \cref{eq:mu0-lambda0-fps},
we observe that
\begin{align}
    \frac{p}{q}
    \frac{r(\mA) \lambda_{\max}(\mA)^2}{ (\lambda_{\max}(\mA) + \mu_0)^2} 
    \leq 1 \leq 
    \frac{p}{q}
    \frac{r(\mA) \lambdaminnz(\mA)^2}{ (\lambdaminnz(\mA) + \mu_0)^2},
\end{align}
from which we can conclude for the case that $\tfrac{q}{p} > r(\mA)$ and $\lambdaminnz(\mA) > 0$,
we can bound
\begin{equation}
    \label{eq:mu-0-upper-lower-bounds}
    \paren{\tfrac{p r(\mA)}{q} - 1} \lambda_{\max}(\mA) < 
    \paren{\sqrt{\tfrac{p r(\mA)}{q}} - 1} \lambda_{\max}(\mA)
    \leq \mu_0 \leq 
    \paren{\sqrt{\tfrac{p r(\mA)}{q}} - 1} \lambdaminnz(\mA) < 0.
\end{equation}
Since $\liminf \tfrac{q}{p} > \limsup r(\mA)$ and $\liminf \lambdaminnz(\mA) > 0$, we therefore must have $\limsup \mu_0 < 0$.
Define the set $D_{\mu} = \bigset{z \in \complexset \colon |z| < \tfrac{-\limsup \mu_0}{2}}$. Since $-\liminf \lambdaminnz(\mA) \leq \mu_0$, for all $\mu \in D_{\mu}$, we must have the bound
\begin{equation}
    \bignorm[\op]{\mA^{1/2} \inv{\mA + \mu \mI_p}} \leq \frac{2\bignorm[\op]{\mA^{1/2}}}{-\limsup \mu_0}.
\end{equation}
We also know from \cref{eq:sketched-modified-lambda} that
\begin{equation}
    |\lambda| = |\mu| \big| 1 - \tfrac{1}{q} \tr \bigbracket{\mA \inv{\mA + \mu \mI_p}} \big|.
\end{equation}
One can confirm that the second factor on the right-hand side is uniformly lower bounded away from 0 for $\mu \in D_\mu$
using the first bound in \cref{eq:mu-0-upper-lower-bounds}. Let $D_p = \set{\lambda : \mu(\lambda) \in D_\mu}$ be the inverse image of $D_\mu$ under the map $\lambda \mapsto \mu$ for each $p$. By the above arguments, the set $D = D_\lambda \cap \limsup D_p$ is an open set over which the functions 
\begin{equation}
    f_p(\lambda) = \big|\tr \bigbracket{\mTheta \bigparen{\mA^{1/2} \mS \big( \mS^\ctransp \mA \mS + \lambda \mI_q \big)^{-1} \mS^\ctransp - \mA^{1/2} (\mA + \mu \mI_q )^{-1}} } \big|
\end{equation}
converge uniformly as $\lambda \to 0$ over $p$. By Montel's theorem, these functions form a normal family. Since $f_p(\lambda) \searrow 0$ pointwise for $\lambda \neq 0$, this implies that $f_p(0) \searrow 0$.
\end{proof}

\section{Proofs in \Cref{sec:properties}}
\label{sec:proofs-properties}

We collect the proofs of the various properties of the equivalences obtained in our paper.

\subsection{Proof of \cref{rem:mu0-lambda0-vs-alpha}}
\label{sec:mu0-lambda0-vs-alpha}

\begin{proof}
Recall from \cref{eq:lambda0-mu0-in-alpha} that
for $\alpha \in (0, \infty)$,
\begin{equation}
    \label{eq:lambda0-with-mu0}
    \lambda_0(\alpha)
    = \mu_0
    \paren{ 1 - \tfrac{1}{\alpha} \tfrac{1}{p} 
    \tr \bracket{ \bA (\bA + \mu_0(\alpha) \bI)^{-1} } }.
\end{equation}
From the statement of \cref{rem:mu0-lambda0-vs-alpha},
$\lim_{\alpha \nearrow \infty} \mu_0(\alpha) = - \lambdaminnz(\bA)$.
We will argue below that
\begin{equation}
    \label{eq:mu0-by-alpha-lim}
    \lim_{\alpha \nearrow \infty}
    \frac
    {\tfrac{1}{p} \tr \bracket{ \bA (\bA + \mu_0(\alpha) \bI)^{-1} }}
    {\alpha}
    = 0,
\end{equation}
which combined with \cref{eq:lambda0-with-mu0} provides the desired result.

Observe that the limit on the left-hand side of \cref{eq:mu0-by-alpha-lim} 
is in the indeterminate $\infty/\infty$ form because
$\lim_{\alpha \nearrow \infty} \mu(\alpha) = - \lambdaminnz(\bA)$
and thus
$\lim_{\alpha \nearrow \infty} 
\tfrac{1}{p} \tr \bracket{ \bA (\bA + \mu(\alpha) \bI)^{-1} } = \infty$.
To evaluate the limit, we will appeal to L'H{\^o}pital's rule.
The derivative of the denominator with respect to $\alpha$ is 1,
while the derivative of the numerator with respect to $\alpha$ is
\begin{equation}
    \label{eq:deriv-numer-wrt-alpha}
    \tfrac{1}{p}
    \tr \bracket{ \bA (\bA + \mu_0(\alpha) \bI)^{-2} }
    \frac{\partial \mu_0(\alpha)}{\partial \alpha}.
\end{equation}
Implicitly differentiating \cref{eq:mu0-fp-in-alpha}
with respect to $\alpha$, we have
\begin{equation}
    \label{eq:deriv-mu0-wrt-alpha-relation}
    1 
    = \tfrac{1}{p}
    \tr \bracket{ \bA^2 (\bA + \mu_0(\alpha) \bI)^{-3} }
    \frac{\partial \mu_0(\alpha)}{\partial \alpha}.
\end{equation}
Substituting for $\tfrac{\partial \mu_0(\alpha)}{\partial \alpha}$
from \cref{eq:deriv-mu0-wrt-alpha-relation}
into \cref{eq:deriv-numer-wrt-alpha},
we can write the derivative of the numerator as
\begin{align}
    \frac
    {
    \tfrac{1}{p}
    \tr \bracket{ \bA (\bA + \mu_0(\alpha) \bI)^{-2} }
    }
    {
    \tfrac{1}{p}
    \tr \bracket{ \bA^2 (\bA + \mu_0(\alpha) \bI)^{-3}}
    }.
\end{align}
As $\alpha \nearrow \infty$ and $\mu_0(\alpha) \searrow - \lambdaminnz(\bA)$,
the limit of the quantity in the display above becomes
\begin{align}
    \lim_{\alpha \nearrow \infty}
    \frac
    {\lambdaminnz(\bA)}
    {(\lambdaminnz(\bA) + \mu_0(\alpha))^2}
    \cdot
    \frac
    {(\lambdaminnz(\bA) + \mu(\alpha))^3}{(\lambdaminnz(\bA))^2}
    = 
    \lim_{\alpha \nearrow \infty}
    1 + \frac{\mu_0(\alpha)}{\lambdaminnz(\mA)}
    = 1 - 1
    = 0.
\end{align}
Thus, we can conclude that \cref{eq:mu0-by-alpha-lim} holds,
and the statement then follows.
The remaining claims follow by similar calculations.
\end{proof}

\subsection{Proof of \cref{rem:mu0-lambda0-signs}}
\label{sec:mu0-lambda0-signs}

\begin{proof}
We start by noting that 
\begin{align}
    \lim_{x \searrow 0}
    \tfrac{1}{p}
    \tr \bracket{ \bA^2 (\bA + x \bI)^{-2} }
    &=
    \lim_{x \searrow 0}
    \tfrac{1}{p}
    \sum_{i = 1}^{p}
    \frac{\lambda_i^2(\bA)}{(\lambda_i(\bA) + x)^2}
     \\
    &=
    \lim_{x \searrow 0}
    \tfrac{1}{p}
    \sum_{i = 1}^{p}
    \frac{\lambda_i(\bA)}{\lambda_i(\bA) + x}
    =
    \lim_{x \searrow 0}
    \tfrac{1}{p}
    \tr \bracket{ \bA (\bA + x \bI)^{-1} } \\
    &=
    \tfrac{1}{p} \sum_{i=1}^{p}
    \1\{ \lambda_i(\bA) > 0 \} 
    = r(\bA).
\end{align}
Now, write the first equation in \cref{eq:mu0-lambda0-fps} 
in terms of $\alpha$ as
\[
    \alpha = \tfrac{1}{p} \tr \bracket{\mA^2 \paren{\mA + \mu_0 \mI}^{-2}}.
\]
Thus, when $\alpha = r(\bA)$, we have $\mu_0 = 0$
as the solution to the first equation of \cref{eq:mu0-lambda0-fps}.
Because $\mu \mapsto \tfrac{1}{p} \tr[\bA^2 (\bA + \mu \bI)^{-2}]$
is monotonically decreasing in $\mu$,
if $\alpha < r(\bA)$, we have $\mu_0 > 0$,
while if $\alpha > r(\bA)$, we have $\mu_0 < 0$.

Next we argue about sign pattern of $\lambda_0$.
When $\alpha > r(\bA)$,
we have
\begin{align}
    \alpha
    = \tfrac{1}{p}
    \tr \bracket{ \bA^2 (\bA + \mu_0 \bI)^{-2} }
    = \tfrac{1}{p}
    \sum_{i = 1}^{p}
    \frac{\lambda_i^2(\bA)}{(\lambda_i(\bA) + \mu_0)^2}
    &\overset{(a)}{>}
    \tfrac{1}{p}
    \sum_{i = 1}^{p}
    \frac{\lambda_i(\bA)}{\lambda_i(\bA) + \mu_0} \\
    &=
    \tfrac{1}{p}
    \tr \bracket{ \bA (\bA + \mu_0 \bI)^{-1} },
\end{align}
where the inequality $(a)$ follows because $\mu_0 < 0$.
From \cref{eq:mu0-lambda0-fps},
it thus follows that $\lambda_0 < 0$.
Similarly, when $\alpha < r(\bA)$,
note that
\begin{align}
    \alpha
    = \tfrac{1}{p}
    \tr \bracket{ \bA^2 (\bA + \mu_0 \bI)^{-2} }
    = \tfrac{1}{p}
    \sum_{i = 1}^{p}
    \frac{\lambda_i^2(\bA)}{(\lambda_i(\bA) + \mu_0)^2}
    &\overset{(b)}{<}
    \tfrac{1}{p}
    \sum_{i = 1}^{p}
    \frac{\lambda_i(\bA)}{\lambda_i(\bA) + \mu_0} \\
    &=
    \tfrac{1}{p}
    \tr \bracket{ \bA (\bA + \mu_0 \bI)^{-1} },
\end{align}
where inequality $(b)$ follows from the fact that
\[
    0
    <
    \left(\frac{\lambdaminnz(\bA)}{\lambdaminnz(\bA) + \mu_0}\right)^2
    <
    \frac{\lambdaminnz(\bA)}{\lambdaminnz(\bA) + \mu_0}
    < 1,
\]
since $\mu_0 > 0$ in this case
and $\lambdaminnz(\bA) > 0$.
From \cref{eq:mu0-lambda0-fps},
it thus again follows that $\lambda_0 < 0$.
This completes the proof.
\end{proof}

\subsection{Proof of \cref{prop:monotonicities-lambda-alpha}}
\label{sec:monotonicies-lambda-alpha}

\begin{proof}
The claims follow from simple derivative calculations.
We split into two cases, one with respect to $\lambda$,
and the other with respect to $\alpha$.

\subsubsection[Monotonicity with respect to ridge regularization parameter]{Monotonicity with respect to $\lambda$}

For a fixed $\alpha$,
implicitly differentiating the fixed-point equation \cref{eq:sketched-modified-lambda}
with respect to $\lambda$, we obtain
\begin{equation}
    \label{eq:fp-differentiation-lambda}
    1 =
    \frac{\partial \mu}{\partial \lambda}
    - \left( \tfrac{1}{q} \tr \bracket{\bA \inv{\bA + \mu \bI} } 
     - \mu \tfrac{1}{q}\tr \bracket{ \bA \paren{\bA + \mu \bI}^{-2} }
    \right) 
    \frac{\partial \mu}{\partial \lambda}.
\end{equation}
Note the following algebraic simplification:
\begin{align}
    \bA \paren{ \bA + \mu \bI }^{-1}
    - \mu \bA \paren{ \bA + \mu \bI }^{-2}
    &= \bA \paren{ \bA + \mu \bI }^{-1}
    \paren{ \bI - \mu \paren{ \bA + \mu \bI }^{-1} } \nonumber \\
    &= \bA \paren{ \bA + \mu \bI }^{-1} \bA \paren{ \bA + \mu \bI }^{-1}
    = \bA^2 \paren{ \bA + \mu \bI }^{-2} \label{eq:matrix-diff-manip}.
\end{align}
Substituting \cref{eq:matrix-diff-manip} into \cref{eq:fp-differentiation-lambda},
we have
\begin{equation}
    \label{eq:mu-deriv-lambda}
    \frac{\partial \mu}{\partial \lambda}
    = \frac
    {1}
    {
    1 - \tfrac{1}{q} \tr \bracket{ \bA^2 \paren{ \bA + \mu \bI }^{-2} }}.
\end{equation}
Observe that $\mu \mapsto \tfrac{1}{q} \tr[\bA^2 (\bA + \mu \bI)^{-2}]$
is monotonically decreasing function of $\mu$
over $(\mu_0, \infty)$
and because $1 = \tfrac{1}{q} \tr[\bA^2 (\bA + \mu_0 \bI)^{-2}]$
from the first equation in \cref{eq:mu0-lambda0-fps},
the denominator of \cref{eq:mu-deriv-lambda} is positive over $(\mu_0, \infty)$.
Consequently, $\tfrac{\partial \mu}{\partial \lambda}$ is positive,
and $\mu$ is a monotonically increasing function of $\lambda$.
Finally, note that as $\lambda \searrow \lambda_0$, $\mu(\lambda) \searrow \mu_0$,
and as $\lambda \nearrow \infty$, $\mu(\lambda) \nearrow \infty$.
This completes the proof of the first part.

\subsubsection[Monotonicity with respect to sketch aspect ratio]{Monotonicity with respect to $\alpha$}

We begin by writing \cref{eq:sketched-modified-lambda} 
in $\alpha$ as
\begin{equation}
    \label{eq:sketched-modified-lambda-in-alpha}
    \lambda
     = \mu
     \paren{ 1 - \tfrac{1}{\alpha} \tfrac{1}{p} \tr \bracket{ \bA (\bA + \mu \bI)^{-1} } }.
\end{equation}
For a fixed $\lambda$,
implicitly differentiating \cref{eq:sketched-modified-lambda}
with respect to $\alpha$, we have
\begin{equation}
    \label{eq:fp-differentiation-alpha}
    0
    =
    \frac{\partial \mu}{\partial \alpha}
    +
    \frac{\mu}{\alpha^2}
    \tfrac{1}{p}
    \tr \bracket{ \bA (\bA + \mu \bI)^{-1} }
    -
    \frac{1}{\alpha}
    \left(
        \tfrac{1}{p}
        \tr \bracket{ \bA (\bA + \mu \bI)^{-1} }
        - \mu \tfrac{1}{p}
        \tr \bracket{ \bA (\bA + \mu \bI)^{-2} }
    \right)
    \frac{\partial \mu}{\partial \alpha}.
\end{equation}
Solving for $\tfrac{\partial \mu}{\partial \alpha}$,
we obtain
\begin{equation}
    \label{eq:mu-deriv-alpha-1}
    \frac{\partial \mu}{\partial \alpha}
    = 
    \frac
    {- \tfrac{1}{\alpha^2} \mu \tfrac{1}{p}
    \tr \bracket{ \bA (\bA + \mu \bI)^{-1}} }
    {1 - \tfrac{1}{\alpha} \tfrac{1}{p}
    \left( \tr \bracket{ \bA (\bA + \mu \bI)^{-1} 
    - \mu \bA (\bA + \mu \bI)^{-2} } \right)}.
\end{equation}
Similar to the part above,
substituting the relation \cref{eq:matrix-diff-manip} into \cref{eq:fp-differentiation-alpha}
and simplifying yields
\begin{equation}
    \label{eq:mu-deriv-alpha-2}
    \frac{\partial \mu}{\partial \alpha}
    = 
    \frac
    {- \tfrac{1}{\alpha} \mu \tfrac{1}{q}
    \tr \bracket{ \bA (\bA + \mu \bI)^{-1}} }
    {1 - \tfrac{1}{q} \tr \bracket{ \bA^2 \paren{ \bA + \mu \bI }^{-2} }}.
\end{equation}

Because the denominator of \cref{eq:mu-deriv-alpha-2} is positive
from \cref{eq:mu0-lambda0-fps}
as argued above
and $\tr[\bA (\bA + \mu \bI)^{-1}]$ is positive
for $\mu \in (\mu_0, \infty)$,
the sign of $\frac{\partial \mu}{\partial \alpha}$
is opposite the sign of $\mu$.
Because when $\lambda \ge 0$, $\mu \ge 0$
(from the first part of \cref{rem:joint-signs-lambda-mu}),
in this case, $\frac{\partial \mu}{\partial \alpha}$
is negative, and $\mu$ is monotonically decreasing in $\alpha$.
When $\lambda < 0$,
for $\alpha \le r(\bA)$,
we have $\mu(\lambda) \ge 0$ 
(from the second part of \cref{rem:joint-signs-lambda-mu}).
Thus, over $(0, r(\bA))$, $\mu$ is monotonically decreasing in $\alpha$.
On the other hand, for $\alpha > r(\bA)$,
$\mu(\lambda) < 0$ 
(since $\mathrm{sign}(\mu(\lambda)) 
= \mathrm{sign}(\lambda)$ and $\lambda < 0$),
and consequently, $\mu$ is monotonically increasing in $\alpha$
over $(r(\bA), \infty)$.

Finally, to obtain the limit of $\mu(\alpha)$
as $\alpha \searrow 0$,
we write \cref{eq:sketched-modified-lambda-in-alpha} as
\[
    \lambda \alpha
    = \mu \alpha
    - \mu \tfrac{1}{p}
    \tr \bracket{ \bA (\bA + \mu \bI)^{-1} }.
\]
Now, for any $\lambda \in (\lambda_0, \infty)$,
$\lim_{\alpha \searrow 0} \lambda \alpha = 0$.
Thus, we have
\[
    \lim_{\alpha \searrow 0}
    \mu(\alpha)
    = \lim_{\alpha \searrow 0}
    f^{-1}(\alpha),
\]
where $f(x) = \tfrac{1}{p} \tr[\bA (\bA + x \bI)^{-1}]$.
Observe that function $f$ is strictly decreasing
over $(\mu_0, \infty)$,
and $\lim_{x \nearrow \infty} f(x) = 0$.
Hence, the function $f^{-1}$ is strictly decreasing
and $\lim_{\alpha \searrow 0} f^{-1}(\alpha) = \infty$.
This provides us with the first limit.
To obtain the limit of $\mu(\alpha)$ as $\alpha \nearrow \infty$,
write from \cref{eq:sketched-modified-lambda}
\[
    \mu
    = \lambda 
    + \tfrac{1}{\alpha} \tfrac{1}{p}
    \tr \bracket{ \mu \bA (\bA + \mu \bI)^{-1} }.
\]
Observe that $\tfrac{1}{p} \tr[\mu \bA (\bA + \mu \bI)^{-1}]$
is bounded for $\mu \in (\mu_0, \infty)$.
Thus, taking the limit $\alpha \nearrow \infty$,
we conclude that $\lim_{\alpha \nearrow \infty} \mu(\alpha) = \lambda$.
This finishes the second part, and completes the proof.
\end{proof}

\subsection{Proof of \cref{rem:joint-signs-lambda-mu}}
\label{sec:joint-signs-lambda-mu}

\begin{proof}
We start by writing \cref{eq:sketched-modified-lambda}
in terms of $\alpha$ as
\[
    \lambda = \mu
    \paren{1 - \tfrac{1}{\alpha} \tfrac{1}{p} \tr \bracket{ \bA (\bA + \mu \bI)^{-1} } }.
\]
For the subsequent argument, it will help to rearrange the terms in the equation in display above to arrive at
the following equivalent equation:
\begin{equation}
    \label{eq:sketched-modified-lambda-rewriting}
    1 - \frac{\lambda}{\mu}
    = \tfrac{1}{\alpha} \tfrac{1}{p} \tr \bracket{ \bA (\bA + \mu \bI)^{-1} }.
\end{equation}
We consider two separate cases depending on $\lambda \ge 0$ and $\lambda < 0$.

\textbf{Case $\lambda \ge 0$}:
Fix $\alpha > 0$.
Observe that the left side of \cref{eq:sketched-modified-lambda-rewriting}
is an increasing function of $\mu$,
and the right side of \cref{eq:sketched-modified-lambda-rewriting}
is a decreasing function of $\mu$.
As $\mu$ varies from $0^{+}$ to $\infty$, 
the right hand side decreases from $\tfrac{r(A)}{\alpha}$ to $0$,
while the left hand side increases from $-\infty$ to $1$.
Since $1 > 0$,
there is a unique intersection for $\mu \ge 0$.

\textbf{Case $\lambda < 0$}:
Fix $\alpha \le r(\bA)$.
For this subcase,
from \Cref{rem:mu0-lambda0-signs},
$\mu_0 \ge 0$.
Thus, there is a unique intersection for $\mu \ge 0$.
Fix now $\alpha > r(\bA)$.
For this subcase, the term in the parenthesis of \cref{eq:sketched-modified-lambda}
is positive. Thus, $\mathrm{sign}(\mu) = \mathrm{sign}(\lambda)$.

This completes all the three cases, and finishes the proof.
\end{proof}

\subsection{Proof of \cref{rem:concavity-mu-in-lambda}}
\label{sec:concavity-mu-in-lambda}

\begin{proof}
Recall that $\mu_0 > - \lambdaminnz(\bA)$.
For $x \in (\mu_0, \infty)$,
observe that
\[
    \frac{\partial}{\partial x}
    \tr \bracket{ \bA^2 (\bA + x \bI)^{-1} }
    = - \tr \bracket{ \bA^2 (\bA + x \bI)^{-2} }
    < 0,
\]
\[
    \frac{\partial^2}{\partial x^2}
    \tr \bracket{ \bA^2 (\bA + x \bI)^{-1} }
    = 2 \tr \bracket{ \bA^2 (\bA + x \bI)^{-3} } 
    > 0.
\]
Thus, the function
\[
    x \mapsto
    \tfrac{1}{q}
    \tr \bracket{ \bA^2 (\bA + x \bI)^{-1} }
\]
is strictly decreasing and convex over $(\mu_0, \infty)$,
and consequently
the function
\[
    x \mapsto
    \tfrac{1}{q}
    \tr \bracket{ x \bA (\bA + x \bI)^{-1} }
    = 
    \tfrac{1}{q}
    \tr \bracket{ \bA (\bI - \bA (\bA + x \bI)^{-1}) }
    =
    \tfrac{1}{q}
    \tr[ \bA ]
    - \tfrac{1}{q}
    \tr \bracket{ \bA^2 (\bA + x \bI)^{-1} }
\]
is strictly increasing and concave over $(\mu_0, \infty)$.
Hence, the function $f$ 
(appearing in the right-hand side of \cref{eq:sketched-modified-lambda} in $\mu$)
defined by
\begin{equation}
    \label{eq:sketched-modified-lambda-rhs}
    f(x)
    =
    x 
    - x \tfrac{1}{q}
    \tr \bracket{ \bA (\bA + x \bI)^{-1} }
    =
    x
    \paren{ 1 - \tfrac{1}{q} \tr \bracket{ \bA (\bA + x \bI)^{-1} } }
\end{equation}
is strictly increasing and convex over $(\mu_0, \infty)$.

Now, observe from \cref{eq:sketched-modified-lambda} that
for a given $\lambda$,
$\mu(\lambda) = f^{-1}(\lambda)$,
where $f$ is as defined in \cref{eq:sketched-modified-lambda-rhs}.
Because inverse of a strictly increasing, continuous, and convex function
is strictly increasing, continuous, and concave 
(see, e.g., Proposition 3 of \cite{hiriart_urruty-martinez_legaz_2003}),
we conclude that $\lambda \mapsto \mu(\lambda)$
where $\mu(\lambda)$ solves \cref{eq:sketched-modified-lambda}
is concave in $\lambda$ over $(\lambda_0, \infty)$.
We remark that, more directly,
we can also compute the second derivative of $\mu(\lambda)$ 
with respect to $\lambda$.
From \cref{eq:mu-deriv-lambda}, we have
\begin{equation}
    \label{eq:mu-deriv-lambda-in-alpha}
    \frac{\partial \mu}{\partial \lambda}
    = \frac
    {1}
    {
    1 - \tfrac{1}{\alpha} \tfrac{1}{p} \tr \bracket{ \bA^2 \paren{ \bA + \mu \bI }^{-2} }}.
\end{equation}
Taking partial derivative of \cref{eq:mu-deriv-lambda-in-alpha}
with respect to $\lambda$, we get
\[
    \frac{\partial^2 \mu}{\partial \lambda^2}
    = 
    \frac
    {
        -2 \tfrac{1}{\alpha} \tfrac{1}{p} \tr \bracket{ \bA^2 (\bA + \mu \bI)^{-3} }
    }
    {
        \left( 
            1 - \tfrac{1}{\alpha} \tfrac{1}{p} \tr \bracket{ \bA^2 (\bA + \mu \bI)^{-2} } 
        \right)^2
    }
    \frac{\partial \mu}{\partial \lambda}
    =
    \frac
    {
        -2 \tfrac{1}{\alpha} \tfrac{1}{p} \tr \bracket{ \bA^2 (\bA + \mu \bI)^{-3} }
    }
    {
        \left( 
            1 - \tfrac{1}{\alpha} \tfrac{1}{p} \tr \bracket{ \bA^2 (\bA + \mu \bI)^{-2} } 
        \right)^3
    }
    < 0,
\]
from which the concavity claim follows.

Using the concavity of $\mu$ in $\lambda$,
we can write for 
$\lambda, \widetilde{\lambda} \in (\lambda_0, \infty)$,
\begin{equation}
    \label{eq:concavity-bound-mu-lambda-1}
    \mu(\lambda)
    \le \mu(\widetilde{\lambda})
    + \frac{\partial \mu}{\partial \lambda} 
   \mathrel{\Big |}_{\lambda = \widetilde{\lambda}} 
    (\lambda - \widetilde{\lambda}).
\end{equation}
Now, from \cref{eq:sketched-modified-lambda},
for any $\widetilde{\lambda} \in (\lambda_0, \infty)$, 
we have 
\begin{equation}
    \label{eq:concavity-bound-mu-lambda-2}
    \mu(\widetilde{\lambda}) - \widetilde{\lambda}
    = \tfrac{1}{q}
    \tr \bracket{ \mu(\widetilde{\lambda}) \bA (\bA + \mu(\widetilde{\lambda}) \bI)^{-1} }
    = \tfrac{1}{\alpha} \tfrac{1}{p}
    \tr \bracket{ \mu(\widetilde{\lambda}) \bA (\bA + \mu(\widetilde{\lambda}) \bI)^{-1} }.
\end{equation}
Substituting in \cref{eq:concavity-bound-mu-lambda-2}
in \cref{eq:concavity-bound-mu-lambda-1} yields
\begin{equation}
    \label{eq:concavity-bound-mu-lambda-3}
    \mu(\lambda)
    \le
    \frac{\partial \mu}{\partial \lambda} 
    \mathrel{\Big |}_{\lambda = \widetilde{\lambda}} 
    \lambda 
    +
    \tfrac{1}{\alpha}
    \tfrac{1}{p}
    \tr
    \bracket{ \mu(\widetilde{\lambda}) \bA (\bA + \mu(\widetilde{\lambda}) \bI)^{-1} }.
\end{equation}
From \Cref{prop:monotonicities-lambda-alpha},
$\lambda \mapsto \mu(\lambda)$ is monotonically increasing in $\lambda$
and $\lim_{\lambda \nearrow \infty} \mu(\lambda) = \infty$.
In addition, $\mu \mapsto \tr[\bA^2 (\bA + \mu \bI)^{-2}]$
is monotonically decreasing in $\mu$
and $\lim_{\mu \nearrow \infty} \tr[\bA^2 (\bA + \mu \bI)^{-2}] = 0$,
while 
$\mu \mapsto \tr[\mu \bA (\bA + \mu \bI)^{-1}]$
is monotonically increasing in $\mu$,
and 
$
\lim_{\mu \nearrow \infty} \tr[\mu \bA (\bA + \mu \bI)^{-1}] = \tr[ \bA ].
$
Thus, 
from \cref{eq:mu-deriv-lambda-in-alpha},
choosing $\widetilde{\lambda}$ large enough
so that $\mu(\widetilde{\lambda})$ is large enough,
for any $\epsilon > 0$, we can write
\begin{equation}
    \label{eq:deriv-asymp-mu-lambda}
    \frac{\partial \mu}{\partial \lambda} 
    \mathrel{\Big |}_{\lambda = \widetilde{\lambda}}
    = 
    \frac
    {1}
    {
    1 - \tfrac{1}{\alpha} 
    \tfrac{1}{p} \tr \bracket{ \bA^2 (\bA + \mu(\widetilde{\lambda}) \bI)^{-2} } }
    \le
    1 
    +
    \epsilon,
\end{equation}
\begin{equation}
    \label{eq:concavity-intercept-asymp}
    \tfrac{1}{\alpha}
    \tfrac{1}{p}
    \tr \bracket{ \mu(\widetilde{\lambda}) \bA (\bA + \mu(\widetilde{\lambda}) \bI)^{-1} }
    \le
    \tfrac{1}{\alpha} \tfrac{1}{p} \tr \bracket{ \bA } + \epsilon.
\end{equation}
Combining 
\cref{eq:concavity-bound-mu-lambda-2,eq:deriv-asymp-mu-lambda,eq:concavity-intercept-asymp},
one then has
\[
    \mu(\lambda)
    \le
    (1 + \epsilon)
    \lambda 
    + 
    \tfrac{1}{\alpha}
    \tfrac{1}{p}
    \tr[\bA]
    + \epsilon.
\]
Since the inequality holds for any arbitrary $\epsilon$, 
the desired upper bound on $\mu(\lambda)$ follows.
For the lower bound,
observe from \cref{eq:concavity-bound-mu-lambda-2} that
for any $\lambda \in (\lambda_0, \infty)$
\[
    \mu(\lambda)
    = \lambda + \tfrac{1}{q} \tr \bracket{ \mu(\lambda) \bA 
    (\bA + \mu(\lambda) \bI)^{-1} }.
\]
From \Cref{rem:joint-signs-lambda-mu},
$\mu(\lambda) \ge 0$ either when $\lambda \ge 0$,
or when $\alpha \le r(\bA)$.
In either of the cases,
the term $\tfrac{1}{q} \tr[\mu(\lambda) \bA (\bA + \mu(\lambda) \bI)^{-1}]$
is positive,
and thus $\mu(\lambda) \ge \lambda$.
Finally, the limit as $\lambda \nearrow \infty$
follows simply by noting that
$\mu(\lambda) \nearrow \infty$ and 
$\tr[\mu \bA (\bA + \mu \bI)^{-1}] \nearrow \tr[\bA]$
as $\lambda \nearrow \infty$.
This finishes the proof.
\end{proof}

\subsection{Proof of \cref{rem:alt-mu-prime}}

\begin{proof}
We begin by rewriting \cref{eq:mu-prime} using \cref{eq:thm:sketched-pseudoinverse}:
\begin{align}
    \mu' = \frac{\frac{\mu^3}{q} \tr \bracket{\mPsi \paren{\mA + \mu \mI}^{-2} }}{\mu \paren{ 1 - \tfrac{1}{q} \tr \bracket{\mA \inv{\mA + \mu \mI_p}}} + \frac{\mu^2}{q} \tr \bracket{\mA \paren{\mA + \mu \mI}^{-2}} }.
\end{align}
After dividing both the numerator and denominator by $\mu$, we note that the denominator has a form which has already been simplified in \Cref{sec:monotonicies-lambda-alpha}, and immediately obtain the factorization in terms of $\tfrac{\partial \mu}{\partial \lambda}$.
\end{proof}

\section{Proofs in \Cref{sec:applications}}
\label{sec:proofs:applications}
\begin{editedtwo}
    This section collects proofs for various results in \Cref{sec:applications}.
\end{editedtwo}

\begin{edited}
\subsection{Proof of \Cref{eq:sketch-and-project}}
\label{sec:proof:eq:sketch-and-project}
\begin{proof}
    First, we can write $\vb = \mL \vx_*$. Next, we can subtract $\vx_*$:
    \begin{align}
        \vx_t - \vx_* = (\mI_n - \mL^\ctransp \mS_t (\mS_t^\ctransp \mL \mL^\ctransp \mS_t)^\dagger \mS_t^\ctransp \mL) (\vx_{t - 1} - \vx_*).
    \end{align}
    Because $(\mI_n - \mL^\ctransp \mS_t (\mS_t^\ctransp \mL \mL^\ctransp \mS_t)^\dagger \mS_t^\ctransp \mL)$ is a projection matrix and therefore idempotent, 
    \begin{align}
        \norm[2]{\vx_t - \vx_*}^2 &= (\vx_{t - 1} - \vx_*)^\ctransp (\mI_n - \mL^\ctransp \mS_t (\mS_t^\ctransp \mL \mL^\ctransp \mS_t)^\dagger \mS_t^\ctransp \mL)   (\vx_{t - 1} - \vx_*) \\
        &\asympequi 
        (\vx_{t - 1} - \vx_*)^\ctransp (\mI_n - \mL^\ctransp \biginv{\mL \mL^\ctransp  + \mu \mI_p} \mL) (\vx_{t - 1} - \vx_*) \\
        &\leq \rho \norm[2]{\vx_{t-1} - \vx_*}^2,
    \end{align}
    where the asymptotic equivalence is the result of applying \cref{thm:sketched-pseudoinverse} with $\mA = \mL \mL^\ctransp$, and $\rho = \lambda_{\max}(\mI_n - \mL^\ctransp \biginv{\mL \mL^\ctransp  + \mu \mI_p} \mL)$. Thus the stated convergence bound holds almost surely for any $t$.
\end{proof}
\subsection{Proof of \cref{rem:sketch-and-project}}
\begin{proof}
Since $\lambda = 0$, we know that $\alpha \mapsto \mu$ is an invertible mapping from $(0, r(\mL))$ onto $\mu \in (0, \infty)$ by \cref{prop:monotonicities-lambda-alpha} and \cref{rem:joint-signs-lambda-mu}, while for $\alpha \geq r(\mL )$, $\mu = 0$ and therefore a solution is reached in $t_\varepsilon = 1$ steps. Thus, it remains only to consider $\mu \in (0, \infty)$. Generalizing to galactic inversion algorithms of complexity $O(m^{1+\delta}p)$, we can write the relative computation factor in terms of $\mu$ as
\begin{align}
    \alpha^{1 + \delta} t_\varepsilon = 
    \paren{\tfrac{1}{n} \tr \bracket{\mL \mL^\ctransp \biginv{\mL \mL^\ctransp + \mu \mI_n}}}^{1 + \delta} 
    \Bigg\lceil \frac{\log\paren{\frac{a_+ \norm[2]{\vx_0 - \vx_*}^2}{\varepsilon}}}
    {\log (1 + \tfrac{a_-}{\mu})} \Bigg\rceil,
\end{align}
where $a_+ \defeq \lambda_{\max}(\mL \mL^\ctransp)$ and $a_- \defeq \lambdaminnz(\mL \mL^\ctransp)$. For any fixed $\mu$ (and equivalently any fixed $\alpha < r(\mL)$), as $\varepsilon \searrow 0$, we clearly have $t_\varepsilon \nearrow \infty$. For fixed $\varepsilon$, the limiting behavior as $\mu \nearrow \infty$ (equivalently $\alpha \searrow 0$) is determined by the ratio
\begin{align}
    \frac{\paren{\tfrac{1}{n} \tr \bracket{\mL \mL^\ctransp \biginv{\mL \mL^\ctransp + \mu \mI_n}}}^{1 + \delta}}
    {\log (1 + \tfrac{a_-}{\mu})}
    &= 
    \frac{\paren{\tfrac{1}{n} \tr \bracket{\mL \mL^\ctransp \biginv{\mL \mL^\ctransp + \mu \mI_n}}}^{1 + \delta}}
    {\tfrac{a_-}{\mu} + o(\tfrac{1}{\mu})}
    \searrow 0.
\end{align}
\end{proof}
\end{edited}

\section{Proofs for free sketching}
\label{sec:proofs-free}

\begin{editedtwo}

We first establish some notation and useful lemmas. We next provide the proof details for \Cref{thm:general-free-sketching} and then provide minor derivation details for orthogonal sketching in \Cref{cor:orthonormal-sketch}.

With some abuse of notation, we will let $\mA$ denote both the finite $p \times p$ matrix as well as the limiting element in the free probability space (which can be understood for example as being a bounded linear operator on a Hilbert space).  We note that all notions that we need, in particular logarithms of determinants, are well defined in this limit as well, as long as they are appropriately normalized. For this reason, we define normalized versions $\barlogdet(\mA) \defeq \tfrac{1}{p} \logdet(\mA)$ and $\bartr[\mA] \defeq \tfrac{1}{p} \tr[\mA]$ which extend nicely to the limit.

We will also use the following straightforward result from differential calculus allowing us to draw conclusions about first derivatives from second derivatives.
\begin{lemma}[Controlling derivatives]
\label{lem:control-derivatives}
Let $g \colon \setT \times \setZ \subseteq \complexset^2 \to \complexset$ be holomorphic. Then for each $t \in \setT$, 
if $\inf_{z \in \setZ} \Big| \frac{\partial g(t, z)}{\partial t} \Big| = 0$
and 
$\frac{\partial^2 g(t, z)}{\partial t \partial z} = 0$ for all $z \in \setZ$, then
$\frac{\partial g(t, z)}{\partial t} = 0$
for all $z \in \setZ$.
\end{lemma}
\begin{proof}
By the fundamental theorem of calculus, for some $z_0 \in \setZ$,
\begin{align}
    \frac{\partial g(t, z)}{\partial t} 
    &= \frac{\partial}{\partial t}
    \paren{\int_{z_0}^z \frac{\partial g(t, u)}{\partial u} du + g(t, z_0)} \\
    &= \int_{z_0}^z \frac{\partial^2 g(t, u)}{\partial t \partial u} du + \frac{\partial g(t, z_0)}{\partial t}\\
    &= 0,
\end{align}
where the final equality follows by our hypotheses since $z_0$ is arbitrary.
\end{proof}

We lastly introduce a series of invertible transformations from free probability \cite{mingo2017free}:
\begin{align}
    G_\mA(z) = \bartr \bigbracket{\biginv{z \mI - \mA}}
    \;
    \longleftrightarrow
    \;
    M_\mA(z) = \frac{1}{z} G_\mA\paren{\frac{1}{z}} - 1
    \;
    \longleftrightarrow
    \;
    \Sxf_\mA(z) = \frac{1 + z}{z} M_\mA^{\langle -1\rangle}(z),
    \nonumber
\end{align}
which are the Cauchy transform (negative of the Stieltjes transform),
moment generating series $M_\mA(z) = \sum_{k=1}^\infty \bartr[\mA^k] z^k$,
and $S$-transform of $\mA$, respectively.
Here $M_\mA^{\langle -1\rangle}$ denotes inverse under composition of $M_\mA$.
We also recall the property of free products that $\Sxf_{\mA \mB}(z) = \Sxf_\mA(z) \Sxf_\mB(z)$, or equivalently $M_{\mA \mB}^{\langle -1\rangle}(z) = \tfrac{1 + z}{z} M_\mA^{\langle -1\rangle}(z) M_\mB^{\langle -1\rangle}(z) = \Sxf_\mA(z) M_\mB^{\langle -1\rangle}(z)$.

\end{editedtwo}

\subsection{Proof of \Cref{thm:general-free-sketching}}

\begin{proof}
\begin{editedthree}
We begin with the simpler case where $\mTheta$ in the equivalence definition is such that $p \mTheta$ has uniformly bounded operator norm. For this proof, we will simply write $\mTheta$ instead of $p\mTheta$ to be compatible with the normalized trace.
First, we can decompose $\mTheta$ into real and imaginary parts $\mTheta = \mTheta_{\Re} + i \mTheta_{\Im}$, so without loss of generality we can assume $\mTheta$ is real. Similarly,%
\end{editedthree}
\begin{editedtwo}
we note that $\bartr[\mTheta \mB] = \bartr[\tfrac{1}{2} (\mTheta + \mTheta^\ctransp) \mB]$ for any self-adjoint matrix $\mB \in \complexset^{p \times p}$, so we can assume $\mTheta$ is symmetric and therefore diagonalizable without loss of generality.
We let $\tmS = (\mS \mS^\ctransp)^{1/2}$ and note that we can now work entirely in dimension $p$ instead of both dimensions $p$ and $q$:
\begin{align}
    \mS \biginv{\mS^\ctransp \mA \mS - z \mI_q} \mS^\ctransp = \tmS \biginv{\tmS \mA \tmS - z \mI_p} \tmS.
\end{align}
Consider now the limit where $(\mTheta, \mA, \tmS)$ have converged spectrally with $\tmS$ free from $\mTheta$ and $\mA$.
We need only show that for some $\zeta \in \complexset^+$,
\begin{align}
    \bartr[\mTheta \tmS \biginv{\tmS \mA \tmS - z \mI} \tmS]
    = \bartr[\mTheta \biginv{\mA - \zeta \mI}].
\end{align}
We now define parameterized operators 
$\mB_{t, \zeta} = \mA + t \mTheta - \zeta \mI$
and
$\mB^{\tmS}_{t,z} = \tmS(\mA + t \mTheta) \tmS - z \mI$. By Jacobi's formula, we have the following two equalities
\begin{align}
    \bartr[\mTheta \biginv{\mA - \zeta \mI}] 
    &=
    \frac{\partial \barlogdet(\mB_{t,\zeta})}{\partial t} \Big|_{t = 0}, \\
    \bartr[\mTheta \tmS \biginv{\tmS \mA \tmS - z \mI} \tmS] &=
    \frac{\partial \barlogdet(\mB^{\tmS}_{t,z})}{\partial t} \Big|_{t = 0}.
\end{align}
Suppose that $z \mapsto \zeta$ is a holomorphic map. Then another way of stating our condition to be proven is that for $t = 0$ and all $z \in \complexset^+$, we must have $\frac{\partial g(t, z)}{\partial t} = 0$, where
\begin{align}
    g(t, z) = 
    \barlogdet(\mB_{t, \zeta})
    - \barlogdet(\mB^{\tmS}_{t,z})
    .
\end{align}
By \Cref{lem:control-derivatives}, it is sufficient to show that $\Im(\zeta) \nearrow \infty$ as $z \to i \infty$ (implying the condition $\inf_{z \in \setZ} \Big| \frac{\partial g(t, z)}{\partial t} \Big| = 0$) and that $\frac{\partial^2 g(t, z)}{\partial t \partial z} = 0$ for all $z \in \complexset^+$.

We therefore seek a choice of $z \mapsto \zeta$ that satisfies these conditions.
In particular, we need only to show that the last condition holds, and the rest will follow. 
The main idea is that we can control the derivative of $g$ in $t$, which has a dependence on $\mTheta$, in terms of the derivative of $g$ in $z$, which does not. 
For succinctness in the subsequent arguments, we will use the following notation for derivatives: for a function $f_t \colon \complexset \to \complexset$, we denote $\dot{f}_t(z) = \tfrac{\partial f_t(z)}{\partial t}$ and $f_t'(z) = \tfrac{\partial f_t(z)}{\partial z}$. 
That is, $\dot{f_t}$ is the derivative with respect to its index $t$, and $f_t'$ is the derivative with respect to its argument (typically $z$). Although we omit the argument $z$ of $\zeta$, we let $\zeta' = \frac{\partial \zeta}{\partial z}$.

Define $\mA_t \defeq \mA + t \mTheta$.
Appealing again to Jacobi's formula, we have two further equalities:
\begin{align}
    \frac{\partial \barlogdet(\mB_{t,\zeta})}{\partial z}
    &=
    -\bartr \bigbracket{ \mB_{t,\zeta}^{-1}}\zeta'
    = G_{\mA_t}(\zeta) \zeta', \\
    \frac{\partial \barlogdet(\mB^{\tmS}_{t,z})}{\partial z} 
    &= -\bartr \bigbracket{\mB^{\tmS\;-1}_{t,\zeta}}
    = G_{\mA_t \tmS^2}(z)
    .
\end{align}
The last equality follows because $\tmS \mA_t \tmS$ has the same spectrum as $\mA_t \tmS^2$ (to see this, note that they have the same moments due to the cyclic invariance of the tracial state $\bartr$). We therefore need $\zeta$ such that at $t = 0$, for all $z \in \complexset^+$,
\begin{align}
    \dot{G}_{\mA_t \tmS^2}(z) = \dot{G}_{\mA_t}(\zeta) \zeta'.
\end{align}
Equivalently, in terms of the moment generating series, we need
\begin{align}
    \frac{\dot{M}_{\mA_t \tmS^2}(\frac{1}{z})}{z} = \frac{\dot{M}_{\mA_t}(\frac{1}{\zeta}) \zeta'}{\zeta}.
    \label{eq:m-dot-z-zeta}
\end{align}
This is finally the condition that we will show. 

Now, from the property of free products, we know that for $m \in \complexset^-$,
\begin{align}
    M_{\mA_t \tmS^2}^{\langle -1 \rangle}(m) = \Sxf_{\tmS^2}(m) M_{\mA_t }^{\langle -1 \rangle}(m).
\end{align}
Choose now $m = M_{\mA_t \tmS^2}(\frac{1}{z})$, which gives us
\begin{align}
    z \Sxf_{\tmS^2}(m) = \frac{1}{M_{\mA_t }^{\langle -1 \rangle}(m)}
    \iff
    m = M_{\mA_t} \paren{\frac{1}{z \Sxf_{\tmS^2}(m)}}.
    \label{eq:m-M-S}
\end{align}
Matching the forms of \cref{eq:m-dot-z-zeta,eq:m-M-S}, we can form a guess of $\zeta = z \Sxf_{\tmS^2}(m)$, which we can also prove is the correct choice. To do so, we note that $m$ is parameterized by both $t$ and $z$. We first implicitly differentiate with respect to $t$:
\begin{align}
    \dot{m} = \dot{M}_{\mA_t} \paren{\frac{1}{z \Sxf_{\tmS^2}(m)}}
    - M_{\mA_t}' \paren{\frac{1}{z \Sxf_{\tmS^2}(m)}} \frac{\Sxf_{\tmS^2}'(m) \dot{m}}{z \Sxf_{\tmS^2}(m)^2},
\end{align}
which after plugging in $\zeta = z \Sxf_{\tmS^2}(m)$ gives us 
\begin{align}
    \dot{m} = \frac{\dot{M}_{\mA_t} \paren{\frac{1}{\zeta}}}{1 + M_{\mA_t}' \paren{\frac{1}{\zeta}} \frac{z \Sxf_{\tmS^2}'(m)}{\zeta^2}}.
\end{align}
Next, noting that $\zeta' = \Sxf_{\tmS^2}(m) + z \Sxf_{\tmS^2}'(m) m'$, we differentiate \cref{eq:m-M-S} with respect to $z$:
\begin{align}
    m' = -M_{\mA_t}' \paren{\frac{1}{\zeta}} \frac{\zeta'}{\zeta^2}
    \implies 
    \zeta' = \frac{\Sxf_{\tmS^2}(m)}{1 + M_{\mA_t}' \paren{\frac{1}{\zeta}} \frac{z \Sxf_{\tmS^2}'(m)}{\zeta^2}}.
\end{align}
We can deduce from the previous two equations and the fact that $\Sxf_{\tmS^2}(m) = \frac{\zeta}{z}$ that
\begin{align}
    \dot{m} = \dot{M}_{\mA_t} \paren{\frac{1}{\zeta}} \frac{ z \zeta'}{\zeta},
\end{align}
which is equivalent to \cref{eq:m-dot-z-zeta}, which we needed to show. Therefore, specializing to $t = 0$, we have that $\zeta = z \Sxf_{\tmS^2}(M_{\mA \tmS^2}(\tfrac{1}{z}))$ makes the the second derivative condition of \Cref{lem:control-derivatives} satisfied. 
Additionally, we have that $\Im(\zeta) \nearrow \infty$ as $z \to i \infty$: note that $M_{\mA \tmS^2}(\tfrac{1}{z}) = \bartr(\mA \tmS^2) \tfrac{1}{z} + o(\tfrac{1}{z})$ 
and similarly $\Sxf_{\tmS^2}(m) = \frac{1}{\bartr(\tmS^2)} + o(m)$, such that $\zeta = z (\frac{1}{\bartr(\tmS^2)} + o(\tfrac{1}{z}))$.

To obtain the equation for $\zeta$ in terms of $\Sxf_{\tmS^2}$ and $M_\mA$, combine $\zeta = z \Sxf_{\tmS^2}(m)$ and \cref{eq:m-M-S}. To obtain the equation for $\zeta$ in terms of $\mS^\ctransp \mA \mS$ and $z$, use the fact that $m = M_{\tmS \mA \tmS}(\frac{1}{z})$.
\end{editedtwo}

\begin{editedthree}
\paragraph{Trace norm bounded $\mTheta$}    
For more general trace norm bounded $\mTheta$, such as rank one vector outer products, $p \mTheta$ does not have bounded operator norm and so the previous argument cannot be applied. However, with a stronger notion of freeness, called first-order or infinitesimal freeness~\cite{shlyakhtenko2018infinitesimal}, this extension is also possible. Following~\cite{shlyakhtenko2018infinitesimal,cebron2022freeness}, the key condition is to require sufficiently fast convergence of $G_{\mA \tmS^2}(z)$ in $p$. 
Concretely, let $\widetilde{G}_{\mA \tmS^2}$ be the Cauchy transform of the multiplicative free convolution of the spectra of $\mA$ and $\tmS^2$---that is, what the Cauchy transform of $\mA \tmS^2$ would be if $\mA$ and $\tmS^2$ were free, which is not possible in finite dimensions. 
Then we need almost sure convergence in the sense that
\begin{align}
    G_{\mA \tmS^2}(z) = \widetilde{G}_{\mA \tmS^2}(z) + \epsilon(p)
    \quad \text{where} \quad
    \epsilon(p) = o(\tfrac{1}{p}) .
\end{align}
Fortunately, this rate is known to hold in the almost sure sense when $\tmS$ is a unitarily invariant ensemble~\cite[Theorem 3.5]{cebron2022freeness}, so this assumption is satisfiable. 

We apply the same approach as in the previous case when $p \mTheta$ had bounded operator norm. 
Even though $p \mTheta$ now does not converge to a limiting bounded operator, the first-order terms like $\dot{G}_{\mA}(z)$ remain well-defined due to the bounded trace norm.
We note that a trace norm bounded perturbation does not change the limiting spectral distribution, which means that $\mA_t$ and $\mA$ asymptotically have the same spectrum and thus the same result of multiplicative convolution with $\tmS^2$. However, given some $t(p)$, we have the Taylor expansion
\begin{align}
    G_{\mA_{t(p)} \tmS^2}(z) = \widetilde{G}_{\mA \tmS^2}(z) + t(p) \dot{\widetilde{G}}_{\mA \tmS^2}(z) + O(t(p)^2 + \epsilon(p)).
\end{align}
Meanwhile, also taking the Taylor expansion of $G_{(\mA_t)_p}(\zeta)$,
\begin{align}
    G_{\mA_{t(p)}}(\zeta) \zeta' = G_{\mA}(\zeta) \zeta' + t(p) \dot{G}_{\mA}(\zeta) \zeta' + O(t(p)^2).
\end{align}
Therefore, choosing $t(p) = \sqrt{\tfrac{1}{p} \epsilon(p)}$ and taking the derivative of these two expansions, we can finally say that
\begin{align}
    \tr \bigbracket{\mTheta \bigparen{ &\biginv{\mA - \zeta \mI_p} \zeta' - \tmS \biginv{\tmS \mA \tmS - z \mI_p} \tmS }} \\
    &= \dot{G}_{\mA}(\zeta) \zeta' -  \dot{\widetilde{G}}_{\mA \tmS^2}(z) + O\bigparen{t(p) + \tfrac{\epsilon(p)}{t(p)}} \\
    &= O \bigparen{t(p) + \tfrac{\epsilon(p)}{t(p)}} \\
    &\asto 0,
\end{align}
where the final equality follows by choosing $\zeta$ as in the bounded operator norm case.
Then by similar application of \Cref{lem:control-derivatives} as before, we obtain the desired equivalence.

\end{editedthree}

\end{proof}

\subsection{Proof details for orthogonal sketching}
\begin{editedtwo}
To obtain the $S$-transform for the normalized orthogonal sketch, we first note that $\mQ \mQ^\ctransp$ has $q$ eigenvalues of $\frac{1}{\alpha}$ and $p - q$ eigenvalues of $0$. Therefore, it has
\begin{align}
    M_{\mQ \mQ^\ctransp}(z) = \bartr \bigbracket{\mQ \mQ^\ctransp \biginv{\tfrac{1}{z} \mI - \mQ \mQ^\ctransp}}
    = \frac{\alpha z}{\alpha - z},
\end{align}
which has inverse $M_{\mQ \mQ^\ctransp}^{\langle -1 \rangle}(w) = \frac{\alpha w}{\alpha + w}$ and therefore $S_{\mQ \mQ^\ctransp}(w) = \frac{\alpha (1 + w)}{\alpha + w}$.

To obtain the fixed point equation, we first solve $\gamma = \lambda S_{\mQ \mQ^\ctransp}(w)$ for $w$:
\begin{align}
    w = \frac{\alpha (\lambda - \gamma)}{\gamma - \alpha \lambda}.
\end{align}
Then, we plug in $w = -\bartr[\mA \biginv{\mA + \gamma \mI_p}] = \gamma \bartr[\biginv{\mA + \gamma \mI_p}] - 1$:
\begin{align}
    \gamma \bartr[\biginv{\mA + \gamma \mI_p}] =
    \frac{\alpha (\lambda - \gamma)}{\gamma - \alpha \lambda} + 1
    =\frac{\gamma (1 - \alpha)}{\gamma - \alpha \lambda}.
\end{align}
The stated relation follows directly.
\end{editedtwo}

\end{document}